\documentclass[a4paper,11pt,reqno,twoside]{amsart}
\usepackage[latin1]{inputenc}
\usepackage{style_n}
\begin{document}%
\title[Fourier coefficients in arithmetic progressions]
      {Fourier coefficients of $GL(N)$ automorphic forms in arithmetic progressions}
\author[Emmanuel Kowalski]{Emmanuel Kowalski}
\address{ETH Z\"urich -- DMATH \\
R\"amistrasse 101 \\
8092 Z\"urich \\
Switzerland}
\email{emmanuel.kowalski@math.ethz.ch}
\author[Guillaume Ricotta]{Guillaume Ricotta}
\address{Université Bordeaux 1 \\
Institut de Mathématiques de Bordeaux \\
351, cours de la Liberation \\
33405 Talence cedex \\
France}
\email{Guillaume.Ricotta@math.u-bordeaux1.fr}


\date{Version of \today} 



\begin{abstract}
We show that the multiple divisor functions of integers in invertible residue classes modulo a prime number, as well as the Fourier coefficients of $GL(N)$ Maass cusp forms for all $N\geq 2$, satisfy a central limit theorem in a suitable range, generalizing the case $N=2$ treated by \'{E}. Fouvry, S. Ganguly, E. Kowalski and P. Michel in \cite{FoGaKoMi}. Such universal Gaussian behaviour relies on a deep equidistribution result of products of hyper-Kloosterman sums.
\end{abstract}

\maketitle
\tableofcontents

\section{Introduction and statement of the results}%

Problems concerning the asymptotic distribution of arithmetic
functions in residue classes are very classical in analytic number
theory, and have been considered from many different points of
view. Recently, \'{E}. Fouvry, S. Ganguly, E. Kowalski and P. Michel~\cite{FoGaKoMi}
proved that the classical divisor function, as well as Fourier
coefficients of classical (primitive) holomorphic cusp forms,
satisfies a form of central limit theorem concerning the distribution
in non-zero residue classes modulo a large prime number.
\par
It seems natural to explore the same type of statistical questions for
higher divisor functions, or Fourier coefficients of automorphic forms
on higher-rank groups, especially because of the philosophy which
relates the distribution properties of primes in arithmetic
progressions with that of higher divisor functions. We will show that
a suitable central limit theorem holds for these divisor functions as
well as for Fourier coefficients of cusp forms on $GL(N)$ for all $N\geq 2$, taken to be of
full level over $\Q$. To simplify the notation, we will not consider
holomorphic cusp forms in the case $N=2$, since this case is treated
in~\cite{FoGaKoMi}. 
\par
We remark that there are not many statements of analytic number theory
which are currently known to hold for an individual (and not on average over a family) cusp form on
$GL(N)$ for arbitrary $N$. The best known results of this type are the
approximations to the Ramanujan--Petersson--Selberg conjectures (see
the paper of Z. Rudnick, W. Luo and P. Sarnak~\cite{MR1334872}), and
the distribution properties of zeros of the standard $L$-functions for
test functions with suitable restrictions (due to Z. Rudnick and
P. Sarnak~\cite{RuSa}). The present paper adds a further example of
such properties. It is interesting to note that we require a deep
result of equidistribution of hyper-Kloosterman sums to obtain a
``universal'' Gaussian behavior, which is derived from the
determination of the monodromy groups of Kloosterman sheaves, due to
N. Katz~\cite{MR955052}. As far as we are aware, this is a new
ingredient in such studies.


\subsection{Statement of the results}%

We refer to the introduction of~\cite{FoGaKoMi} for a survey of the
literature prior to that paper, and we now state our results.  We fix
throughout an integer $N\geq 2$. 
\par
Let $p$ be an odd prime number. We will consider the group $\mathbb{F}_p^\times$
of invertible residue classes modulo $p$ as a probability space with
its uniform probability measure $\mu_p$, so that
\begin{equation*}
\forall E\subset\mathbb{F}_p^\times,\quad\mu_p(E)=\frac{\abs{E}}{p-1}.
\end{equation*}

\subsubsection{The case of $GL(N)$ Maass cusp forms}%

We fix a Hecke-Maass cusp form $f$ on $GL(N)$ with level $1$. We
denote by $a_f(m_1,\ldots,m_{N-1})$ its Fourier coefficients, for
integers $m_1$, $\ldots$, $m_{N-2}\geq 1$ and
$m_{N-1}\in\Z-\{0\}$. We also use the shorthand notation
\begin{equation}\label{eq-af}
\af(n)=a_f(n,1,\ldots, 1)
\end{equation}
for $n\geq 1$, and we recall that we then have also
\begin{equation}\label{eq-aff}
  \afs(m)=a_{f^{\ast}}(m,1,\ldots, 1)=a_f(1,\ldots, 1,m)
\end{equation}
for $m\not=0$ an integer, where $f^\ast$ is the dual of $f$. We also
assume that $f$ is arithmetically normalized so that $\af(1)=1$. In
particular, $\af(n)$ is then the eigenvalue of $f$ for the $n$-th
Hecke operator $T_n$.
\par
We will fix a test function $w:\R_+^\ast\to\R$, which is a non-zero
smooth function compactly supported on an interval
$[x_0,x_1]\subset\R_+^\ast$.  For $X\geq 1$, we then define
\begin{eqnarray*}
  S_f(X,p,a) & \coloneqq & \sum_{\substack{n\geq 1 \\
      n\equiv a\bmod{p}}}\af(n)w\left(\frac{n}{X}\right) \\
  M_f(X,p) & \coloneqq & 
  \frac{1}{p}\sum_{n\geq 1}\af(n)w\left(\frac{n}{X}\right).
\end{eqnarray*}
\par
The quantity $M_f(X,p)$ is a natural ``fake'' main term for the
quantity $S_f(X,p,a)$, which naturally occurs in the process but is
extremely small in view of the use of the smooth weight. Having in
mind that the number of terms in $S_f(X,p,a)$ is roughly $X/p$, the
\emph{square root cancellation philosophy} suggests to define
\begin{equation}\label{eq-def-ef}
E_f(X,p,a)=\frac{S_f(X,p,a)-M_f(X,p)}{(X/p)^{1/2}}
\end{equation}
for $a$ an invertible residue class modulo $p$.
\par
An important observation is that, in general, $E_f(X,p,a)$ is not
real-valued, and thus the distribution results will involve
probability measures on $\C$. More precisely, recall that $f$ is said
to be \emph{self-dual} if $f$ is equal to its dual form $f^*$, which
is equivalent with requiring that the Fourier coefficients are real
numbers, in which case $E_f(X,p,a)$ is a real number. If $f$ is not self-dual then we define
\begin{equation*}
Z_f(X,p,a)=\left(\Re{\left(E_f(X,p,a)\right)},\Im{\left(E_f(X,p,a)\right)}\right)\in\R^2.
\end{equation*}
\par
We view these quantities as \emph{random variables} $a\mapsto E_f(X,p,a)$ and \emph{random vectors} $a\mapsto Z_f(X,p,a)$ on the finite set of invertible residue classes modulo $p$ endowed with the uniform probability measure $\mu_p$ described above, and we will attempt to determine their distribution when $p$ is large, for suitable values of
$X$.
\par
We will use the method of moments to study their distribution. This allows us to prove results of interest even in
situations where we cannot currently prove an equidistribution
statement.
\par
For any pair $(\kappa,\lambda)$ non-negative integers, we define the
$(\kappa,\lambda)$-th \emph{mixed moment} of $E_f(X,p,a)$ by
\begin{equation}\label{eq-mixed-moment}
\mathsf{M}_f(X,p,(\kappa,\lambda))\coloneqq
\frac{1}{p}\sum_{\substack{a\bmod{p} \\
(a,p)=1}}E_f(X,p,a)^{\kappa}\overline{E_f(X,p,a)}^{\lambda}.
\end{equation}
The next theorem states an asymptotic expansion for these mixed
moments in specific ranges for $X$ with respect to $p$.  Before
stating it, we recall that the $k$-th moment of a centered Gaussian
random variable with variance $V=\sigma^2\geq 0$ is given by
\begin{equation*}
m_kV^{k/2}=\delta_{2\mid k}\frac{k!}{2^{k/2}(k/2)!}V^{k/2}.
\end{equation*}
\par

\begin{theoint}[Mixed moments]\label{theo_moments}
  Let $f$ be an even or odd $GL(N)$ Hecke-Maass cusp form, which is
  not induced from a holomorphic form if $N=2$, and $w:\R_+^\ast\to\R$
  be a smooth and compactly supported function. Let
  $\;2p^{N-1}<X\leq p^N$. Then we have
\begin{multline}\label{eq_theo_moments}
\mathsf{M}_f(X,p,(\kappa,\lambda))=\delta_{f=f^\ast}m_{\kappa+\lambda}\left(2c_{f,w}\right)^{(\kappa+\lambda)/2}+\delta_{f\neq f^\ast}\delta_{\kappa=\lambda}2^\kappa\kappa!c_{f,w}^{\kappa} \\
+O_{\epsilon, f}\left(p^{1+\epsilon}\left(\frac{p^{N-1}}{X}\right)^{(\kappa+\lambda)/2-1}+\left(\frac{X}{p^N}\right)^{1/2-\theta+\epsilon}+\frac{1}{\sqrt{p}}\left(\frac{p^N}{X}\right)^{\frac{\kappa+\lambda}{2}+\epsilon}\right)
\end{multline}
for all $\epsilon>0$, where $\theta=1/2-1/(N^2+1)$ and $c_{f,w}>0$ is a constant given by
\begin{equation*}
  c_{f,w}=\frac{r_fH_{f,f^\ast}(1)}{2}\abs{\abs{w}}_2^2,
\end{equation*}
where $r_f$ is the residue at $s=1$ of the Rankin-Selberg $L$-function
$L(f\times f^{\ast},s)$ (see Proposition \ref{propo_residue}), $\|w\|_2$ is the $L^2$-norm of $w$ with respect to the Lebesgue measure on $\R_+^\ast$, and
$H_{f,f^{\ast}}(1)$ is an Euler product defined in Proposition
\ref{propo_variance}. 
\end{theoint}

In this theorem, the error term in \eqref{eq_theo_moments} only tends
to $0$ as $X$ and $p$ tend to infinity in suitable ranges. For
instance, if $X=p^\gamma$ with $N-1<\gamma<N$, then this theorem
implies that
\begin{equation}\label{eq-limit-moments}
  \lim_{p\to+\infty}\mathsf{M}_f(p^\gamma,p,(\kappa,\lambda))=
  \delta_{f=f^\ast}m_{\kappa+\lambda}\left(2c_{f,w}\right)^{(\kappa+\lambda)/2}+
  \delta_{f\neq f^\ast}\delta_{\kappa=\lambda}2^\kappa\kappa!c_{f,w}^{\kappa},
\end{equation}
for all $\kappa$ and $\lambda$ such that
$\kappa+\lambda<\frac{1}{N-\gamma}$.
\par
The most restrictive error term in \eqref{eq_theo_moments} is the last
one, which we will see comes from deep equidistribution theorems of
hyper-Kloosterman sums. One can expect that the estimate for this term
is not best possible, and that the asymptotic formula for all moments
should be valid when $X=p^\gamma$ with $N-1<\gamma<N$. This seems to
be a rather difficult problem.
\par
Nevertheless, the limit holds for all moments when $X$ is a suitable
function of $p$, and standard techniques from probability theory then
lead to central limit theorems for the random variables
$E_f(X,p,\ast)$ and $Z_f(X,p,\ast)$ for such functions $X$.

\begin{corint}[Central limit theorems]\label{coro_law}
  Let $f$ be an even or odd $GL(N)$ Hecke-Maass cusp form, which is
  not induced from a holomorphic form if $N=2$, and $w:\R_+^\ast\to\R$
  be a smooth and compactly supported function. Let $X=p^N/\Phi(p)$
  for a function $\Phi:[2,+\infty[\to[1,+\infty[$ satisfying
\begin{equation*}
\lim_{x\to+\infty}\Phi(x)=+\infty\;\;\;\text{ and }\;\;\;\Phi(x)=O_\epsilon(x^\epsilon)
\end{equation*}
for all $\epsilon>0$.
\begin{itemize}
\item If $f$ is self-dual then the sequence of random variables
  $E_f(X,p,\ast)$ converges in law to a centered Gaussian random
  variable with variance $2c_{f,w}$, as $p$ goes to
  infinity among the prime numbers. In other words, for all real
  numbers $\alpha<\beta$, we have
\begin{equation*}
\lim_{\substack{p\in\prem \\
p\to+\infty}}\frac{1}{p-1}\left\vert\left\{a\in\mathbb{F}_p^\times, \alpha\leq E_f(X,p,a)\leq\beta\right\}\right\vert=\frac{1}{\sqrt{2\pi\times 2c_{f,w}}}\int_{x=\alpha}^\beta \exp{\left(-\frac{x^2}{2\times 2c_{f,w}}\right)}\dd x.
\end{equation*} 
\item
If $f$ is not self-dual then the sequence of random vectors
$Z_f(X,p,\ast)$ converges in law to a Gaussian random vector with covariance matrix
\begin{equation}\label{eq_covariance}
c_{f,w}\begin{pmatrix}
1 & 0 \\
0 & 1
\end{pmatrix}
\end{equation}
as $p$ goes to infinity among the prime numbers. In other words, for
real numbers $\alpha_1<\beta_1$ and $\alpha_2<\beta_2$, we have
\begin{multline*}
  \lim_{\substack{p\in\prem \\
      p\to+\infty}}\frac{1}{p-1}\left\vert\left\{a\in\mathbb{F}_p^\times,
      Z_f(X,p,a)\in\left[\alpha_1,\beta_1\right]
      \times\left[\alpha_2,\beta_2\right]\right\}\right\vert \\
  =\frac{1}{2\pi
    c_{f,w}}\int_{x=\alpha_1}^{\beta_1}\int_{y=\alpha_2}^{\beta_2}
  \exp{\left(-\frac{x^2+y^2}{2 c_{f,w}}\right)}\dd x\dd y.
\end{multline*}
\end{itemize}
\end{corint}

\begin{remark}
(1)  The same central limit theorem would follow for all $X$ with
  $2p^{N-1}<X<p^N$ if one could prove that
  the limit~(\ref{eq-limit-moments}) holds for all $\kappa$ and
  $\lambda$ in that range. At the very least, for $X=p^{\gamma}$ with
  $N-1<\gamma<N$, we obtain convergence of all moments up to
  $\kappa+\lambda<\frac{1}{N-\gamma}$. 
\par
(2) It is a very interesting question whether one can establish this
result with the smooth weight $w(n/X)$ replaced by a characteristic
function of $1\leq n\leq X$. For $N=2$, Lester and
Yesha~\cite[Th. 1.1, Th. 1.2]{lester-yesha} have recently shown that
this can be done. This is a non-trivial fact, which has not been
extended to $N\geq 3$ at the moment.
\end{remark}

\subsubsection{The case of the multiple divisor functions}%
\label{ssec-dn}

The same techniques also enable us to study a similar problem for the
$N$-th multiple divisor function $d_N$. The only notable
difference is the existence of a significant main term.
\par
Thus, for an invertible residue class $a$ in
$\mathbb{F}_p^\times$, we define
$$
E_{d_N}(X,p,a)=\frac{S_{d_N}(X,p,a)-M_{d_N}(X,p)}{(X/p)^{1/2}},
$$ 
where
\begin{align}\label{eq-def-dn1}
S_{d_N}(X,p,a) & = \sum_{\substack{n\geq 1 \\
n\equiv a\bmod{p}}}d_N(n)w\left(\frac{n}{X}\right) \\
\label{eq-def-dn2}
M_{d_N}(X,p) & = \frac{1}{p}\sum_{n\geq
  1}d_N(n)w\left(\frac{n}{X}\right)-\frac{1}{p^2}\int_{x=0}^{+\infty}\sum_{k=1}^N\frac{\beta_k(p)}{(k-1)!}\log^{k-1}{(x)}w(x)\dd
x
\end{align}
where $w:\R_+^\ast\to\R$ is again a fixed non-zero smooth function
compactly supported on $[x_0,x_1]\subset\R_+^\ast$ and $\beta_k(p)$
are certain coefficients that we will define precisely in
Section~\ref{sec-dn}. Once again,
the normalisation is suggested by the square root cancellation philosophy.
\par
We will study the convergence in law of the sequence of random
variables $a\mapsto E_{d_N}(X,p,a)$ on $\mathbb{F}_p^{\times}$ endowed with its uniform probabiblity measure $\mu_p$.
\par
For $\kappa$ a non-negative integer, let us define the $\kappa$-th \emph{moment} of $E_{d_N}(X,p,a)$ by
\begin{equation*}
\mathsf{M}_{d_N}(X,p,\kappa)\coloneqq\frac{1}{p}\sum_{\substack{a\bmod{p} \\
(a,p)=1}}E_{d_N}(X,p,a)^{\kappa}.
\end{equation*}
The next theorem states an asymptotic expansion for these moments in
specific ranges for $X$ with respect to $p$.

\begin{theoint}[Moments for $d_N$]\label{theo_moments_d_N}
Let $N\geq 3$ and $w:\R_+^\ast\to\R$ be a smooth and compactly supported function. If $\;2p^{N-1}<X\leq p^N$ then
\begin{multline}\label{eq_theo_moments_d_N}
  \mathsf{M}_{d_N}(X,p,\kappa)=m_{\kappa}Q\Bigl(\log
    \frac{p^N}{X}\Bigr)^{\kappa/2}
\\
  +O_{\epsilon,
    f}\left(p^{1+\epsilon}\left(\frac{p^{N-1}}{X}\right)^{\kappa/2-1}+\left(\frac{X}{p^N}\right)^{1/2+\epsilon}+\frac{1}{\sqrt{p}}\left(\frac{p^N}{X}\right)^{\frac{\kappa}{2}+\epsilon}\right)
\end{multline}
for all $\epsilon>0$, where $Q$ is a polynomial of degree $N^2-1$ with
leading coefficient 
\begin{equation*}
  \frac{\abs{\abs{w}}_2^2}{(N^2-1)!}
  \prod_{q\text{ prime}}\Bigl(1-\frac{1}{q}\Bigr)^{(N-1)^2}
\times  \sum_{k=0}^{N-1}\binom{N-1}{k}^2p^{-k}
\end{equation*}
as a leading coefficient, where $\|w\|_2$ is the $L^2$-norm of $w$
with respect to the Lebesgue measure on $\R_+^\ast$.
\end{theoint}

\par
As in the case of Maass cusp forms, we deduce a central limit theorem for $E_{d_N}(X,p,a)$ for a large class of functions $X$.
\begin{corint}[Central limit theorems for $d_N$]\label{coro_law_d_N}
Let $N\geq 3$ and $w:\R_+^\ast\to\R$ be a smooth and compactly supported function. Let $X=p^N/\Phi(p)$ for a function $\Phi:[2,+\infty[\to[1,+\infty[$ satisfying
\begin{equation*}
\lim_{x\to+\infty}\Phi(x)=+\infty\;\;\;\text{ and }\;\;\;\Phi(x)=O_\epsilon(x^\epsilon)
\end{equation*}
for all $\epsilon>0$. Let
$$
H= \prod_{q\text{ prime}}\Bigl(1-\frac{1}{q}\Bigr)^{(N-1)^2}
\times  \sum_{k=0}^{N-1}\binom{N-1}{k}^2p^{-k}>0.
$$
\par
The sequence of random variables
\begin{equation*}
\frac{E_{d_N}(X,p,\ast)}{\sqrt{\frac{H\abs{\abs{w}}_2^2\log^{N^2-1}{\left(\Phi(p)\right)}}{(N^2-1)!}}}
\end{equation*}
converges in law to a centered Gaussian random variable, whose variance is $1$, as $p$ goes to infinity among the prime numbers. In other words, for all real numbers $\alpha<\beta$, we have
\begin{equation*}
\lim_{\substack{p\in\prem \\
p\to+\infty}}\frac{1}{p-1}\left\vert\left\{a\in\mathbb{F}_p^\times, \alpha\leq \frac{E_{d_N}(X,p,a)}{\sqrt{\frac{H_N(1)\abs{\abs{w}}_2^2\log^{N^2-1}{\left(\Phi(p)\right)}}{(N^2-1)!}}}\leq\beta\right\}\right\vert=\frac{1}{\sqrt{2\pi}}\int_{x=\alpha}^\beta \exp{\left(-\frac{x^2}{2}\right)}\dd x.
\end{equation*} 
\end{corint}

\begin{remark}
  As a final remark, we note that it is possible to extend the Central
  Limit Theorems, for cusp forms as well as for $d_N$, to restrict the
  average over residue classes $a$ which are considered, in either of
  the following manners (which may be combined):
\begin{itemize}
\item We may assume that $a$ ranges over the reduction modulo $p$ of
  an interval $I_p$ of integers of length $p^{1/2+\delta}\ll |I_p|\leq
  p-1$ for any fixed $\delta>0$;
\item We may assume that $a$ ranges over the set $f(\mathbb{F}_p)$ of
  values in $\mathbb{F}_p$ of a fixed non-constant polynomial
  $f\in\Z[X]$, for instance that $a$ is restricted to be a quadratic
  residue. 
\end{itemize}
\par
This essentially only requires an extension of the results of
Section~\ref{sec-katz}, as recently discussed by \'E. Fouvry,
E. Kowalski and Ph. Michel in~\cite[Section 5.3]{FoKoMi22}.
\end{remark}
\subsection{Strategy of the proof}%

The basic strategy follows that of \'{E}. Fouvry, S. Ganguly,
E. Kowalski and P. Michel~\cite{FoGaKoMi} of applying the Vorono\u{\i}
summation formula, followed by equidistribution theorems for
hyper-Kloosterman sums. There is a significant increase in complexity
due to the context of $GL(N)$ automorphic forms, and to the fact that
the Fourier coefficients are not always real-valued. However, more
importantly, a number of crucial facts which could be checked
relatively easily by direct calculations or ad-hoc methods in the
$GL(2)$ case (for holomorphic forms or the divisor function) require
much more intrinsic arguments. This is the case for instance of the
unitarity of the integral transform that appears in the Vorono\u{\i}
summation formula, and also of certain multiplicity computations from
representation theory which seem very difficult to handle with
explicit integrals. In addition, the computation of the limiting
variance is surprisingly delicate, and indeed is the only place where
we need to invoke an upper-bound for the Fourier coefficients of $f$,
which goes beyond the ``local'' Jacquet-Shalika bound, that follows
from genericity of the local representations.
\par
The most technical part of the argument is contained in
Section~\ref{sec_7}, where we obtain the asymptotic expansion of the
moments. However, the idea underlying this computation can be
motivated using probabilistic analogies, and we do this at the
beginning of that section.

\subsection{Organisation of the paper}%
The general background on $GL(N)$ Maass cusp forms is given in Section
\ref{sec_2}. The Vorono\u{\i} summation formula is introduced in
Section \ref{sec_3}, which also contains the analytic and the
unitarity properties of the generalized Bessel transforms occuring in
this summation formula. The algebraic ingredient required to prove the
crucial equidistribution result for products of hyper-Kloosterman sums
is done in Section \ref{sec-katz}. The technical ingredient needed to
achieve the variance computation is proved in Section
\ref{sec_variance}. The first steps in the proof of Theorem
\ref{theo_moments}, such as the input of the Vorono\u{\i} summation
formula, are done in Section \ref{sec_6}. The combinatorial analysis
in the proof of Theorem \ref{theo_moments} appears in Section
\ref{sec_7}. Corollary \ref{coro_law} is proved in Section \ref{sec_8} and Theorem \ref{theo_moments_d_N} in Section \ref{sec_9}. The general properties of Maass cusp forms, which are stated in Section \ref{sec_2} without proof in \cite{MR2254662}, are proved in Appendix \ref{sec_A}. A generating series involving a product of Schur polynomials (respectively a product of multiple divisor functions) is studied in Appendix \ref{sec_B} (respectively in Appendix \ref{sec_C}). 
\begin{notations}

As already mentioned, $N\geq 2$ is a fixed integer. We denote
$e(z)\coloneqq\exp(2i\pi z)$ for $z\in\C$. The sign of a non-zero real number $x$
is denoted $\sgn(x)\in \{-1,1\}$.
\par
$\prem$ stands for the set of prime numbers. The main parameters in
this paper are an \emph{odd} prime number $p$, which goes to infinity
among $\prem$ and a positive real number $X$, which goes to infinity
with $p$. Thus, if $f$ and $g$ are some $\C$-valued functions on
$\R^2$ then the symbols $f(p,X)\ll_{A}g(p,X)$ or equivalently
$f(p,X)=O_A(g(p,X))$ mean that $\abs{f(p,X)}$ is smaller than a
constant, which only depends on $A$, times $g(p,X)$ at least for $p$ a
large enough prime number.
\par
The Mellin transform of a function $\;\psi:\R_+^\ast\to\C$ is denoted
$\mathcal{M}[\psi]$ and is given by
\begin{equation*}
\mathcal{M}[\psi](s)=\int_{x=0}^{+\infty}\psi(x)x^s\frac{\dd x}{x}
\end{equation*}
for all complex numbers $s$ for which the integral exists. If $G$ is a
holomorphic function defined for $s\in \C$ with real part
$>\sigma_0\geq -\infty$ and with fast enough decay as the imaginary
part grows, then $\mathcal{M}^{-1}[G]$ will denote its inverse Mellin
transform defined by
\begin{equation*}
\mathcal{M}^{-1}[G](x)=\frac{1}{2i\pi}\int_{(\sigma)}G(s)\frac{\dd s}{x^s}
\end{equation*}
for a fixed $\sigma>\sigma_0$ and for all positive real number $x$.
\par
If $\;E$ is a finite set then $\vert E\vert$ stands for its cardinality.
\par
If $\mathcal{Q}$ is an assertion then the Kronecker symbol
$\delta_{\mathcal{Q}}$ equals $1$ if $\mathcal{Q}$ is true and $0$
otherwise.
\end{notations}

\begin{merci}
\par
The authors would like to thank Valentin Blomer, Farrell Brumley, \'{E}tienne Fouvry,
Dorian Goldfeld and Henry Gould for stimulating exchange related to
this project.
\par
We heartily thank the referee for a very careful reading of the
manuscript; the long list of constructive suggestions contained in his
or her report greatly improved the presentation.
\par
The research of G. Ricotta is supported by a Marie Curie Intra
European Fellowship within the 7th European Community Framework
Programme. The grant agreement number of this project, whose acronym
is ANERAUTOHI, is PIEF-GA-2009-25271. He would like to thank ETH and
its entire staff for the excellent working conditions. His interest in
the analytic theory of $GL(N)$ automorphic forms arose throughout the
reading of the very explicit book writen by D. Goldfeld \cite{MR2254662} and reached his peak after the stimulating founding
workshop ``Analytic Theory of $GL(3)$ automorphic forms and
applications'' at the American Institute of Mathematics in November
2008. He would like to thank the organisers (H. Iwaniec, P. Michel and
K. Soundararajan) for their kind invitation. Last but not least, he
would like to express his gratitude to K. Belabas for his crucial
support among the Number Theory research team A2X (Institut de
Mathématiques de Bordeaux).
\end{merci}

\section{Quick review of $GL(N)$ automorphic forms}\label{sec_2}%
A convenient reference for this section is \cite{MR2254662}. Let $f$
be a $GL(N)$ Hecke-Maass cusp form of level $1$ and let $f^\ast$ be
its dual. If $N=2$, we require (for convenience) that the corresponding
classical modular form is not holomorphic. 
\par
For positive integers $m_1,\dots,m_{N-2}$ and a non-zero integer
$m_{N-1}$, we denote by 
$$
a_f\left(m_1,\dots,m_{N-1}\right)
$$ 
the $\left(m_1,\dots,m_{N-1}\right)$'th Fourier coefficient of $f$. We
assume that $f$ is arithmetically normalized, namely
$a_f(1,\dots,1)=1$. Since $f$ is a Hecke eigenform, the multiplicity
$1$ theorem shows that $f$ is either even or odd, i.e.,
that
\begin{equation}\label{eq_even}
a_f\left(m_1,\dots,-m_{N-1}\right)=\epsilon_f a_f\left(m_1,\dots,m_{N-1}\right)
\end{equation}
where
\begin{equation}\label{eq_parity}
\epsilon_f\coloneqq\begin{cases}
+1 & \text{if $f$ is even,} \\
-1 & \text{if $f$ is odd}
\end{cases}
\end{equation}
by \cite[Proposition 9.2.5, Proposition 9.2.6]{MR2254662}. More precisely, a $GL(N)$ Maass cusp form of level $1$ is always a linear combination of an even and an odd one by \cite[Definition 9.2.4]{MR2254662}. If $N$ is
odd, then it is known that a $GL(N)$ Maass cusp form of level $1$ is always even (see~\cite[Proposition
6.3.5]{MR2254662}). If $N$ is even and $f$ is a $GL(N)$ Hecke-Maass cusp form of level $1$ then one can check that $f$ and $K(f)$ defined by
\begin{equation*}
K(f)(z)\coloneqq f\left(\text{diag}(-1,1,\dots,1)z\right)
\end{equation*}
for $z$ in the generalized upper-half plane have the same Hecke eigenvalues (this follows from the fact that $K$
commutes with the Hecke algebra). Since $K$ is an involution, the
multiplicity $1$ theorem implies the result (see~\cite[Theorem
6.28]{Iw} and ~\cite[Section 9.2]{MR2254662} for more details).
\par
The Fourier coefficients satisfy the Ramanujan-Petersson bound on
average, by Rankin-Selberg theory. Recall that the Rankin-Selberg
$L$-function of $f$ and another $GL(N)$ Hecke-Maass cusp form $g$
of level $1$ is the Dirichlet series
\begin{equation*}
L(f\times g,s)=\zeta(Ns)\sum_{m_1,\dots,m_{N-1}\geq 1}\frac{a_f\left(m_1,\dots,m_{N-1}\right)a_g\left(m_1,\dots,m_{N-1}\right)}{(m_1^{N-1}m_2^{N-2}\dots m_{N-1})^s}.
\end{equation*}
This $L$-function has an analytic continuation to $\C$ if $g\neq
f^\ast$, and a meromorphic continuation to $\C$ with a simple pole at
$s=1$ if $g=f^{\ast}$ (see \cite[Theorem 12.1.4]{MR2254662}). The
residue of $L(f\times f^{\ast},s)$ at $s=1$ is denoted $r_f$. It is a
positive real number, and it may be expressed in terms of invariants
of $f$, see Proposition \ref{propo_residue}.
\par
The Rankin-Selberg $L$-function has also an Euler product of degree
$N^2$ given by
\begin{equation}\label{eq_Euler}
L(f\times g,s)=\prod_{q\in\prem}\prod_{1\leq j,k\leq N}
\left(1-\frac{\alpha_{j,q}(f)\alpha_{k,q}(g)}{q^s}\right)^{-1}
\end{equation}
by \cite[Proposition 12.1.3]{MR2254662} where the $\alpha_{j,q}(f)$'th
are the complex roots of the monic polynomial
\begin{equation}\label{eq_alpha}
X^N+\sum_{\ell=1}^{N-1}(-1)^{\ell}
a_f(\overbrace{1,\dots,1}^{\text{$\ell-1$
    terms}},q,1,\dots,1)X^{N-\ell}+ (-1)^N\in \C[X]
\end{equation}
(where the Fourier coefficient corresponding to $\ell$ has index $q$
at the $\ell$-th position).
\par
For a prime number $q$, we denote for convenience
\begin{equation}\label{eq_alphap}
\alpha_{q}(f)\coloneqq\left\{\alpha_{j,q}(f), 1\leq j\leq N\right\}
\end{equation}
and we remark that \eqref{eq_alpha} and \eqref{eq_Fourier_dual}
imply that
\begin{equation*}
\alpha_q(f^\ast)=\left\{\overline{\alpha_{j,q}(f)}, 1\leq j\leq N\right\}.
\end{equation*}
From \eqref{eq_alpha}, we also find that
\begin{equation}\label{eq_alpha_sym}
a_f(\overbrace{1,\dots,1}^{\text{$\ell-1$}},q,1,\dots,1)
=e_\ell(\alpha_q(f))
\coloneqq\sum_{1\leq j_1<\dots<j_\ell\leq N}\alpha_{j_1,q}(f)\dots\alpha_{j_\ell,q}(f)
\end{equation}
for $1\leq\ell\leq N-1$. More generally, it follows from the works of
Shintani and of Casselman--Shalika (see also \cite[Proposition
5.1]{Zh}) that, for a prime number $q$ and $N-1$ non-negative integer
$k_1,\dots,k_{N-1}$, we have
\begin{equation}\label{eq_Fourier_schur}
a_f\left(q^{k_1},\dots,q^{k_{N-1}}\right)=S_{k_{N-1},\dots,k_1}\left(\alpha_{1,q}(f),\dots,\alpha_{N,q}(f)\right) 
\end{equation}
where
\begin{equation}\label{eq_def_schur}
S_{k_{N-1},\dots,k_{1}}\left(x_1,\dots,x_N\right)=\frac{1}{V(x_1,\dots,x_N)}
\det\begin{pmatrix}
x_1^{N-1+k_{N-1}+\dots+k_{1}} & \dots & x_N^{N-1+k_{N-1}+\dots+k_{1}} \\
\vdots & \vdots & \vdots \\
x_1^{2+k_{N-1}+k_{N-2}} & \dots & x_N^{2+k_{N-1}+k_{N-2}} \\
x_1^{1+k_{N-1}} & \dots & x_N^{1+k_{N-1}} \\
1 & \dots & 1
\end{pmatrix}
\end{equation}
is a Schur polynomial and where $V(x_1,\dots,x_N)$ stands for the usual
Vandermonde determinant
\begin{equation*}
V(x_1,\dots,x_N)\coloneqq\prod_{1\leq i<j\leq N}\left(x_i-x_j\right).
\end{equation*}
\par
We will need the following property, which we will explain in
Proposition~\ref{propo_schur}: there exist polynomials
$P_N(\uple{x},\uple{y},T)$, where $\uple{x}=(x_1,\ldots, x_N)$ and
$\uple{y}=(y_1,\ldots, y_N)$ are indeterminates, such that
\begin{equation}\label{eq-schur}
\sum_{k\geq 0}S_{0,\dots,0,k}(\uple{x})S_{0,\dots,0,k}(\uple{y})T^k=\frac{P_N(\uple{x},\uple{y},T)}{\prod_{1\leq j,k\leq N}\left(1-x_jy_kT\right)}
\end{equation}
\par
The analytic properties of the Rankin-Selberg $L$-functions are known
to imply that
\begin{equation}\label{eq_RP_average}
\sum_{\substack{m_1,\dots,m_{N-1}\geq 1 \\
m_1^{N-1}m_2^{N-2}\cdots m_{N-1}\leq X}}
\abs{a_f\left(m_1,\dots,m_{N-1}\right)}^2\ll_{\epsilon, f} X^{1+\epsilon}
\end{equation}
for all real number $X\geq 1$ and $\epsilon>0$. This bound on average
for the Fourier coefficients of $f$ is strong enough in all the
analytic estimates in this work, except when computing the variance in
Section \ref{sec_variance}, which requires a non-trivial individual
bound for Satake parameters which is stronger than what is implied by
this bound.
\par
More precisely, recall that W. Luo, Z. Rudnick and P. Sarnak have
proved in \cite{MR1334872} and \cite{MR1703764} that
\begin{equation}\label{eq_RP}
  \max_{1\leq j\leq N}{\abs{\alpha_{j,q}(f)}}\leq q^{1/2-1/(N^2+1)}
\end{equation}
for all Hecke-Maass cusp forms $f$ of level $1$ and all prime numbers
$q$. (The Ramanujan-Petersson conjecture claims that this should hold
with $1$ on the right-hand side, and the Jacquet-Shalika local bound
shows that it does with $q^{1/2}$ instead).
\par
By \cite[Theorem 9.3.11]{MR2254662}, the Fourier coefficients of $f$ satisfy the multiplicativity relations
\begin{equation}\label{eq_multiplicative}
  a_f(m,1,\dots,1)a_f\left(m_1,\dots,m_{N-1}\right)=\sum_{\substack{\prod_{\ell=1}^Nc_\ell=m \\
      c_j\mid m_j (1\leq j\leq N-1)}}a_f\left(\frac{m_1c_N}{c_1},\frac{m_2c_1}{c_2},\dots,\frac{m_{N-1}c_{N-2}}{c_{N-1}}\right)
\end{equation}
for positive integers $m, m_1,\dots,m_{N-2}$ and non-zero integer
$m_{N-1}$ (there is no typo), as well as
\begin{equation}\label{eq_multiplicative_bis}
a_f\left(m_1m_1^\prime,\dots,m_{N-1}m_{N-1}^\prime\right)=a_f\left(m_1,\dots,m_{N-1}\right)a_f\left(m_1^\prime,\dots,m_{N-1}^\prime\right)
\end{equation}
for positive integers $m_1, m_1^\prime,\dots, m_{N-2},\ m_{N-2}^\prime$,
and non-zero integers $m_{N-1}, m_{N-1}^\prime$ such that
$$
\left(m_1\dots m_{N-1},m_1^\prime\dots m_{N-1}^\prime\right)=1.
$$
\par
We also mention that, for positive integers $m_1,\dots, m_{N-1}$, we
have
\begin{equation}\label{eq_Fourier_dual}
a_{f^\ast}(m_1,\dots,m_{N-2},m_{N-1})=a_{f}(m_{N-1},m_{N-2},\dots,m_{1})
\end{equation}
by \cite[Theorem 9.3.11, Addendum]{MR2254662}. Using the fact that $f$
is a Hecke eigenfunction, one derives by Möbius inversion the relation
\begin{equation}\label{eq_Fourier_coeffs}
a_f\left(m_1,\dots,m_{N-2},m_{N-1}\right)=
\overline{a_f\left(m_{N-1},m_{N-2},\dots,m_1\right)},
\end{equation} 
(see \cite[Theorem 9.3.6, Theorem 9.3.11, Addendum]{MR2254662}) and in
particular, we see that the Fourier coefficients of $f$ are real if
$f$ is self-dual, i.e., if $f=f^{\ast}$. Recalling the
definition~(\eqref{eq-af} and \eqref{eq-aff}), we see that
\begin{equation}\label{eq-afst}
  \overline{\af(m)}=\overline{a_f(m,1,\ldots,1)}=
  a_f(1,\ldots, 1,m)=\afs(m)
\end{equation}
for $m\geq 1$.
\par
We now consider analogues of some of these properties at the infinite
place. We denote by 
$$
\nu(f)=(\nu_1(f),\dots,\nu_{N-1}(f))\in\C^{N-1}
$$
the \emph{type} of $f$. The components of the type of $f$ are complex
numbers characterized by the property that, for every invariant
differential operator $D$ in the center of the universal enveloping
algebra of $GL(N,\R)$, the cusp form $f$ is an eigenfunction of $D$
with the same eigenvalue as the power function $I_{\nu(f)}$ which is
defined in \cite[Equation (5.1.1)]{MR2254662}.
\par
On the other hand, we denote by
$$
\alpha_\infty(f)=\left\{\alpha_{j,\infty}(f), 1\leq j\leq N\right\}
$$
the Langlands parameters\footnote{For reference, we note that
  $\alpha_{j,\infty}(f)$ is denoted $\lambda_j(\nu)$ and
  $\alpha_{j,\infty}(f^\ast)$ is denoted $\widetilde{\lambda_j}(\nu)$
  in \cite{MR2254662}.}  of $f$.
\par
The Langlands parameters are obtained as a set of affine combinations
of the coefficients of the type. They satisfy
\begin{equation}\label{eq_sum_alpha}
\sum_{j=1}^N\alpha_{j,\infty}(f)=0.
\end{equation}
and
\begin{equation}\label{eq_Langlands_fact}
\alpha_{\infty}(f^\ast)=-\alpha_{\infty}(f)
\end{equation}
since the type of $f^\ast$ is
$\nu(f^\ast)=(\nu_{N-1}(f),\nu_{N-2}(f),\dots,\nu_1(f))$ (see
\cite[Proposition 9.2.1]{MR2254662}).
\par
We also have the unitarity property (see~\cite[Equation A.2]{RuSa})
\begin{equation}\label{eq_unitaricity}
\alpha_\infty(f)=-\overline{\alpha_{\infty}(f)}
\end{equation}
or equivalently
\begin{equation}\label{eq_unitaricity_2}
\alpha_\infty(f^\ast)=\overline{\alpha_{\infty}(f)}
\end{equation}
by \eqref{eq_Langlands_fact}. It is known that
\begin{equation}\label{eq_Selberg}
  \max_{1\leq j\leq N}{\abs{\Re{\left(\alpha_{j,\infty}(f)\right)}}}\leq
  \frac{1}{2}-\frac{1}{N^2+1},
\end{equation}
(see~\cite{MR1334872,MR1703764}), and this analogue of~\eqref{eq_RP}
will also be used.

\section{Generalized Bessel transforms for $GL(N)$}\label{sec_3}%
\subsection{The Vorono\u{\i} summation
  formula}%

The first case of the Vorono\u{\i} summation formula beyond $GL(2)$ is
due to S.J. Miller and W. Schmid, for $GL(3)$ cusp forms (see
\cite{MR2247965}; note that according to \cite[Section 1.2]{IcTe} and
\cite[Page 4]{MR2418857}, P. Sarnak and T. Watson had developed before
a version of the Vorono\u{\i} summation formula for $GL(3)$ for prime
denominators). D. Goldfeld and X. Li developed a Vorono\u{\i}
summation formula for $GL(N)$ for prime denominators in
\cite{MR2233713} and for general denominators in
\cite{MR2418857}.
Independently, 
S.J. Miller and W. Schmid found a more general version of the
Vorono\u{\i} summation formula for $GL(N)$ in \cite{MR2882444}.

The version we use is both a particular case and a slightly
renormalized version of the formulas given in \cite[Theorem
4.1]{MR2233713} and in \cite[Theorem 1]{IcTe}, which, among other
things, takes into account the properties \eqref{eq_even} and
\eqref{eq_Fourier_coeffs} satisfied by the Fourier coefficients of
$f$.

In order to state the formula, we first define the required integral
transforms.

Given an $N$-tuple $\uple{\alpha}=(\alpha_1,\ldots,\alpha_N)$ of
complex numbers and an integer $k\in \{0,1\}$, we denote
\begin{equation*}
  \fGamma_{k,\uple{\alpha}}(s)\coloneqq\prod_{1\leq j\leq N}\fGamma_\R\left(s+\alpha_j+k\right)
\end{equation*}
where $\fGamma_\R(s)\coloneqq\pi^{-s/2}\fGamma(s/2)$ for all complex
number $s$. We write
$\uple{\alpha}^{\ast}=\overline{\uple{\alpha}}=(\overline{\alpha_j})_{1\leq j\leq N}$.

Given a smooth function $w$ with compact support on $\R_+^{\ast}$,
we then define
\begin{align}
\mathcal{B}_{k,\uple{\alpha}}[w](x) & \coloneqq\frac{1}{2i\pi}
\int_{(\sigma)}\frac{\fGamma_{k,\uple{\alpha}}(s)}
{\fGamma_{k,\uple{\alpha}^{\ast}}(1-s)}
\mathcal{M}[w](1-s)\frac{\dd s}{x^s} \label{eq_Mellin_psi} \\
& =\mathcal{M}^{-1}
\left[s\mapsto\frac{\fGamma_{k,\uple{\alpha}}(s)}
{\fGamma_{k,\uple{\alpha}^{\ast}}(1-s)}
\mathcal{M}[w](1-s)\right](x) \label{eq_Mellin_psi_2}
\end{align}
for all positive real number $x$ and $\sigma>\max_{1\leq j\leq
  N}{(-\Re{(\alpha_{j})})}$, and
\begin{equation}\label{eq_CL}
  \mathcal{B}_{\uple{\alpha}}^{\pm}[w]
  \coloneqq\frac{1}{2}
  \left(\mathcal{B}_{0,\uple{\alpha}}[w]
    \mp\frac{1}{i^N}\mathcal{B}_{1,\uple{\alpha}}[w]\right),
\end{equation}
which are functions defined for $x>0$, and finally
\begin{equation}\label{eq_CL_2}
\mathcal{B}_{\uple{\alpha}}[w](x)\coloneqq
\mathcal{B}_{\uple{\alpha}}^{\sgn(x)}[w](\abs{x})
\end{equation}
for all non-zero real numbers $x$.

Moreover, we recall the definition of hyper-Kloosterman sums. For $r
\geq 1$ a positive integer, $\mathbb{F}$ a finite field of characteristic $p$
with $\abs{\mathbb{F}}=q$ and $u\in\mathbb{F}$, we denote
\begin{equation}\label{eq_Kloos_def}
  K_{r}(u,q)=\frac{1}{q^{\frac{r-1}{2}}}\sum_{\substack{(x_1,\dots,x_{r})\in
      (\mathbb{F}^{\ast})^{r} \\
      x_1\dots x_{r}=u}}\psi_\mathbb{F}(x_1+\cdots+ x_{r}),
\end{equation}
where $\psi_\mathbb{F}$ denotes the additive character given by
$$
\psi_\mathbb{F}(x)=e\left(\frac{\mathrm{Tr}_{\mathbb{F}/\mathbb{F}_p}(x)}{p}\right).
$$

\begin{proposition}[Vorono\u{\i} summation formula for
  $GL(N)$]\label{propo_formule_voronoi}
  Let $N\geq 2$ be an integer and $f$ a Hecke-Maass cusp form on
  $GL(N)$ of level $1$.  Let $w:\R_+^\ast\to\R$ be a smooth and
  compactly supported function. Let $p$ be a prime number
  and let $b$ be an integer. If $p$ does not divide $b$ then
\begin{multline}\label{eq_Voronoi}
\sum_{n\geq 1}\af(n)e\left(\frac{bn}{p}\right)w(n)=\frac{\epsilon_f}{p^{\frac{N}{2}}}\sum_{m\in\Z^\ast}\afs(m)K_{N-1}(\bar{b}m,p)\mathcal{B}_{\alpha_\infty(f)}[w]\left(\frac{m}{p^N}\right) \\
+\epsilon_f\sum_{\ell=1}^{N-2}\frac{(-1)^{\ell+1}}{p^\ell}\sum_{m\in\Z^\ast}a_f(\overbrace{1,\dots,1}^{\text{$\ell-1$}},p,1,\dots,1,m)\mathcal{B}_{\alpha_\infty(f)}[w]\left(\frac{m}{p^{\ell}}\right)
\end{multline}
where $\bar{b}$ denotes the inverse of $b$ modulo $p$. The second sum
is zero if $N=2$.
\end{proposition}

\begin{proof}[\proofname{} of proposition \ref{propo_formule_voronoi}]
  When $N$ is odd, \eqref{eq_Voronoi} can be deduced directly from
  \cite[Theorem 4.1]{MR2233713}. Let us assume then that $N$ is even
  and let us check that \eqref{eq_Voronoi} can be deduced from
  \cite[Theorem 2]{IcTe}. The explicit links between their notations
  and ours are given in \cite[Remark 3]{IcTe}. Let $\pi(f)$ be the
  automorphic cusp form of $GL(N,\mathbb{A}_\Q)$ associated to $f$ and
  let $\pi_{\infty}(f)$ be its archimedean component. For
  $\chi\in\left\{1,\text{sgn}\right\}$ one of the two unitary
  characters of $\R^\ast$, the duality between $w$ and
  $\mathcal{B}_{\alpha_\infty(f)}[w]$ is given by
\begin{multline*}
\int_{y=0}^{+\infty}\left(\mathcal{B}_{\alpha_\infty(f)}[w](y)+\mathcal{B}_{\alpha_\infty(f)}[w](-y)\chi(-y)\right)y^s\frac{\dd s}{y} \\
=\chi(-1)^{N-1}\gamma\left(1-s,\pi_\infty(f^\ast)\times\chi,\psi_\infty\right)\int_{y=0}^{+\infty}w(y)\chi(y)y^{1-s}\frac{\dd s}{y}
\end{multline*}
according to \cite[Lemma (5.2)]{IcTe} for all $s$ of real part
sufficiently large, where
\begin{equation*}
\gamma\left(1-s,\pi_\infty(f^\ast)\times\chi,\psi_\infty\right)=\epsilon\left(s,\pi_\infty(f^\ast)\times\chi,\psi_\infty\right)\frac{L\left(s,\pi_\infty(f)\times\chi\right)}{L\left(1-s,\pi_\infty(f^\ast)\times\chi\right)}
\end{equation*}
by \cite[Section 5.1]{IcTe}. Consequently,
\begin{multline*}
\mathcal{M}\left[y\mapsto\mathcal{B}_{\alpha_\infty(f)}[w](\epsilon y)\right](s)=\mathcal{M}[w](1-s) \\
\times\frac{1}{2}\left(\gamma\left(1-s,\pi_\infty(f^\ast),\psi_\infty\right)+\epsilon\times(-1)^{N-1}\gamma\left(1-s,\pi_\infty(f^\ast)\times\text{sgn},\psi_\infty\right)\right)
\end{multline*}
for $\epsilon=\pm 1$. 
We have
\begin{eqnarray*}
\epsilon\left(s,\pi_\infty(f^\ast),\psi_\infty\right) & = & \epsilon_f, \\
\epsilon\left(s,\pi_\infty(f^\ast)\times\text{sgn},\psi_\infty\right) & = & \epsilon_fi^N
\end{eqnarray*}
and
\begin{eqnarray*}
L\left(s,\pi_\infty(f)\right) & = & \prod_{j=1}^N\fGamma_\R\left(s+\alpha_{j,\infty}(f)\right), \\
L\left(s,\pi_\infty(f)\times\text{sgn}\right) & = & \prod_{j=1}^N\fGamma_\R\left(s+\alpha_{j,\infty}(f)+1\right).
\end{eqnarray*}
Noting that
\begin{equation*}
(-1)^{N-1}i^N=-\frac{1}{i^N},
\end{equation*}
the formula follows as stated.
\end{proof}

The following useful lemma relates the Bessel transforms for $f$ and
its dual.

\begin{lemma}\label{lemma_easy_Bessel}
Let $k\in\{0,1\}$ and $w:\R_+^\ast\to\R$ a smooth and compactly supported function. One has
\begin{equation*}
  \overline{\mathcal{B}_{k,\alpha_\infty(f)}[w]}=\mathcal{B}_{k,\alpha_\infty(f^\ast)}[w]
\end{equation*}
and
\begin{equation*}
  \overline{\mathcal{B}_{\alpha_\infty(f)}^{\pm}[w]}=\mathcal{B}_{\alpha_\infty(f^\ast)}^{\pm(-1)^N}[w] \;\;\;\text{ and } \;\;\;\overline{\mathcal{B}_{\alpha_\infty(f)}[w](x)}=\mathcal{B}_{\alpha_\infty(f^\ast)}[w]((-1)^Nx)
\end{equation*}
for all non-zero real number $x$.
\end{lemma}

\begin{proof}[\proofname{} of lemma \ref{lemma_easy_Bessel}]
  The second and third equalities are direct consequences of the first
  one by \eqref{eq_CL} and \eqref{eq_CL_2}. Let us quickly check the
  first one. Denote $\uple{\alpha}=\alpha_{\infty}(f)$ so that
  $\alpha_\infty(f^{\ast})=\uple{\alpha}^{\ast}$
  by~(\ref{eq_unitaricity_2}). By \eqref{eq_Mellin_psi}, we have
\begin{align*}
\overline{\mathcal{B}_{k,\uple{\alpha}}[w](x)} & =\frac{1}{2i\pi}\int_{(\sigma)}\frac{\overline{\fGamma_{k,\uple{\alpha}}(s)}}{\overline{\fGamma_{k,\uple{\alpha}^{\ast}}(1-s)}}\overline{\mathcal{M}[w](1-s)}\frac{\dd s}{x^{\overline{s}}} \\
& =\frac{1}{2i\pi}\int_{(\sigma)}\frac{\fGamma_{k,\overline{\uple{\alpha}}}(\overline{s})}{\overline{\fGamma_{k,\overline{\uple{\alpha}^{\ast}}}(1-\overline{s})}}\mathcal{M}[w](1-\overline{s})\frac{\dd s}{x^{\overline{s}}} \\
&
=\frac{1}{2i\pi}\int_{(\sigma)}\frac{\fGamma_{k,\uple{\alpha}^{\ast}}(\overline{s})}{\overline{\fGamma_{k,\uple{\alpha}}(1-\overline{s})}}\mathcal{M}[w](1-\overline{s})\frac{\dd
  s}{x^{\overline{s}}} \\
  & =\mathcal{B}_{k,\alpha_\infty(f^\ast)}[w]
\end{align*}
by \eqref{eq_unitaricity_2}.
\end{proof}
\subsection{Unitarity of the generalized Bessel transforms}%

A key ingredient in the computation of the variance in Sections
\ref{sec_variance_1}, \ref{sec_variance_2}, \ref{sec_variance_3} and
\ref{sec_variance_4} below will be the unitarity of the generalized
Bessel transforms in the following sense.

\begin{proposition}[Unitarity of the generalized Bessel transforms]\label{prop_unitary_Bessel}
  If $w:\R_+^\ast\to\R$ is a smooth and compactly supported
  function and $k\in\{0,1\}$ then
\begin{equation}\label{eq_unitary_Bessel}
  \abs{\abs{\mathcal{B}_{k,\alpha_\infty(f)}[w]}}_2=\abs{\abs{w}}_2
\end{equation}
where the $L^2$-norms are computed with respect to the Lebesgue measure $\dd x$ on $\R_+^\ast$.
\end{proposition}

\begin{proof}[\proofname{} of proposition \ref{prop_unitary_Bessel}]
Denote $\uple{\alpha}=\alpha_{\infty}(f)$. 
One gets successively
\begin{align*}
\abs{\abs{\mathcal{B}_{k,\uple{\alpha}}[w]}}_2^2 & =\int_{x=0}^{+\infty}\abs{\mathcal{B}_{k,\uple{\alpha}}[w](x)}^2\dd x \\
& =\int_{x=0}^{+\infty}\mathcal{B}_{k,\uple{\alpha}}[w](x)\overline{\mathcal{B}_{k,\uple{\alpha}}[w](x)}\dd x \\
& =\int_{x=0}^{+\infty}\mathcal{B}_{k,\uple{\alpha}}[w](x)\mathcal{B}_{k,\uple{\alpha}^{\ast}}[w](x)\dd x
\end{align*}
by Lemma \ref{lemma_easy_Bessel}. Then, the Parseval formula for the Mellin transform (namely the fact that the (suitably renormalized version of the) Mellin transform is a unitary operator) asserts that
\begin{equation*}
\abs{\abs{\mathcal{B}_{k,\uple{\alpha}}[w]}}_2^2=\frac{1}{2i\pi}\int_{(\sigma)}\mathcal{M}\left[\mathcal{B}_{k,\uple{\alpha}}[w]\right](s)\mathcal{M}\left[\mathcal{B}_{k,\uple{\alpha}^{\ast}}[w]\right](1-s)\dd s
\end{equation*}
for $\sigma$ large enough (see \cite[Theorem 1.17]{MR1391243}). By \eqref{eq_Mellin_psi_2},
\begin{align}
\abs{\abs{\mathcal{B}_{k,\uple{\alpha}}[w]}}_2^2 & =\frac{1}{2i\pi}\int_{(\sigma)}\frac{\fGamma_{k,\uple{\alpha}}(s)}{\fGamma_{k,\uple{\alpha}^{\ast}}(1-s)}\mathcal{M}[w](1-s) \\
& \times\frac{\fGamma_{k,\uple{\alpha}^{\ast}}(1-s)}{\fGamma_{k,\uple{\alpha}}(1-(1-s))}\mathcal{M}[w](1-(1-s))\dd s \\
& =\frac{1}{2i\pi}\int_{(\sigma)}\mathcal{M}[w](1-s)\mathcal{M}[w](s)\dd s \label{eq_func_local} \\
& =\abs{\abs{w}}_2^2
\end{align}
once again by the Parseval formula for the Mellin transform.
\end{proof}


\begin{corollary}\label{coro_unitary_Bessel}
Let $w:\R_+^\ast\to\R$ be a smooth and compactly supported function.
\begin{itemize}
\item
If $N$ is odd then
\begin{equation}\label{eq_mellin_odd}
\sum_{g\in\{f,f^\ast\}}\mathcal{M}\left[\left\vert\mathcal{B}_{\alpha_\infty(g)}[w]\right\vert^2\right](1)=\abs{\abs{w}}_2^2.
\end{equation}
\item
Independently of the parity of $N$,
\begin{equation}\label{eq_mellin_even}
\sum_{\epsilon\in\{\pm 1\}}\mathcal{M}\left[\left\vert\mathcal{B}_{\alpha_\infty(f)}^\epsilon[w]\right\vert^2\right](1)=\abs{\abs{w}}_2^2.
\end{equation}
\end{itemize}
\end{corollary}
\begin{proof}[\proofname{} of corollary \ref{coro_unitary_Bessel}]%
By Lemma \ref{lemma_easy_Bessel}, we have
\begin{equation*}
\mathcal{M}\left[\left\vert\mathcal{B}_{\alpha_\infty(g)}[w]\right\vert^2\right](1)=\mathcal{M}\left[\left\vert\mathcal{B}_{\alpha_\infty(g)}^{+1}[w]\right\vert^2\right](1)
\end{equation*}
for $g=f, f^\ast$ and
\begin{equation*}
\mathcal{M}\left[\left\vert\mathcal{B}_{\alpha_\infty(g)}^\epsilon[w]\right\vert^2\right](1)=\int_{x=0}^{+\infty}\mathcal{B}_{\alpha_\infty(g)}^\epsilon[w](x)\mathcal{B}_{\alpha_\infty(g^\ast)}^{(-1)^N\epsilon}[w](x)\dd x
\end{equation*}
for $g=f, f^\ast$ and $\epsilon=\pm 1$. A straightforward computation reveals that
\begin{multline*}
\mathcal{M}\left[\left\vert\mathcal{B}_{\alpha_\infty(g)}^\epsilon[w]\right\vert^2\right](1)=\frac{1}{4}\left(\abs{\abs{\mathcal{B}_{0,\alpha_\infty(g)}[w]}}_2^2+\abs{\abs{\mathcal{B}_{1,\alpha_\infty(g)}[w]}}_2^2\right) \\
-\frac{\epsilon}{4i^N}\mathcal{M}\left[\mathcal{B}_{0,\alpha_\infty(g^\ast)}[w]\mathcal{B}_{1,\alpha_\infty(g)}[w]+(-1)^N\mathcal{B}_{0,\alpha_\infty(g)}[w]\mathcal{B}_{1,\alpha_\infty(g^\ast)}[w]\right](1).
\end{multline*}
Proposition \ref{prop_unitary_Bessel} implies both \eqref{eq_mellin_odd}, if $N$ is odd, and \eqref{eq_mellin_even}.
\end{proof}
\subsection{Asymptotic behaviour of the generalized Bessel transforms}%
Bounds for the generalized Bessel transforms $\mathcal{B}_{\alpha_\infty(f)}^{\pm}[w]$ both for small and large arguments are required in this work.
\begin{proposition}\label{propo_bounds}
Let $w:\R_+^\ast\to\R$ be a smooth and compactly supported
function, $x$ be a positive real number and $K$ be a positive
integer. Let $\uple{\alpha}=\alpha_{\infty}(f)$ for some cusp form $f$
as before. 
\begin{itemize}
\item 
If $\;0<x\leq 1$ then
\begin{equation*}
\mathcal{B}_{\uple{\alpha}}^{\pm}[w](x)\ll\max_{1\leq j\leq
  N}x^{\Re{(\alpha_{j,\infty}(f))}}.
\end{equation*}
In particular, if $\;0<x\leq 1$ then
\begin{equation*}
\mathcal{B}_{\uple{\alpha}}^{\pm}[w](x)\ll x^{-\left(1/2-1/(N^2+1)\right)}
\end{equation*}
by~\emph{(\ref{eq_Selberg})}.
\item
If $x>0$ then
\begin{equation*}
\mathcal{B}_{\uple{\alpha}}^{\pm}[w](x)\ll_{A, \uple{\alpha}, w}\frac{1}{x^{A}}
\end{equation*}
for all positive real number $A$.
\end{itemize}
\end{proposition}
\begin{proof}[\proofname{} of proposition \ref{propo_bounds}]%
  For the first part, we can shift the contour in
  \eqref{eq_Mellin_psi} to the left, passing through simple poles at
  $z=-2n-\alpha_{j}-k$ for all non-negative integers $n$ and $1\leq
  j\leq N$, the largest contribution occuring when $n=0$.
\par
For the second part, we can shift the contour to the right to
$\Re{(s)}=A$ without encountering any singularity.
\end{proof}

For the next corollary, we recall the definition~(\eqref{eq-af} and \eqref{eq-aff}) of
$\af(m)$ for all integers $m\geq 1$.

\begin{corollary}\label{coro_useful}
Let $Z$ be a positive real number, $M_1\geq 1$ be a real number, $1\leq M_1\leq M_2\leq +\infty$ and $w:\R_+^\ast\to\R$ be a smooth and compactly supported function. One has
\begin{multline*}
\sum_{\substack{m\in\Z^\ast \\
M_1\leq\abs{m}\leq M_2}}\left\vert \af(m)\right\vert\left\vert\mathcal{B}_{\alpha_\infty(f)}[w]\left(\frac{m}{Z}\right)\right\vert\ll_{\epsilon, f}\delta_{Z\leq M_1}M_1^{1+\epsilon}\left(\frac{Z}{M_1}\right)^{A}  \\
+\delta_{M_1\leq Z\leq M_2}Z^{1+\epsilon}+\delta_{M_2\leq Z}M_2^{1+\epsilon}\left(\frac{Z}{M_2}\right)^{1/2}
\end{multline*}
for all $\epsilon>0$ and all real number $A>1$.
\end{corollary}

\begin{proof}[\proofname{} of corollary \ref{coro_useful}]%
Let us assume that $Z\geq M_1$ and $M_2=+\infty$. Then, Proposition
\ref{propo_bounds} tells us that
for all positive real number $A$, the $m$-sum is bounded by
\begin{equation*}
Z^{1/2}\sum_{M_1\leq m\leq Z}\frac{\left\vert \af(m)
\right\vert}{m^{1/2}}+Z^{A}\sum_{m>Z}\frac{\left\vert \af(m)
\right\vert}{m^{A}}.
\end{equation*}
By the Cauchy-Schwarz inequality, the first term is bounded by
\begin{align*}
& \ll Z^{1/2}\left(\sum_{M_1\leq m\leq Z}\left\vert \af(m)\right\vert^2\right)^{1/2}\left(\sum_{M_1\leq m\leq Z}\frac{1}{m}\right)^{1/2} \\
& \ll_{\epsilon, f}Z^{1/2}\left(Z^{1+\epsilon}\right)^{1/2} \\
& =Z^{1+\epsilon}
\end{align*}
for all $\epsilon>0$ by \eqref{eq_RP_average}. By summation by parts, the second term equals
\begin{equation*}
Z^{A}\left[\frac{1}{x^{A}}\sum_{1\leq m\leq x}\left\vert \af(m)\right\vert\right]_{x=Z}^{+\infty}+AZ^{A}\int_{x=Z}^{+\infty}\frac{1}{x^{A+1}}\sum_{1\leq m\leq x}\left\vert \af(m)\right\vert\dd x.
\end{equation*}
Choosing $A>1$, the Cauchy-Schwarz inequality and
\eqref{eq_RP_average} ensure that this quantity is also
$\ll Z^{1+\epsilon}$.
\par
Let us assume that $Z<M_1$ and $M_2=+\infty$. Similarly, the $m$-sum is bounded by
\begin{equation*}
M^{1+\epsilon}\left(\frac{Z}{M_1}\right)^{A}
\end{equation*}
by summation by parts and \eqref{eq_RP_average}.
\par
The argument in the case where $M_2$ is a real number are essentially
the same.
\end{proof}

\section{Equidistribution of products of hyper-Kloosterman sums}%
\label{sec-katz}

This section contains the crucial algebraic ingredient involved in the
determination of the asymptotic behaviour of certain combinations of
hyper-Kloosterman sums which will arise in Section~\ref{ssec-reduce-katz}.
\par
Let $k\geq 1$ be a positive integer, let $\uple{m}=(m_1,\ldots,m_k)$,
$\uple{n}=(n_1,\ldots, n_k)$ be two tuples of non-negative integers,
and let $\uple{c}=(c_1,\ldots,c_k)\in\left(\mathbb{F}_p^{\ast}\right)^k$ be given.
\par
We define
\begin{equation}\label{eq-sum-s}
  S_{\uple{m};\uple{n}}(\uple{c};p)\coloneqq\frac{1}{p}
  \sum_{a\in\mathbb{F}_p^{\ast}}
  \prod_{j=1}^kK_N(ac_j,p)^{n_j}K_N(-ac_j,p)^{m_j},
\end{equation}
where we recall that $K_N(x,p)$ denotes the normalized hyper-Kloosterman
sum defined in \eqref{eq_Kloos_def}. We will determine the behavior of these sums as $p$ tends to
infinity.
\par
For $\mathbf{G}$ either the special linear group $\mathrm{SL}_N$ or
the symplectic group $\mathrm{Sp}_{N}$ (if $N$ is even), we denote by
$$
\mathrm{Std}\,:\, \mathbf{G}\rightarrow \mathrm{GL}_N
$$
the \emph{standard} $N$-dimensional representation of $\mathbf{G}$.  When
$\mathbf{G}=\mathrm{SL}_N$, we denote by $\overline{\mathrm{Std}}$ the
contragredient of the standard representation.

\begin{theorem}\label{theo_katz}
Let $p$ be an \emph{odd} prime number, $k\geq 1$,
$\uple{c}=(c_1,\ldots,c_k)\in\left(\mathbb{F}_p^{\ast}\right)^k$ and let
$\uple{m}=(m_1,\ldots,m_k)$ and $\uple{n}=(n_1,\ldots, n_k)$ be two
tuples of non-negative integers.
\begin{itemize}
\item
If $N$ is odd and if the parameters $c_j$'s are distinct in $\mathbb{F}_p^{\ast}/\{\pm 1\}$ then
\begin{equation*}
  S_{\uple{m};\uple{n}}(\uple{c};p)=A_{\uple{m},\uple{n}}+O(p^{-1/2})
\end{equation*}
where the implied constant depends only on $(k,N,\uple{m},\uple{n})$
and where
\begin{equation*}
A_{\uple{m},\uple{n}}=\prod_{j=1}^kA_{m_j,n_j},
\end{equation*}
with $A_{m,n}\geq 0$ given by the multiplicity of the trivial
representation of $\;\mathrm{SL}_N$ in the tensor product
\begin{equation*}
\rho_{m,n}=\overline{\mathrm{Std}}^{\otimes m}\otimes
\mathrm{Std}^{\otimes n}
\end{equation*}
for all non-negative integers $m$ and $n$.
\item If $N$ is even and if the parameters $c_j$'s are distinct in
  $\mathbb{F}_p^{\ast}/\{\pm 1\}$, then
\begin{equation}\label{eq-formula-s}
  (-1)^s  S_{\uple{m};\uple{n}}(\uple{c};p)=B_{\uple{m},\uple{n}}+O(p^{-1/2})
\end{equation}
where 
\begin{equation*}
s=\sum_{1\leq j\leq k}(m_j+n_j),
\end{equation*}
the implied constant depends only on $(k,N,\uple{m},\uple{n})$,
and where
\begin{equation*}
  B_{\uple{m},\uple{n}}=\prod_{j=1}^kB_{m_j}B_{n_j},
\end{equation*}
with $B_{m}\geq 0$, for $m\geq 0$, given by the multiplicity of the
trivial representation of $\mathrm{Sp}_N$ in
$\rho_{m}=\mathrm{Std}^{\otimes m}$.
\end{itemize}
\end{theorem}


\begin{remark}
(1)  Note that the ``main terms'' $A_{\uple{m},\uple{n}}$ and
  $B_{\uple{m},\uple{n}}$ are independent of the tuple $c$ (with their
  respective restrictions). However, this independence is only
  meaningful when these main terms do not vanish.
\par
(2) Opening all the hyper-Kloosterman sums in
$S_{\uple{m};\uple{n}}(\uple{c};p)$, we can transform this sum into an
additive character sum in 
$$
1+(N-1)\sum_{j=1}^k(m_j+n_j)
$$
variables over $\mathbb{F}$. Comparing the normalization shows that
Theorem~\ref{theo_katz} is equivalent to \emph{uniform square-root
  cancellation over primes} for these sums whenever the main term
vanishes. 
\end{remark}

This is the analogue of~\cite[Proposition 3.2]{FoGaKoMi}, and proceeds
along similar lines. Let us decompose the proof in several steps. We
begin with a lemma.

\begin{lemma}\label{lm-generating}
  Let $\mathbb{F}$ be a finite field with $\abs{\mathbb{F}}=q$ elements, let $r\geq 1$
  be an integer, and let $a\in \mathbb{F}^{\ast}$.  We have
$$
\exp\Bigl(\sum_{\nu\geq 1}\frac{1}{\nu} \Bigl(
\sum_{x\in \mathbb{F}_{\nu}^{\ast}}{\overline{K_r(x,q^{\nu})}K_r(ax,q^{\nu})}
\Bigr)T^{\nu}\Bigr)=\frac{P_a(T)}{(1-T)\cdots (1-q^{r-1}T)}
$$
as a formal power series in $\C[[T]]$, where $\mathbb{F}_{\nu}$ denotes the
extension of degree $\nu$ of $\mathbb{F}$ and
$$
P_a(T)=\begin{cases}
1&\text{ if } a\not=1\\
1+qT&\text{ if } a=1.
\end{cases}
$$
\end{lemma}

\begin{proof}[\proofname{} of lemma \ref{lm-generating}]
  Let
$$
S_r(a,\mathbb{F})=\sum_{x\in \mathbb{F}^{\ast}}{\overline{K_r(x,q)}K_r(ax,q)}
$$
for $r\geq 1$.
\par
A straightforward application of the definition of Kloosterman sums
and of orthogonality of characters (see~\cite[p. 170]{MR955052} for a
similar computation) shows that, for $r\geq 2$, we have the relation
$$
S_r(a,\mathbb{F})=\overline{S_{r-1}(a,\mathbb{F})}-\frac{1}{q^{r-1}}.
$$
\par
Since it is clear that
$$
S_1(a,\mathbb{F})=\begin{cases}
q-1&\text{ if } a=1\\
-1&\text{ if } a\not=1,
\end{cases}
$$
we obtain, by induction on $r$ first, and then by replacing $\mathbb{F}$ by
$\mathbb{F}_{\nu}$ for $\nu\geq 1$, the formula
$$
\sum_{x\in \mathbb{F}_{\nu}^{\ast}}{\overline{K_r(x,q^{\nu})}K_r(ax,q^{\nu})}=
\begin{cases}
q^{\nu}-1-q^{-\nu}-\cdots-q^{\nu(r-1)}&\text{ if } a=1,\\
-1-q^{-\nu}-\cdots-q^{\nu(r-1)}&\text{ if } a\not=1.\\
\end{cases}
$$
\par
Summing over $\nu$ and taking the exponential, the result
follows. (One could also invoke the Plancherel formula for the
discrete Mellin transform, and the fact that the Mellin transforms of
hyper-Kloosterman sums are products of Gauss sums,
see~\cite[8.2.8,8.2.9]{MR1081536}).
\end{proof}

The next proposition is the key to Theorem \ref{theo_katz}.

\begin{proposition}\label{propo_arith_geo}
  Let $p$ be an \emph{odd} prime number, $k\geq 1$,
  $\uple{c}=(c_1,\ldots,c_k)\in\left(\mathbb{F}_p^{\ast}\right)^k$. Let
  $\ell\not=p$ be a prime number, and let $\mathcal{K}_N$ be the rank
  $N$ Kloosterman $\ell$-adic sheaf on the multiplicative group over
  $\mathbb{F}_p$.
\begin{itemize}
\item
If $N$ is odd and the parameters $c_j$'s are distinct in $\mathbb{F}_p^{\ast}/\{\pm 1\}$ then the arithmetic and geometric monodromy groups of the sheaf
\begin{equation*}
  \mathcal{F}(\uple{c})\coloneqq[\times c_1]^*\mathcal{K}_N\oplus\cdots\oplus 
  [\times c_k]^*\mathcal{K}_N
\end{equation*}
coincide and are equal to $\mathrm{SL}_N^k$ (the direct product of $k$ copies of $\mathrm{SL}_N$).
\item
If $N\geq 2$ is even and the parameters $c_j$'s are distinct in $\mathbb{F}_p^{\ast}/\{\pm 1\}$ then the arithmetic and geometric monodromy groups of the sheaf
\begin{equation*}
  \mathcal{G}(\uple{c})\coloneqq[\times c_1]^*\mathcal{K}_N\oplus\cdots\oplus [\times
  c_k]^*\mathcal{K}_N\oplus [\times
  (-c_1)]^*\mathcal{K}_N\oplus\cdots\oplus [\times
  (-c_k)]^*\mathcal{K}_N
\end{equation*}
coincide and are equal to $\mathrm{Sp}_N^{2k}$ (the direct product of $2k$ copies of $\mathrm{Sp}_N$).
\end{itemize}
\end{proposition}

\begin{proof}
  In both cases, we will apply the Goursat-Kolchin-Ribet
  criterion~\cite[Proposition 1.8.2]{MR1081536}, much as
  in~\cite{MR1345284}. 
\par
We consider first the case when $N$ is \emph{odd}. Then, for each
$1\leq j\leq k$, the geometric and arithmetic monodromy group of
$[\times c_j]^*\mathcal{K}_N$ coincide and are equal to
$\mathrm{SL}_N$ (as proved by N.~Katz~\cite[Theorem 11.1]{MR955052}). It
follows that there is a natural inclusion of the geometric and
arithmetic monodromy groups of $\mathcal{F}(\uple{c})$ in
$\mathrm{SL}_N^k$ in that case. We thus need to prove that this
inclusion is an isomorphism.
\par
The Goursat-Kolchin-Ribet criterion shows that this follows if there
does not exist a rank $1$ sheaf $\mathcal{L}$ such that either
\begin{equation}\label{eq-goursat}
[\times c_i]^*\mathcal{K}_N \simeq [\times
c_j]^*\check{\mathcal{K}}_N\otimes\mathcal{L}
\quad\text{ or }\quad
[\times c_i]^*\mathcal{K}_N \simeq [\times
c_j]^*\mathcal{K}_N\otimes\mathcal{L}
\end{equation}
for any $1\leq i\not=j\leq k$, where $\simeq$ denotes geometric
isomorphism and $\check{\mathcal{K}}_N$ is the dual of $\mathcal{K}_N$
(see~\cite[Proposition 1.8.2]{MR1081536} and~\cite[Example
1.8.1]{MR1081536}).
\par
We therefore assume that there exists a rank $1$ sheaf $\mathcal{L}$
satisfying \eqref{eq-goursat} for some $1\leq i\neq j\leq k$; we
will find a contradiction.
\par
If we have a geometric isomorphism
\begin{equation*}
[\times c_{i}]^*\mathcal{K}_N \simeq [\times
c_{j}]^*\check{\mathcal{K}}_N\otimes\mathcal{L}
\end{equation*}
then we also get an isomorphism
\begin{equation*}
[\times c_{i}]^*\mathcal{K}_N \simeq [\times
(-c_{j})]^*\mathcal{K}_N\otimes\mathcal{L}
\end{equation*}
since $\check{\mathcal{K}}_N\simeq [\times (-1)]^*\mathcal{K}_N$ for
$N$ odd.
\par
Hence the assumption implies that there is a geometric isomorphism
$$
[\times a]^*\mathcal{K}_N \simeq \mathcal{K}_N\otimes\mathcal{L}
$$
for some (possibly different) rank $1$ sheaf $\mathcal{L}$ with
$a=c_{i}/c_{j}$ or $a=-c_{i}/c_{j}$. 
\par
From this geometric isomorphism, as in~\cite[Lemma 2.4]{MR1345284}, it
would follow that $\mathcal{L}$ is tame at $\infty$ (because the
unique slope of the Kloosterman sheaf is $1/N<1$, whereas the unique
slope of the rank $1$ sheaf $\mathcal{L}$, if it were wildly ramified,
would be a positive integer). Tensoring with $\check{\mathcal{K}}_N$,
we deduce that an isomorphism as above implies an equality of Swan
conductors at infinity
$$
\mathrm{Swan}_{\infty}([\times a]^*\mathcal{K}_N\otimes\check{\mathcal{K}}_N)
=\mathrm{Swan}_{\infty}(\mathcal{K}_N\otimes\check{\mathcal{K}}_N),
$$
where the point is that $\mathcal{L}$ has disappeared because
tensoring with a tame sheaf leaves the Swan conductor unchanged.
\par
Again as in~\cite[Lemma 2.4]{MR1345284}, the Swan conductors are the
degrees of the corresponding zeta functions, as rational functions,
i.e., they are the degrees of
$$
\exp\Bigl(\sum_{\nu\geq 1}\frac{1}{\nu} \Bigl(
\sum_{x\in \mathbb{F}_{p^\nu}^{\ast}}{\overline{K_N(x,p^{\nu})}
K_N(ax,p^{\nu})}
\Bigr)T^{\nu}\Bigr)
$$
and
$$
\exp\Bigl(\sum_{\nu\geq 1}\frac{1}{\nu} \Bigl(
\sum_{x\in \mathbb{F}_{p^\nu}^{\ast}}{\overline{K_N(x,p^{\nu})}
K_N(x,p^{\nu})}
\Bigr)T^{\nu}\Bigr).
$$
\par
But Lemma~\ref{lm-generating} shows that these degrees differ except
if $a=1$, since the first is $N$ for $a\not=1$, and the second is
$N+1$. Thus, we get $a=1$, and hence $c_{i}=\pm c_{j}$, a
contradiction to our assumption on $\uple{c}$ that concludes the case
where $N$ is odd.
\par
Let us now assume that $N\geq 2$ is \emph{even}. We denote
$c_i=-c_{i-k}$ for $k+1\leq i\leq 2k$. Again
N.~Katz~\cite[Th. 11.1]{MR955052} has show that, for each $1\leq j\leq
2k$, the geometric and arithmetic monodromy groups of $[\times
c_j]^*\mathcal{K}_N$ coincide and are equal to $\mathrm{Sp}_N$, and it
follows that there is a natural inclusion of the geometric and
arithmetic monodromy groups of $\mathcal{G}(\uple{c})$ in
$\mathrm{Sp}_N^{2k}$. To prove that this is an isomorphism using the
Goursat-Kolchin-Ribet criterion, we need to show that there does not
exist a rank $1$ sheaf $\mathcal{L}$ and a geometric isomorphism
$$
[\times c_{i}]^*\mathcal{K}_N \simeq [\times
c_{j}]^*\check{\mathcal{K}}_N\otimes\mathcal{L} \quad\text{ or
}\quad[\times c_{i}]^*\mathcal{K}_N \simeq [\times
c_{j}]^*\mathcal{K}_N\otimes\mathcal{L}
$$
for some $1\leq i\not=j\leq 2k$. Since
$\check{\mathcal{K}_N}\simeq\mathcal{K}_N$ for $N$ even (the
arithmetic monodromy group being self-dual), this reduces to checking
that we can not have
$$
 [\times c_i]^*\mathcal{K}_N \simeq [\times
  c_j]^*\mathcal{K}_N\otimes\mathcal{L} \quad\text{ or }\quad
  [\times c_i]^*\mathcal{K}_N \simeq [\times
  (-c_j)]^*\mathcal{K}_N\otimes\mathcal{L}
$$
for $1\leq i\not=j\leq k$.  But this follows by the same reasoning as
for $N$ odd, taking advantage of the fact that the $c_i$ are distinct
modulo $\pm 1$.
\end{proof}

Now, we can get back to the proof of Theorem \ref{theo_katz}.

\begin{proof}[\proofname{} of theorem \ref{theo_katz}]
  We only consider the case when $N$ is \emph{odd}, since the proof is
  similar for $N$ even, using the second part of
  Proposition~\ref{propo_arith_geo} instead of the first.
\par
The point is that, for some isomorphism
$\iota\,:\,\bar{\mathbb{Q}}_{\ell}\simeq \C$, we have
$$
S_{\uple{m};\uple{n}}(\uple{c};p)=\sum_{a\in\mathbb{F}_p^{\ast}}
\iota(\mathrm{Tr}(\mathrm{Frob}_{a,p}\mid
\rho_{\uple{m},\uple{n}}(\mathcal{F}(\uple{c})))),
$$ 
where 
$$
\rho_{\uple{m},\uple{n}}=\bigboxtimes_{j=1}^k \rho_{m_j,n_j}
$$ 
is a representation of the arithmetic monodromy group of
$\mathcal{F}(\uple{c})$, and $\mathrm{Frob}_{a,p}$ is the geometric
Frobenius conjugacy class at $a$ relative to $\mathbb{F}_p$. Indeed,
this follows immediately from the definition of the Kloosterman
sheaves, which implies that, for a suitable $\iota$, we have
$$
\iota(\mathrm{Tr}(\mathrm{Frob}_{a,p}\mid
\mathcal{K}_N))=(-1)^{N-1}K_N(a;p).
$$
\par
By Proposition \ref{propo_arith_geo}, the arithmetic and the geometric
monodromy group of $\mathcal{F}(\uple{c})$ coincide and are equal to
$\mathrm{SL}_N^k$. Thus, Katz's effective version of the Deligne
equidistribution theorem for curves (see~\cite[Section 3.6]{MR955052})
shows that 
$$
S_{\uple{m};\uple{n}}(\uple{c};p)=\mu+O(p^{-1/2})
$$
where $\mu$ is the multiplicity of the trivial representation of the
geometric monodromy group in the representation
$\rho_{\uple{m},\uple{n}}$, where the implied constant depends only on
$k$, $N$ and $\uple{m}$, $\uple{n}$ (the crucial property of
independence of the implied constant on $p$ arises from the fact that,
for $p$ varying, the sheaf $\mathcal{F}(\uple{c})$ always has the same
rank, number of singularities and Swan conductors).
\end{proof}


It is now essential to determine when the leading terms
$A_{\uple{m},\uple{n}}$ and $B_{\uple{m},\uple{n}}$ are non-zero. This
happens in very special configurations only.

\begin{proposition}\label{lemma_multiplicity} Let $N\geq 2$. \\
\emph{(1)} We have
$$
A_{0,0}=1,\quad\quad B_0=1
$$
for $N$ odd or $N$ even, respectively.
\par
\noindent\emph{(2)} Let $m$ and $n$ be non-negative integers with
$(m,n)\neq(0,0)$.
\begin{itemize}
\item For $N$ odd, $A_{1,1}=1$ and $A_{m,n}\geq 1$ if and only if $N$
  divides $m-n$.
\item For $N$ even, $B_2=1$ and $B_m\geq 1$ if and only if $2$ divides $m$.
\end{itemize}
\end{proposition}
\begin{proof}[\proofname{} of proposition \ref{lemma_multiplicity}]
  The first point is clear since $\rho_{0,0}$ (resp. $\rho_0$) is the
  one-dimensional trivial representation.
\par
We come to the second point, first when $N$ is odd.
\par
Then $A_{1,1}$ is the multiplicity of the trivial representation of
$\mathrm{SL}_N$ in $\overline{\mathrm{Std}}\otimes \mathrm{Std}\simeq
\mathrm{End}(\mathrm{Std})$. Since $\mathrm{Std}$ is an irreducible
representation of $\mathrm{SL}_N$, Schur's Lemma implies that
$A_{1,1}=1$.
\par
We now consider the action of the center of $\mathrm{SL}_{N}$ on
$\rho_{m,n}$. This group is isomorphic to the cyclic group of $N$-th
roots of unity. Since a generator $\xi$ of this group acts on
$\mathrm{Std}$ by multiplication by $\xi$, and on the contragredient
by multiplication by $\xi^{-1}$, we see that $\xi$ acts on
$\rho_{m,n}$ by multiplication by $\xi^{n-m}$. But the action of the
center must also be trivial on any subrepresentation, and therefore
$\xi^{n-m}=1$ if $A_{m,n}\geq 1$, i.e., $m\equiv n\bmod{N}$ whenever
$A_{m,n}\geq 1$.
\par
Conversely, assume $N\mid m-n$. We can assume (up to exchanging
$(m,n)$ with $(n,m)$, which we can since $A_{m,n}=A_{n,m}$, simply
because the contragredient of $\rho_{m,n}$ is $\rho_{n,m}$) that
$n\geq m$, say $n=m+qN$ with $q\geq 0$. Then
$$
\rho_{m,n}\simeq \mathrm{End}(\mathrm{Std})^{\otimes m}\otimes
\mathrm{Std}^{\otimes qN}.
$$
\par
The first tensor factor always contains the trivial representation,
and therefore it is enough to show that the second does for any $q\geq
0$. By writing
$$
\mathrm{Std}^{\otimes Nq}=(\mathrm{Std}^{\otimes N})^{\otimes q},
$$
we then reduce to the case of $\mathrm{Std}^{\otimes N}$. But this
representation contains the trivial representation, as one can most
easily see by considering the contragredient, which acts on the space
of $N$-multilinear forms on $\C^N$, and contains the space of
antisymmetric $N$-linear forms on $\C^N$, in which the determinant is
a non-trivial invariant vector for the action of $\mathrm{SL}_N$.
\par
Consider finally the case when $N$ is even. Since $\mathrm{Std}$ is
then self-dual, we have $B_2=1$ again by Schur's Lemma. The
center of $\mathrm{Sp}_N$ contains $-1$, and considering its action
shows that $2\mid n$ if $B_n\geq 1$. Finally, if $2\mid n$, we
see as above that $B_n\geq 1$ because $B_2\geq 1$ (which
may be interpreted by the existence of the invariant alternating
bilinear form on the standard representation of the symplectic group.)
\end{proof}

\begin{remark}
In particular, note that if $N$ is even and
$$
B_{\uple{m},\uple{n}}\not=0,
$$
then $s=\sum_{1\leq j\leq k}(m_j+n_j)$ is even, and therefore the
formula~(\ref{eq-formula-s}) becomes
$$
S_{\uple{m};\uple{n}}(\uple{c};p)=B_{\uple{m},\uple{n}}
+O(p^{-1/2}).
$$
\end{remark}
\begin{remark}
  For $N=3$, G.~Djankovi\'{c} (see \cite{Dj}) has computed the first
  few moments of hyper-Kloosterman sums and found that
\begin{gather*}
  S_{0;1}(1;p)  =  -\frac{1}{p^2},\quad
  S_{0;2}(1;p)  =  -\frac{1}{p}-\frac{1}{p^2}-\frac{1}{p^3},\quad
  S_{1;1}(1;p)  = 1-\frac{1}{p}-\frac{1}{p^2}-\frac{1}{p^3},\\
  S_{0;3}(1;p)  =  1-\left(1+\left(\frac{-3}{p}\right)\right)\frac{1}{p}-\frac{3}{p^2}-\frac{2}{p^3}-\frac{1}{p^4},
\end{gather*}
for all odd prime numbers $p$. He also proved elementarily the
upper-bound
\begin{equation*}
S_{0;4}(1;p)\ll\frac{1}{\sqrt{p}}
\end{equation*}
(already known due to the results of N.~Katz). Of course, these results
are compatible with Theorem \ref{theo_katz} and Proposition
\ref{lemma_multiplicity}.
\par
G.~Djankovi\'{c} observed that ``curiously there is no cancellation in
the sum'' $S_{0;3}(1;p)$. But Proposition \ref{lemma_multiplicity}
explains this feature, simply by the fact that the trivial
representation occurs in $\mathrm{Std}^{\otimes 3}$.
\par
We also note that D.~W\"ursch, in his (unpublished) 2011 Master Thesis
at ETH Zürich, computed $S_{(2,2)}(1;p)$ for $N=3$ in terms of the
number of points on a certain elliptic surface.
\end{remark}

\begin{remark}
One can use character theory and explicit descriptions of the Haar
measure on the relevant maximal compact subgroups of the monodromy
groups to give ``concrete'' integral formulas for $A_{m,n}$ and
$B_m$. Since we will not use such descriptions, we omit the details.
\end{remark}

\section{Asymptotic of sums related to the variance}\label{sec_variance}%
In this section, we find the asymptotic behavior of certain sums, which will allow us to finalize the proof of our main results, by identifying the \emph{main terms} with data depending on the input cusp form $f$ and test function $w$. The proof may be skipped in a first reading. As before, $f$ is a cusp form on $GL(N)$ with level $1$ and $w$ is smooth and compactly supported on $\R_+^\ast$.
\par
For $g\in\{f,f^\ast\}$, $Y, Z$ some positive real numbers and $\mathrm{B}$ a smooth function on $\R$, we consider the sum
\newcommand{\ags}{a_{g^{\ast}}}
\newcommand{\ag}{\uple{a}_g}
\begin{equation*}
  V_{(f,g)}(Y,Z)\coloneqq\frac{1}{Y}\sum_{1\leq m<Z}\af(m)\ag(m)\mathrm{B}\left(\frac{m}{Y}\right).
\end{equation*}
We will use Rankin-Selberg theory to derive the following asymptotic
expansion of such sums. Because the result might be applicable in
other contexts, we include a parameter in the statement measuring the
approximation to the Ramanujan-Petersson conjecture at finite places;
in our case, taking $\theta=1/2-1/(N^2+1)$ is possible by the work of
Luo, Rudnick and Sarnak~\cite{MR1334872,MR1703764}.

\begin{proposition}\label{propo_variance}
Let $\theta\in ]0,1/2[$ be a real number such that the Satake
parameters of $f$ satisfy
$$
\abs{\alpha_{j,q}(f)}\leq q^{\theta}
$$
for all primes $q$ and $1\leq j\leq N$.
\par
Let $0<Y<Z$ be real numbers. If the function $\mathrm{B}$ satisfies
the bounds
\begin{equation}\label{eq_bound_petit}
\mathrm{B}(x)\ll x^{-\eta}\text{ for } 0<x<1,
\end{equation}
for some $0\leq\eta<1$ and
\begin{equation}\label{eq_bound_grand}
\mathrm{B}(x)\ll_A x^{-A}\text{ for } x>0,
\end{equation}
for all positive real numbers $A$, then we have
\begin{equation}\label{eq_before_unitary}
V_{(f,g)}(Y,Z)=\left(\delta_{\substack{f^\ast\neq f \\
g=f^\ast}}+\delta_{f^\ast=f}\right)r_fH_{f,f^\ast}(1)\mathcal{M}\left[\mathrm{B}\right](1) \\
+O_{\epsilon, f}\left(Z^{\epsilon}\left(\frac{Y}{Z}\right)^{A}+Y^{-1/2+\theta+\epsilon}\right)
\end{equation}
for all $A>0$, where $r_f$ is the residue at $s=1$ of the Rankin-Selberg $L$-function $L(f\times f^\ast,s)$ and
\begin{equation*}
H_{f,f^\ast}(1)=\prod_{q\in\prem}P_N(\alpha_q(f^\ast),\alpha_q(f),q),
\end{equation*}
in terms of the polynomials $P_N(\uple{x},\uple{y},T)$, which are
defined by~(\ref{eq-schur}). Furthermore, we have
$$
H_{f,f^\ast}(1)>0.
$$
\end{proposition}


\begin{remark}
We illustrate here the special cases $N=2$ and $N=3$:
\par
(1) If $N=2$ then
\begin{equation*}
H_{f,f^\ast}(1)=\frac{6}{\pi^2}>0.
\end{equation*}
by Remark \ref{rem_N=2_N=3}.
\par
(2) On the other hand, if $N=3$, then we have
\begin{align*}
H_{f,f^\ast}(1) & =\prod_{q\in\mathcal{P}}\left(1-\frac{\abs{a_f(q,1)}^2}{q^2}+\frac{2a_f(q,q)}{q^3}-\frac{\abs{a_f(q,1)}^2}{q^4}+\frac{1}{q^6}\right) \\
& =\prod_{q\in\mathcal{P}}\left(1-\frac{\abs{a_f(q,1)}^2}{q^2}+\frac{2\left(\abs{a_f(q,1)}^2-1\right)}{q^3}-\frac{\abs{a_f(q,1)}^2}{q^4}+\frac{1}{q^6}\right) \\
& =\prod_{q\in\mathcal{P}}\left(1-\frac{1}{q}\right)^2\left(1+\frac{1+\abs{a_f(q,1)}}{q}+\frac{1}{q^2}\right)\left(1+\frac{1-\abs{a_f(q,1)}}{q}+\frac{1}{q^2}\right)
\end{align*}
by Remark \ref{rem_N=2_N=3}, \eqref{eq_alpha_sym}, \eqref{eq_alpha}
and the formula
\begin{equation*}
  S_{1,1}(x_1,x_2,x_3)=(x_1+x_2)(x_1+x_3)(x_2+x_3)=\sum_{1\leq j_1\neq j_2\leq 3}x_{j_1}x_{j_2}^2+2e_3(x_1,x_2,x_3)
\end{equation*}
where $e_3$ is defined in \eqref{eq_alpha_sym}.
\par
In particular, one can see immediately that $H_{f,f^\ast}(1)>0$, using
the fact that Satake parameters $GL(3)$ cusp forms are bounded by
$q^{1/2-1/10}$.
\end{remark}

\begin{proof}[\proofname{} of proposition \ref{propo_variance}]%
By the Cauchy-Schwarz inequality, summation by parts, \eqref{eq_RP_average} and \eqref{eq_bound_grand}, one gets
\begin{equation}\label{eq_cleaning}
V_{(f,g)}(Y,Z)=V_{(f,g)}^{0}(Y,Z)+O_{\epsilon, f}\left(Z^{\epsilon}\left(\frac{Y}{Z}\right)^{A-1}\right)
\end{equation}
for all $A>1$ since $Y>Z$ and where
$$
V_{(f,g)}^{0}(Y)  
\coloneqq\frac{1}{Y}\sum_{m\geq 1}\af(m)\ag(m)\mathrm{B}\left(\frac{m}{Y}\right)
$$
is the extension to the sum over all the positive integers $m$.
\par
Using Mellin inversion, we obtain
$$
V_{(f,g)}^{0}(Y)
=\frac{1}{Y}\frac{1}{2i\pi}\int_{(3)}D_{f,g}(s)Y^s\mathcal{M}\left[\mathrm{B}\right](s)\dd s \label{eq_before_shift}
$$
where the Dirichlet series
\begin{equation*}
D_{f,g}(s)\coloneqq\sum_{m\geq 1}\frac{\af(m)\ag(m)}{m^s}
\end{equation*}
is absolutely convergent on $\Re{(s)}>1$, by \eqref{eq_RP_average}, and defines a holomorphic function on this half-plane. 
\par
In addition, since $m\mapsto \af(m)\ag(m)$ is a multiplicative function by \eqref{eq_multiplicative_bis}, we have an
Euler product expansion given by
\begin{equation*}
D_{f,g}(s)=\prod_{q\in\mathcal{P}}\sum_{k\geq 0}\frac{\af(q^k)\ag(q^k)}{q^{ks}}\coloneqq\prod_{q\in\mathcal{P}}D_{f,g,q}(s).
\end{equation*}
\par
By \eqref{eq_Fourier_schur}, \eqref{eq_alphap}, \eqref{eq_Fourier_dual} and \eqref{eq-schur}, we have the formula
\begin{equation*}
D_{f,g,q}(s)=\frac{P_N(\alpha_q(f),\alpha_q(g),q^{-s})}{\prod_{1\leq j,k\leq 3}\left(1-\alpha_{j,q}(f)\alpha_{k,q}(g)q^{-s}\right)}
\end{equation*}
for any prime number $q$. As a consequence, the quotient
\begin{equation*}
\frac{D_{f,g}(s)}{L(f\times g,s)}=\prod_{q\in\prem}P_N(\alpha_q(f),\alpha_q(g),q^{-s})\coloneqq H_{f,g}(s)
\end{equation*}
defines a holomorphic function on $\Re{(s)}>1/2+\theta$.
\par
Moreover, the Mellin transform of $\mathrm{B}$ is holomorphic on $\Re{(s)}>\eta$ since
\begin{equation*}
\mathcal{M}\left[\mathrm{B}\right](s)x^{s-1}\ll\begin{cases}
x^{-\eta+\Re{(s)}-1} & \text{when $x$ is close to $0^+$,} \\
x^{-A+\Re{(s)}-1} & \text{when $x$ is close to $+\infty$}
\end{cases}
\end{equation*}
for all $A>0$.
\par
Going back to the integral formula \eqref{eq_before_shift} for
$V_{(f,g)}^{0}(Y)$, we can shift the integral to the line
$$
\Re{(s)}=\max{(1/2+\theta,\eta)}+\epsilon<1
$$
(the assumptions $\theta<1/2$ and $\eta<1$ are crucial here).  Using
the properties of the Rankin-Selberg $L$-function, we see that we
encounter at most a simple pole at $s=1$, and that the latter exists
if and only if
\begin{equation*}
\left(f^\ast\neq f \text{ and } g=f^\ast\right) \text{ or }
\left(f^\ast=f\right)
\end{equation*}
(recall that $g$ is either $f$ or $f^{\ast}$).
\par
The residue at $s=1$, in case there is a pole, is equal to
$$
r_fH_{f,f^{\ast}}(1)\mathcal{M}\left[\mathrm{B}\right](1)
$$
where $r_f$ is the residue at $s=1$ of the Rankin-Selberg $L$-function $L(f\times f^\ast,s)$.
\par
Hence, we have
\begin{equation}\label{eq_after_shift}
V_{(f,g)}^{0}(Y,Z)=
\left(\delta_{\substack{f^\ast\neq f \\
g=f^\ast}}+
\delta_{f^\ast=f}\right)r_fH_{f,f^\ast}(1)\mathcal{M}\left[\mathrm{B}\right](1)+O\left(Y^{-1/2+\theta+\epsilon}\right),
\end{equation}
and~(\ref{eq_before_unitary}) follows from \eqref{eq_cleaning} and
\eqref{eq_after_shift}.
\par
Finally, the positivity property $H_{f,f^\ast}(1)>0$ holds since
$H_{f,f^{\ast}}(1)$ is an absolutely convergent Euler product, and
each term is positive by Proposition~\ref{propo_schur} below, since
the assumption \eqref{eq_constraint_t} is satisfied in view
of \eqref{eq_RP}.
\end{proof}

\section{Applying the Vorono\u{\i} formula}\label{sec_6}%

We continue with a fixed Hecke-Maass cusp form $f$ on $GL(N)$ of level
$1$. Since $f$ is fixed, we will denote
$\uple{\alpha}=\alpha_{\infty}(f)$.
\par
We recall the definition~(\ref{eq-def-ef}) of the error terms
$E_f(X,p,a)$, for an invertible residue $a$ class in
$\mathbb{F}_p^\times$, which depend on the choice of a fixed text
function $w:\R_+^\ast\to\R$, which is assumed to be non-zero, smooth
and compactly supported on $[x_0,x_1]\subset\R_+^\ast$. To simplify
notation, we denote 
\newcommand{\bfn}{\mathcal{B}_{\uple{\alpha}}}
\newcommand{\bfns}{\mathcal{B}_{\uple{\alpha}^{\ast}}}
$$
\bfn(x)=\mathcal{B}_{\uple{\alpha}}[w](x).
$$
\par
In this section, we perform the first steps of the analysis of
these sums before computing their moments.

\begin{proposition}\label{propo_voronoi}
Let $\uple{\alpha}=\alpha_{\infty}(f)$ for some cusp form $f$
as before. If $a$ is an invertible residue class in $\mathbb{F}_p^\times$ then
\begin{multline}\label{eq_equality_voronoi}
E_f(X,p,a)=
\frac{\epsilon_f}{\sqrt{p^N/X}}\sum_{m\in\Z^\ast}\afs(m)
K_N(-am,p)\bfn\left(\frac{m}{p^N/X}\right) \\ 
+\epsilon_f\sum_{\ell=1}^{N-2}\frac{(-1)^\ell}{p^{(\ell+1)/2}\sqrt{p^\ell/X}}
\sum_{m\in\Z^\ast}a_f(\overbrace{1,\dots,1}^{\ell-1},p,1,\dots,1,m)
\bfn\left(\frac{m}{p^\ell/X}\right).
\end{multline}
\par
In particular,
\begin{equation}\label{eq_equality_voronoi_2}
  E_f(X,p,a)=\frac{\epsilon_f}{\sqrt{p^N/X}}
  \sum_{m\in\Z^\ast}\afs(m)
  K_N(-am,p)\bfn\left(\frac{m}{p^N/X}\right)+
  O_{\epsilon, f}\left(\frac{p^\epsilon}{\sqrt{p}}\right)
\end{equation}
and hence we have
\begin{equation}\label{eq_estim_voro}
E_f(X,p,a)\ll_{\epsilon, f}\left(\frac{p^N}{X}\right)^{1/2+\epsilon}
\end{equation}
for all $\epsilon>0$.
\end{proposition}


\begin{remark}
Note that the normalised hyper-Kloosterman sum $K_N(u,p)$ is a real number if $N$ is even and a complex number if $N$ is odd, whose complex conjugate is $K_N(-u,p)$. Hence, in all cases, we have
\begin{equation}\label{eq_complex_K3}
\overline{K_N(u,p)}=K_N\left((-1)^Nu,p\right).
\end{equation}
If $f$ is self-dual then the left-hand side of \ref{eq_equality_voronoi} is obviously a real number by \eqref{eq_Fourier_coeffs}. One can check directly that each $m$-sum in the right-hand side is a real number too by \eqref{eq_complex_K3} and by Lemma \ref{lemma_easy_Bessel}.
\par
To get the previous proposition, we will use the fact that P.~Deligne proved in \cite{MR0463174} that this normalised
hyper-Kloosterman sum satisfies
\begin{equation}\label{eq_bound_K3}
\abs{K_N(u,p)}\leq N.
\end{equation}
\end{remark}


\begin{proof}[\proofname{} of proposition \ref{propo_voronoi}]%
  Using additive characters to detect the congruence class $a$ modulo
  $p$, and isolating the contribution of the trivial character, we
  have
\begin{align*}
  S_f(X,p,a) & =\frac{1}{p}\sum_{b\bmod{p}}e\left(-\frac{ab}{p}\right)
  \sum_{n\geq 1}\af(n)w\left(\frac{n}{X}\right)e\left(\frac{bn}{p}\right) \\
  & =M_f(X,p)+\frac{1}{p}\sums_{b\bmod{p}}
  e\left(-\frac{ab}{p}\right)\sum_{n\geq
    1}\af(n)e\left(\frac{bn}{p}\right)w\left(\frac{n}{X}\right),
\end{align*} 
where $\sums$ restricts the sum to invertible residue classes.
\par
The Vorono\u{\i} summation formula (Proposition \ref{propo_formule_voronoi}) may be
applied to each sum over $n$, with $w_X(x)= w(x/X)$. In this case, we
have
\begin{equation*}
\mathcal{B}_{k,\uple{\alpha}}[w_X](x)=X\mathcal{B}_{k,\uple{\alpha}}[w](Xx)
\end{equation*}
for $k\in\{0,1\}$. This leads to
\begin{multline*}
S_f(X,p,a)=M_f(X,p) \\
+\epsilon_f\frac{X}{p}\sum_{\ell=1}^{N-2}\frac{(-1)^{\ell+1}}{p^\ell}\sum_{m\in\Z^\ast}a_f(\overbrace{1,\dots,1}^{\ell-1},p,1,\ldots,1,m)
\bfn\left(\frac{mX}{p^\ell}\right)
\sums_{b\bmod{p}}e\left(-\frac{ab}{p}\right) \\
+\epsilon_f\frac{X}{p^{N/2+1}}\sum_{m\in\Z^\ast}\afs(m)
\bfn\left(\frac{mX}{p^N}\right)\sums_{b\bmod{p}}
K_{N-1}(\bar{b}m,p)e\left(-\frac{ab}{p}\right).
\end{multline*}
\par
In the second term, the sum over $b$ is a Ramanujan sum, equal to
$-1$. In the last term, the sum over $b$ is easily computed: we have
\begin{equation*}
  \sums_{b\bmod{p}}K_{N-1}(\bar{b}m,p)e\left(-\frac{ab}{p}\right)=p^{1/2}K_N(-am,p).\end{equation*}
\par
We therefore deduce
\begin{multline*}
S_f(X,p,a)=M_f(X,p)+\epsilon_f\frac{X}{p}\sum_{\ell=1}^{N-2}\frac{(-1)^{\ell}}{p^\ell}\sum_{m\in\Z^\ast}a_f(\overbrace{1,\dots,1}^{\ell-1},p,1,\dots,1,m)
\bfn\left(\frac{mX}{p^\ell}\right) \\
+\epsilon_f\frac{X}{p^{(N+1)/2}}\sum_{m\in\Z^\ast}\afs(m)K_N(-am,p)
\bfn\left(\frac{mX}{p^N}\right),
\end{multline*}
which is \eqref{eq_equality_voronoi}. 
\par
Furthermore, for $N\geq 3$ and $1\leq\ell\leq N-2$,
\eqref{eq_multiplicative_bis} tells us that if we write
$m=p^km^\prime$ with $(p,m^\prime)=1$ and $k\geq 0$, then we have
\begin{equation*}
a_f(\overbrace{1,\dots,1}^{\ell-1},p,1,\ldots,1,m)=
a_f(\overbrace{1,\dots,1}^{\ell-1},p,1,\dots,1,p^k)a_f(1,\dots,1,m^\prime).
\end{equation*}
\par
As a consequence, the second term in \eqref{eq_equality_voronoi} is
\begin{equation*}
\epsilon_f\sum_{\ell=1}^{N-2}\frac{(-1)^\ell}{p^{(\ell+1)/2}\sqrt{p^\ell/X}}B_{\ell}
\end{equation*}
where
\begin{equation*}
B_{\ell}\coloneqq\frac{1}{2i\pi}\int_{(\sigma)}F_\ell(s)\frac{L(f^\ast,s)}{L_p(f^\ast,s)}\left(\frac{p^\ell}{X}\right)^s\mathcal{M}\left[\mathcal{B}_{\uple{\alpha}}^+[w]+\epsilon\mathcal{B}_{\uple{\alpha}}^-[w]\right](s)\dd s
\end{equation*}
by the Mellin inversion formula, where $L(f^\ast,s)$ is the Godement-Jacquet
$L$-function of $f^\ast$ (see \cite[Definition 9.4.3]{MR2254662}),
with $p$-factor given by
\begin{equation*}
L_p(f^\ast,s)=\sum_{k\geq 0}\frac{a_f(1,\dots,1,p^k)}{p^{ks}}=
\prod_{j=1}^N\left(1-\frac{\alpha_{j,p}(f^\ast)}{p^s}\right)^{-1}
\end{equation*}
(see~\cite[Equation 9.4.2]{MR2254662})
and
$$
F_\ell(s)= \sum_{k\geq
  0}\frac{a_f(\overbrace{1,\dots,1}^{\ell-1},p,1,\dots,1,p^k)}{p^{ks}}.
$$
\par
By Lemma \ref{lemma_Dirichlet} below, we can shift the contour to the
line $\Re{(s)}=1/2+\epsilon$ for any $\epsilon>0$ without encountering
any pole. This gives the bound
\begin{equation*}
B_{\ell}\ll_\epsilon p^{\ell/2+\epsilon}\left(\frac{p^\ell}{X}\right)^{1/2+\epsilon},
\end{equation*}
which proves \eqref{eq_equality_voronoi_2}. 
\par
More directly, the first term in \eqref{eq_equality_voronoi} is
bounded by
\begin{equation*}
\ll\delta_{p^N<X}\left(\frac{p^N}{X}\right)^{A-1/2}+\delta_{p^N\geq X}\left(\frac{p^N}{X}\right)^{1/2+\epsilon}\ll \left(\frac{p^N}{X}\right)^{1/2+\epsilon}
\end{equation*}
by Corollary \ref{coro_useful} and \eqref{eq_bound_K3}, which is \eqref{eq_estim_voro}.
\end{proof}

We used the following lemma:

\begin{lemma}\label{lemma_Dirichlet}
Let $1\leq\ell\leq N-2$ for $N\geq 3$. The series
\begin{equation*}
  F_\ell(s)=
  \sum_{k\geq 0}\frac{a_f(\overbrace{1,\dots,1}^{\ell-1},p,1,\dots,1,p^k)}{p^{ks}}
\end{equation*}
defines a holomorphic function on $\Re{(s)}\geq1/2+\epsilon$ for any
$\epsilon>0$, which satisfies
$$
F_\ell(s)\ll p^{\ell/2+\epsilon}
$$
for $\Re{(s)}=1/2+\epsilon$.
\end{lemma}

\begin{proof}[\proofname{} of proposition \ref{lemma_Dirichlet}]%
  We prove this lemma by induction on $\ell$. If $\ell=1$ then
\begin{equation*}
F_1(s)=\left(a_f(p,1,\dots,1)-\frac{1}{p^s}\right)L_p(f^\ast,s)
\end{equation*}
since
\begin{equation*}
a_f(1,\dots,1,p^k)a_f(p,1,\dots,1)=a_f(p,1,\dots,1,p^k)+a_f(1,\dots,1,p^{k-1})
\end{equation*}
by \eqref{eq_multiplicative} for all positive integer $k$. The result
follows from the Jacquet-Shalika bound
$$
|\alpha_{j,q}(f)|\leq q^{1/2}
$$
and \eqref{eq_alpha_sym}.
\par
If $2\leq\ell\leq N-2$ then
\begin{equation*}
  F_\ell(s)=a_f(\overbrace{1,\dots,1}^{\ell -1},p,1,\dots,1)
    L_p(f^\ast,s)+\frac{1}{p^s}F_{\ell-1}(s)
\end{equation*}
since
\begin{multline*}
  a_f(1,\dots,1,p^k)a_f(\overbrace{1,\dots,1}^{\ell-1},p,1,\dots,1)=
  a_f(\overbrace{1,\dots,1}^{\ell-1},p,1,\dots,1,p^k) \\
  +a_f(\overbrace{1,\dots,1}^{\ell-1},p,1,\dots,1,p^{k-1})
\end{multline*}
for all positive integers $k$ by \eqref{eq_multiplicative}. Once
again, the result follows from the Jacquet-Shalika bound and
\eqref{eq_alpha_sym}.
\end{proof}

\section{Asymptotic expansion of the mixed moments}\label{sec_7}%

This long section is the heart of the paper, since we will prove
Theorem~\ref{theo_moments}. Before we begin the proof, we explain how
the computation can be understood in probabilistic terms, in analogy
with Lindeberg's proof of the usual Central Limit Theorem for
triangular arrays of random variables using the method of moments.

\subsection{Notation}

We now come back to Theorem \ref{theo_moments}, and begin by recalling
and fixing some notation. Thus $f$ is a fixed cusp form of level $1$
on $GL(N)$, and $w$ is a compactly supported smooth test function. We
denote
$$
\uple{\alpha}=\alpha_{\infty}(f),\quad\quad
\uple{\alpha}^{\ast}=\alpha_{\infty}(f^{\ast})=
\overline{\uple{\alpha}}
$$
(by~(\ref{eq_unitaricity_2}), and
$$
\bfn(x)=\mathcal{B}_{\uple{\alpha}}[w](x),\quad\quad
\bfns(x)=\mathcal{B}_{\uple{\alpha}^{\ast}}[w](x).
$$
\par
We consider the mixed moment
$\mathsf{M}=\mathsf{M}_f(X,p,(\kappa,\lambda))$ for fixed non-negative
integers $\kappa$ and $\lambda$ and an odd prime $p$.
\par
The following notation will also be used througout this section.  We
will denote $\nu=\kappa+\lambda$ and $P=(p-1)/2$. By
$\uple{m}=(m_1\dots,m_{\nu})$, we will always denote a $\nu$-tuple of
non-zero integers, by $\uple{j}=(j_1,\dots,j_{\nu})$ a $\nu$-tuple of
integers in $\{1,\dots,P\}$, and by $\uple{e}=(e_1\dots,e_{\nu})$ a
$\nu$-tuple of elements in $\{\pm 1\}$.

\subsection{Probabilistic analogy}

\newcommand{\cE}{\mathcal{E}}
\newcommand{\ctE}{\widetilde{\mathcal{E}}}

For simplicity, we denote by $E_p$ the random variable $a\mapsto
E_f(X;p,a)$. We can then interpret the Vorono\u{\i} summation formula
as giving an approximate decomposition
$$
E_p=\sum_{m\in\Z^{\ast}}T_{p,m}+O(p^{-1/2+\epsilon})
$$
for any $\epsilon>0$, where $T_{p,m}$ is also viewed as a random
variable given by
$$
T_{p,m}=\frac{\epsilon_f}{\sqrt{p^N/X}}a_{f^{\ast}}(m)
\bfn\left(\frac{m}{p^N/X}\right) K_{p,m}
$$
with $K_{p,m}(a)=K_N(-am;p)$. It is easy to restrict the sum to $1\leq
|m|<p/2$ (using \eqref{eq_bound_K3} and Corollary \ref{coro_useful}),
getting a random  variable
$$
\cE_p=\sum_{1\leq |m|<p/2}T_{p,m}.
$$
\par
Now our computations can be interpreted as comparing the moments of
$\cE_p$ with those of
$$
\ctE_p=\sum_{1\leq |m|<p/2}\widetilde{T}_{p,m}
$$
where
$$
\tilde{T}_{p,m}(a)=\frac{\epsilon_f}{\sqrt{p^N/X}}a_{f^{\ast}}(m)
\bfn\left(\frac{m_{k}}{p^N/X}\right)Z_{p,m},
$$
where the $Z_{p,m}$ are, for a given $p$, random variables (defined on
a different probability space) of the form
$$
Z_{p,m}=\mathrm{Tr}(\Theta_{p,m}),
$$
where $(\Theta_{p,m})_{m\in\mathbb{F}_p^{\times}}$ are
Haar-distributed random variables on
$$
G_N=\begin{cases}
  \mathrm{USp}_{N}(\C)&\text{ if $N$ is even},\\
  \mathrm{SU}_N(\C)&\text{ if $N$ is odd},
\end{cases}
$$
and where we assume:
\begin{itemize}
\item if $N$ is even, that the
  $(\Theta_{p,m})_{\in\mathbb{F}_p^{\times}}$ are independent;
\item if $N$ is odd, that the variables $(\Theta_{p,m})_{1\leq m<p/2}$
  are independent, and furthermore
$$
\Theta_{p,-m}={}^t\Theta_{p,m}^{-1}
$$
for all $m$.
\end{itemize}
\par
Indeed, one may interpret Theorem~\ref{theo_katz} as expressing the
fact that
\newcommand{\expect}{\mathbb{E}}
$$
\expect\Bigl(\prod_{i=1}^{\nu}K_{p,m_i}\Bigr)
=\expect\Bigl(\prod_{i=1}^{\nu}Z_{p,m_i}\Bigr)
\Bigl(1+O(p^{-1/2}\Bigr)
$$
for all $\nu$-tuples $\uple{m}$ of integers with $1\leq |m_i|<p/2$
(where $\expect(\cdot)$ denotes expectation on the relevant
probability space).
\par
Using this, it is not too difficult to prove Corollary~\ref{coro_law}
by exploiting the fact that the corresponding central limit theorem
holds for $\ctE_p$ as $p$ tends to infinity, with $X=p^N/\Phi(p)$ as
in that corollary.  In turn, this probabilistic statement follows
easily from the Lindeberg-Feller Theorem for triangular arrays with
independent rows (see, e.g.,~\cite[Th. 27.2, \S 30]{MR1324786}), after
taking into account the relation $Z_{p,-m}=\overline{Z_{p,m}}$ if $N$
is odd.
\par
However, proceeding in this manner, even if it leads to an elegant
proof of the Central Limit Theorem, would not give the more precise
asymptotic of fixed moments in Theorem~\ref{theo_moments}, valid (for
given $\kappa$ and $\lambda$) in a wider range of $p$ and $X$ (at
least, we are not aware of suitable probabilistic references that
would give such a result). We therefore implement the idea by
computing explicitly the asymptotic behavior of the moments.  The
reasoning above is however a good motivation and check that the
combinatorial extraction of the main terms is done correctly.

\subsection{Initial cleaning}%

We begin by assuming that $\kappa, \lambda\geq 1$, since the remaining
cases are easier. We also denote $Y=p^N/X$ to lighten the notation (in
the setting of the Central Limit Theorem of Corollary~\ref{coro_law},
this is $Y=\Phi(p)$, which the reader should therefore think as a
quantity that grows rather slowly with $p$). We assume throughout that
$Y<p/2$, which corresponds to the assumption $2p^{N-1}<X$ in
Theorem~\ref{theo_moments}.
\par
By \eqref{eq_Fourier_coeffs}, we have
\begin{equation*}
\overline{E_f(X,p,a)}=E_{f^\ast}(X,p,a)
\end{equation*}
for all integers $a$ coprime with $p$. Thus, applying Proposition
\ref{propo_voronoi} to $E_f(X,p,a)$ and its conjugate, and then
expanding the $\kappa$-th (resp. $\lambda$-th) power, we obtain the
expression 
\begin{multline*}
  \mathsf{M}=\left(\frac{\epsilon_f}{\sqrt{Y}}\right)^{\nu}\frac{1}{p}\sums_{a\bmod{p}}\
  \sum_{\uple{m}\in\left(\Z^\ast\right)^{\nu}} 
\prod_{k=1}^\kappa \afs(m_{k})K_N(am_{k},p)
  \bfn\left(\frac{m_{k}}{Y}\right) \\
  \times\prod_{\ell=\kappa+1}^{\nu}\af(m_{\ell})
  K_N(am_{\ell},p)\bfns\left(\frac{m_{\ell}}{Y}\right) 
  +O_{\epsilon,
    f}\left(\frac{p^\epsilon}{\sqrt{p}}Y^{(\nu-1)/2}\right),
\end{multline*}
where the sum over $a$ is restricted to $a$ coprime to $p$ (note that
we made a change of variable $a\mapsto -a$, and that we used the
fact~(\ref{eq-afst}) that
$$
\overline{\afs(m)}=
\overline{a_f(1,\ldots,1,m)}=a_f(m,1,\ldots,1)
$$
in the expansion of the conjugates). 
\par
We then split this expression into 
\begin{equation*}
  \mathsf{M}=\Sigma_1+\Sigma_2+O_{\epsilon,
    f}\left(\frac{p^\epsilon}{\sqrt{p}}
Y^{(\nu-1)/2}\right)
\end{equation*}
where $\Sigma_1$ is the contribution of the $\nu$-tuples $\uple{m}$
where $|m_k|<p/2$ for all $k$, and $\Sigma_2$ is the remaining
contribution.
\par
By \eqref{eq_bound_K3} and Corollary \ref{coro_useful}, we easily
estimate $\Sigma_2$ as follows: we have
\begin{align}
  \Sigma_2 & \ll\frac{1}{Y^{\frac{\nu}{2}}}\left(\sum_{\abs{m}\geq p/2}\abs{\afs(m)}\left\vert\bfn\left(\frac{m}{Y}\right)\right\vert\right) 
\left(\sum_{\abs{m}\geq 1}\abs{\afs(m)}\left\vert\bfn
      \left(\frac{m}{Y}\right)\right\vert\right)^{\nu-1} \\
  & \ll\frac{1}{Y^{\frac{\nu}{2}}}p^{1+\epsilon}\left(\frac{Y}{p}\right)^{A}Y^{(1+\epsilon)(\nu-1)} \\
  & \ll
  p^{1+\epsilon}\left(\frac{p^{N-1}}{X}\right)^{A}Y^{\nu/2-1+\epsilon} \label{eq_bound_sigma2}
\end{align}
for all $A>1$ and if $2p^{N-1}<X$.
\par
Thus, the core of the proof is to determine the asympotic behaviour of
$\Sigma_1$. In order to rearrange conveniently this expression, we
first normalize the tuples $\uple{m}$ that remain in $\Sigma_1$.
\par
Each component of the $\nu$-tuple $\uple{m}$ ranges over a finite set
of representatives of the invertible residues classes modulo the odd
prime number $p$, namely
\begin{equation*}
\left\{(1-p)/2,\dots,-1,+1,\dots,(p-1)/2\right\}.
\end{equation*}
\par
We can uniquely write
$$
\uple{m}=(e_1j_1,\dots,e_{\nu}j_{\nu}),
$$
where the components of the $\nu$-tuple
$\uple{e}=(e_1,\dots,e_{\kappa+\lambda})$ belong to $\{\pm 1\}$ and
those of the $\nu$-tuple $\uple{j}=(j_1,\dots,j_{\kappa+\lambda})$
belong to the subset $R=\{1,\dots,P\}$. 
\par
Using this parameterization, we get
\begin{multline}\label{eq_before_combi}
\Sigma_1=\left(\frac{\epsilon_f}{\sqrt{Y}}\right)^{\nu}\frac{1}{p}
\sums_{a\bmod{p}}
\sum_{\uple{e}\in\{\pm 1\}^{\nu}}
\sum_{\uple{j}\in R^{\nu}}
\prod_{k=1}^{\kappa}g_k(j_k)K_N(ae_kj_k,p) 
\prod_{\ell=\kappa+1}^{\nu}g^\ast_\ell(j_{\ell})K_N(ae_\ell j_\ell,p)
\end{multline}
where we have defined
$$
  g_k(m)  \coloneqq  \afs(m)\bfn\left(\frac{e_km}{Y}\right)
  \quad\quad
  g^\ast_\ell(m)  \coloneqq  \af(m)\bfns\left(\frac{e_\ell
      m}{Y}\right)
$$
for integers $m$ and for $1\leq k\leq\kappa$ and
$\kappa+1\leq\ell\leq\nu$.

\subsection{Combinatorial rearranging}%
\label{ssec-reduce-katz}

If we exchange the order of summation in our last expression for
$\Sigma_1$ in order to sum over $a$ first, we encounter sums which are
very close to those of Section~\ref{sec-katz}, but which differ
because there is no provision for the factors $e_kj_k$ or
$e_{\ell}j_{\ell}$ to be distinct, or distinct modulo $\pm 1$, as
required to apply Theorem~\ref{theo_katz}.
\par
We therefore rearrange the sums via a combinatorial rearrangement.
Assume that $s$ and $t$ are two positive integers with $s\leq t$.  We
denote by $P(t,s)$ the set of \emph{surjective} functions
\begin{equation*}
\sigma:\{1,\dots,t\}\to\{1,\dots,s\}
\end{equation*} 
which satisfy the conditions
\begin{equation}\label{eq_perm}
  \forall j\in\{1,\dots,t\},\quad\sigma(j)=1\;\;\; \text{ or } \;\;\;\exists k<j,\quad\sigma(j)=\sigma(k)+1.
\end{equation}
\par
These conditions ensure that $P(s,t)$ parameterizes \emph{bijectively}
the partitions of a set of $t$ elements into $s$ nonempty subsets,
namely into the pre-images $\sigma^{-1}(j)$ for $1\leq j\leq s$. 
\par
In particular, by a formal rearranging, we obtain the following lemma
(see \cite[Lemma 7.3]{HuRu}):

\begin{lemma}\label{lemma_combi}
  Let $t\geq 1$ be a positive integer. If $f\,:\, V^t\rightarrow \C$
  is any function, where $V$ is a finite set, then we have
\begin{equation*}
  \sum_{\uple{j}\in V^t}f(j_1,\ldots,j_t)
  =\sum_{s=1}^t\sum_{\sigma\in P(t,s)}
  \sum_{\substack{(j_1,\dots,j_s)\in V^s\\\text{
        distinct}}}f(j_{\sigma(1)},
\dots,j_{\sigma(t)}).
\end{equation*}
\end{lemma}

We will apply this to the sum over $\uple{j}\in R^{\nu}$ in the
formula~\eqref{eq_before_combi} for $\Sigma_1$. Doing so, we get
\begin{multline*}
  \Sigma_1=\left(\frac{\epsilon_f}{\sqrt{Y}}\right)^{\nu}
  \frac{1}{p}\sums_{a\bmod{p}} \sum_{\uple{e}\in\{\pm 1\}^{\nu}}
  \sum_{s=1}^{\nu} \sum_{\sigma\in P(\nu,s)}
  \sum_{\substack{(j_1,\dots,j_{s})\in R^s \\
      \text{distinct}}} \\
  \prod_{k=1}^{\kappa}g_k(j_{\sigma(k)}) K_N(ae_kj_{\sigma(k)},p)
  \prod_{\ell=\kappa+1}^{\nu} g^\ast_\ell(j_{\sigma(\ell)})K_N(ae_\ell
  j_{\sigma(\ell)},p).
\end{multline*}
\par
We can now collect terms in the products which are equal.  This must
be done while keeping track of the signs $\uple{e}$, and of the
distinction between the indices $j$ which range from $1$ to $\kappa$
and those which range from $\kappa+1$ to $\nu$, and hence a certain
amount of bookkeeping is required.
\par
For $1\leq s\leq \nu$, $\sigma\in P(\nu,s)$ and any
$u\in\{1,\dots,s\}$, we denote first 
\newcommand{\os}{\sigma_u}  
\begin{equation*}
  \os=\abs{\sigma^{-1}(u)},
\end{equation*}
so that, by definition, we have
\begin{equation}\label{eq_note_1}
  \os\geq 1 \quad \text{ and } \quad \sum_{u=1}^s\os=\nu.
\end{equation}
\par
We next count the pre-images of $u$ according to which of the two
intervals they belong: for $1\leq u\leq s$, we let
\newcommand{\oso}{\beta_u} 
\newcommand{\ost}{\gamma_u} 
\begin{align*}
  \oso = & \abs{\{1\leq k\leq\kappa, \sigma(k)=u\}}\geq 0, \\
  \ost = & \abs{\{\kappa+1\leq \ell\leq\kappa+\lambda,
    \sigma(\ell)=u\}}\geq 0,
\end{align*}
noting that these depend on $\sigma$. Hence, we have
\begin{equation}\label{eq_note_2}
\oso+\ost=\os\geq 1.
\end{equation}
\par
Finally, we count the preimages $j$ with a given sign $e_j$, both
their total number, and the number in the two subintervals. For $1\leq
u\leq s$, for $\epsilon=\pm 1$ and $\uple{e}\in\{\pm 1\}^{\nu}$, we
let \newcommand{\osee}{\sigma_u^{\epsilon}(\uple{e})}
\newcommand{\ose}[1]{\sigma_u^{{#1}}(\uple{e})}
\newcommand{\osoee}{\beta_u^{\epsilon}(\uple{e})}
\newcommand{\osoe}[1]{\beta_u^{{#1}}(\uple{e})}
\newcommand{\ostee}{\gamma_u^{\epsilon}(\uple{e})}
\newcommand{\oste}[1]{\gamma_u^{{#1}}(\uple{e})}
\begin{align}
  \osee & = \abs{\{1\leq a\leq\nu, \sigma(a)=u, e_a=\epsilon\}}\geq 0, \\
  \osoee & = \abs{\{1\leq k\leq\kappa, \sigma(k)=u,
    e_k=\epsilon\}}\geq 0, 
\label{eq-beta}\\
  \ostee & =\abs{\{\kappa+1\leq \ell\leq\nu, \sigma(\ell)=u,
    e_\ell=\epsilon\}}\geq 0.
\label{eq-gamma}
\end{align}
\par
These non-negative integers satisfy the following set of properties.
\begin{gather}\label{eq_note_2_2}
  \osoee+\ostee=\osee\geq 0,
\\
\label{eq_note_2_3}
\osoe{1}+\osoe{-1}=\oso\geq 0,\quad\quad
\oste{1}+\oste{-1}=\ost\geq 0,
\\
\label{eq_note_2_4}
\ose{1}+\ose{-1}=\os\geq 1,
\end{gather}
for $1\leq u\leq s$, $\epsilon=\pm 1$, $\uple{e}\in\{\pm 1\}^{\nu}$.
\par 
In terms of these data, by appealing to Lemma \ref{lemma_easy_Bessel}
and the definition~(\ref{eq-sum-s}), we can collect terms in order to
express $\Sigma_1$ in the form
\begin{multline}\label{eq_sigma_1_last}
  \Sigma_1=\left(\frac{\epsilon_f}{\sqrt{Y}}\right)^{\nu}
  \sum_{\uple{e}\in\{\pm 1\}^{\nu}}\sum_{s=1}^{\nu} \sum_{\sigma\in
    P(\nu,s)} \sum_{\substack{(j_1,\dots,j_{s})\in
      R^s\\\text{distinct}}} \prod_{u=1}^{s}
  \afs(j_u)^{\oso}
  \af(j_u)^{\ost} \\
  \times\prod_{u=1}^{s}
  \bfn\left(\frac{j_u}{Y}\right)^{\osoe{1}}
  \bfn\left(\frac{-j_u}{Y}\right)^{\osoe{-1}}
  \\  \times\prod_{u=1}^{s}
  \bfns\left(\frac{j_u}{Y}\right)^{\oste{1}}
  \bfns\left(\frac{-j_u}{Y}\right)^{\oste{-1}}
S_{\sigma^{-1}(\uple{e}),\sigma^{1}(\uple{e})}^{(N)}(\uple{j};p),
\end{multline}
where $\uple{j}=(j_1,\ldots, j_s)$, and we have defined the tuples
$\sigma^1(\uple{u})$ and $\sigma^{-1}(\uple{e})$ in the sum of
Kloosterman sums by
$$
\sigma^1(\uple{e})=(\ose{1})_{1\leq u\leq s},\quad\quad
\sigma^{-1}(\uple{e})=(\ose{-1})_{1\leq u\leq s}.
$$
\par
We note that the parameters $j_u$ which now appear in this last sum
are not only distinct, but also distinct modulo $\{\pm 1\}$ in
$\mathbb{F}_p^{\ast}$. In particular, we can now apply
Theorem~\ref{theo_katz}. This requires us to distinguish between the
cases of odd $N$ and even $N$.


\subsection{The combinatorial analysis for $N$
  odd}\label{sec_combi_odd}%

In this entire section, $N$ is \textbf{odd}. We recall that, in this
case, we have $\epsilon_f=1$. From \eqref{eq_sigma_1_last}, after
applying 
Theorem \ref{theo_katz} (to estimate the sums
$S_{\sigma^{-1}(\uple{e}),\sigma^{1}(\uple{e})}^{(N)}(\uple{j};p)$)
and Proposition \ref{lemma_multiplicity} (to isolate the main terms),
and Lemma~\ref{lemma_easy_Bessel} (to clean-up the weight functions),
one gets
\begin{multline}\label{eq_after_Katz}
  \Sigma_1= \sum_{\uple{e}\in\{\pm
    1\}^{\nu}}\sum_{s=1}^{\nu}\sum_{\sigma\in
    P(\nu,s)}\Sigma_1(\sigma,\uple{e})
  +O\left(\frac{1}{\sqrt{p}}\frac{1}{Y^{\nu/2}}\left(\sum_{1\leq\abs{m}<p/2}\left\vert
        \afs(m)
      \right\vert\left\vert\bfn\left(\frac{m}{Y}\right)\right\vert\right)^{\nu}\right)
\end{multline}
where
\begin{multline*}
  \Sigma_1(\sigma,\uple{e})\coloneqq\frac{1}{Y^{\frac{\nu}{2}}}\sum_{\substack{(j_1,\dots,j_{s})\in
      R^{s} \\
      \text{distinct}}}\prod_{\substack{1\leq u\leq s \\
      N\mid \ose{1}-\ose{-1}}}\afs(j_u)^{\oso} \af(j_u)^{\ost}\\
  \bfn\left(\frac{j_u}{Y}\right)^{\osoe{1}}\overline{\bfn\left(\frac{j_u}{Y}\right)}^{\oste{-1}} 
  \bfns\left(\frac{j_u}{Y}\right)^{\oste{1}}\overline{\bfns\left(\frac{j_u}{Y}\right)}^{\osoe{-1}}A^{(N)}_{\ose{1},\ose{-1}}
\end{multline*}
(the integer $A^{(N)}_{\ose{1},\ose{-1}}$ being defined in
Theorem~\ref{theo_katz}).
\par
Note that, according to Corollary \ref{coro_useful}, the error term in \eqref{eq_after_Katz} satisfies
\begin{equation*}
  \frac{1}{\sqrt{p}}\frac{1}{Y^{\nu/2}}\left(\sum_{1\leq\abs{m}<p/2}\left\vert
      \afs(m)
    \right\vert\left\vert\bfn\left(\frac{m}{Y}\right)\right\vert\right)^{\nu}
  \ll\frac{1}{\sqrt{p}}Y^{\nu/2+\epsilon}
\end{equation*}
for all $\epsilon>0$ if $2p^{N-1}<X$.
\par
Our next step is to show that the main term in $\Sigma_1$ arises from
the contribution of the terms $\Sigma_1(\sigma,\uple{e})$, where
$\sigma$ is in $P(\nu,s)$ for some $s$ with $1\leq s\leq \nu$ and
$\uple{e}$ is in $\{\pm 1\}^{\nu}$, and they satisfy
\begin{equation}\label{eq_step_1}
  \ose{-1}=\ose{1}=1
\end{equation}
for all $u\in\{1,\dots,s\}$. We call such data $(\sigma,\uple{e})$ \textbf{resonant}.
\par
First of all, if the condition $N\mid \ose{1}-\ose{-1}$ in the product
over $u$ in the sum $\Sigma_1(\sigma,\uple{e})$ is not satisfied for
one $u$ at least, then  Corollary \ref{coro_useful} gives immediately
\begin{equation*}
  \Sigma_1(\sigma,\uple{e})\ll_{\epsilon, f}\frac{1}{Y^{1/2-\epsilon}}.
\end{equation*}
\par
Next, we claim that if $(\sigma,\uple{e})$ is non-resonant, with
$\sigma\in P(\nu,s)$ and $\uple{e}\in \{\pm 1\}^{\nu}$ such that 
$$
N\mid \ose{1}-\ose{-1}
$$
for all $u$, then
\begin{equation}\label{eq-non-resonant}
s\leq \frac{\nu-1}{2}.
\end{equation}
\par
Indeed, note first that if $u$ satisfies~(\ref{eq_step_1}), then we
have $\os=2$ by~(\ref{eq_note_2_4}). On the other hand, if $u$ does
not satisfy~(\ref{eq_step_1}), then in view of the conditions $N\mid
\ose{1}-\ose{-1}$ and
$$
\ose{1}+\ose{-1}=\os\geq 1,
$$
we see that either $\ose{-1}=\ose{1}\geq 2$, and then $\os\geq 4$, or
$\ose{1}+\ose{-1}\geq N\geq 3$. Thus, in all cases, we have
$$
\os\geq 3
$$
unless $u$ satisfies~(\ref{eq_step_1}). Denoting by $U$ the set of
those $u$ which do satisfy~(\ref{eq_step_1}), we note that if $\sigma$
is not resonant, we have $|U|<s$, and hence
$$
3(s-|U|)\geq 2(s-|U|)+1,
$$
and we obtain
$$
\nu=\sum_{1\leq u\leq s}{\os}=
\sum_{u\in U}{\os}+\sum_{u\notin U}{\os}
\geq 2|U|+3(s-|U|)\geq 2|U|+2(s-|U|)+1\geq 2s+1,
$$
which gives~(\ref{eq-non-resonant}).
\par
Using Corollary \ref{coro_useful}, we see that
\begin{equation}\label{eq-non-resonant-bd}
\Sigma_1(\sigma,\uple{e})\ll_{\epsilon, f}\frac{1}{Y^{1/2-\epsilon}}
\end{equation}
for $(\sigma,\uple{e})$ non-resonant, provided $2p^{N-1}<X$.
\par
Observe that if $(\sigma,\uple{e})$ is resonant, then each $\os$ is
equal to $2$, hence $\nu=2s$ is even. We can therefore write
$$
\Sigma_1=\delta_{2\mid\nu}
\sum_{\uple{e}\in\{\pm 1\}^{\nu}}\sum_{\substack{\sigma\in P(\nu,\nu/2) \\
    \sigma\text{ resonant}}}\Sigma_1(\sigma,\uple{e}) +O_{\epsilon,
  f}\left(Y^{-1/2+\epsilon}+\frac{1}{\sqrt{p}}Y^{\nu/2+\epsilon}\right)
$$
if $2p^{N-1}<X$, and the corresponding $\Sigma_1(\sigma,\uple{e})$ are
given by
\begin{multline*}
  \Sigma_1(\sigma,\uple{e})=\frac{1}{Y^{\frac{\nu}{2}}}
  \sum_{\substack{(j_1,\dots,j_{\nu/2})\in R^{\nu/2} \\
      \text{distinct}}}
  \prod_{u=1}^{\nu/2}\afs(j_u)^{\oso}\af(j_u)^{\ost} \\
  \times\bfn\left(\frac{j_u}{Y}\right)^{\osoe{1}}\overline{\bfn\left(\frac{j_u}{Y}\right)}^{\oste{-1}}
  \bfns\left(\frac{j_u}{Y}\right)^{\oste{1}}\overline{\bfns\left(\frac{j_u}{Y}\right)}^{\osoe{-1}}
\end{multline*}
by the condition $A_{1,1}^{(N)}=1$ (see Proposition~\ref{lemma_multiplicity}).
\par
By \eqref{eq_step_1} and \eqref{eq_note_2_2}, for all $u$ with
$1\leq u\leq s=\nu/2$, the $4$-tuple
\newcommand{\bose}{\omega_u(\uple{e})}
\begin{equation}\label{eq-bose}
\bose=\left(\osoe{-1},\oste{-1},\osoe{1},\oste{1}\right)
\end{equation}
(which also depends on $\sigma$) is one of the four tuples in the set
$\omega=\{\omega_1,\omega_2,\omega_3,\omega_4\}$, where
\begin{gather}\label{eq-oms}
  \omega_1=(0,1,1,0),\quad\quad \omega_2=(1,0,0,1),\\
  \omega_3=(1,0,1,0),\quad\quad \omega_4=(0,1,0,1).
\end{gather}
\par
The sum $\Sigma_1(\sigma,\uple{e})$ almost factors as a product of
four independent terms. Indeed, if we sum over all $\uple{j}$,
relaxing the condition that $\uple{j}$ has distinct components, we
only introduce sums whose contributions is dominated by the error
terms already present. Hence we have
$$
  \Sigma_1=\delta_{2\mid\nu}
  \sum_{\uple{e}\in\{\pm 1\}^{\nu}}\sum_{\substack{\sigma\in P(\nu,\nu/2) \\
      \text{resonant}}}\tilde{\Sigma}_1(\sigma,\uple{e}) +O_{\epsilon,
    f}\left(Y^{-1/2+\epsilon}+\frac{1}{\sqrt{p}}Y^{\nu/2+\epsilon}\right)
$$
if $2p^{N-1}<X$, where 
\newcommand{\ube}[1]{u_{{#1}}(\uple{e})}
\begin{multline}\label{eq_before_variance}
  \tilde{\Sigma}_1(\sigma,\uple{e})=
  \left(\frac{1}{Y}\sum_{1\leq m<p/2}|\afs(m)|^2\left\vert\bfn\left(\frac{m}{Y}\right) \right\vert^2\right)^{\ube{1}} \\
  \times\left(\frac{1}{Y}\sum_{1\leq m<p/2} |\af(m)|^2
    \left\vert\bfns\left(\frac{m}{Y}\right) \right\vert^2\right)^{\ube{2}} \\
  \times\left(\frac{1}{Y}\sum_{1\leq m<p/2}\afs(m)^2
    \bfn\left(\frac{m}{Y}\right)
    \overline{\bfns\left(\frac{m}{Y}\right)}\right)^{\ube{3}} \\
  \times\left(\frac{1}{Y}\sum_{1\leq m<p/2}\af(m)^2
    \bfns\left(\frac{m}{Y}\right)\overline{\bfn\left(\frac{m}{Y}\right)}\right)^{\ube{4}},
\end{multline}
with exponents given by
\begin{equation*}
  \ube{b}=\left\vert\left\{1\leq u\leq \nu/2,\ 
     \bose=\omega_b\right\}\right\vert
\end{equation*}
for $1\leq b\leq 4$ (again, these depend on $\sigma$).
\par
The four terms in the product are of the type considered in
Section~\ref{sec_variance}, but note that they will actually look
different if $f$ is self-dual and when $f$ is not. So we split into
two cases again. We begin with the case when $f$ is not self-dual,
which we think of as the generic case. (When $N=3$, the $GL(3)$
self-dual cusp forms are the symmetric square lifts of $GL(2)$ forms,
as explained in~\cite{ramakrishnan}, and hence are very special;
similar characterizations of self-dual representations of $GL(N)$ for
all $N\geq 3$ are expected to hold, but are not known in full
generality).

\subsubsection{The non self-dual case for $N$ odd}\label{sec_variance_1}%

In this subsection, we assume that $f$ is \textbf{not self-dual},
namely $f^\ast\neq f$.
\par
We then show that the main term in $\Sigma_1$ in
\eqref{eq_before_variance} comes from the contribution of the resonant
$\tilde{\Sigma}_1(\sigma,\uple{e})$ where $(\sigma,\uple{e})$ is such
that
\begin{equation}\label{eq_step_2}
\bose=\omega_1\text{ or }
\bose=\omega_2
\end{equation}
for all $u\in\{1,\dots,\nu/2\}$, i.e., those where
$$
\ube{3}=\ube{4}=0,
$$
which we call the \textbf{focusing} pairs.
\par
Indeed, each of the four sums in~(\ref{eq_before_variance}) can be
estimated asymptotically using Proposition~\ref{propo_variance},
applied with $(Y,Z)=(Y,p/2)$ (recall that $p/2>Y$),
$\theta=1/2-1/(N^2+1)$ and suitable smooth functions $\mathrm{B}$, namely
$$
\mathrm{B}(y)=|\bfn(y)|^2,\quad
\mathrm{B}(y)=|\bfns(y)|^2,\quad
\mathrm{B}(y)=\bfn(y)\overline{\bfns(y)},\quad
\mathrm{B}(y)=\overline{\bfn(y)}\bfns(y)
$$
in the four successive terms. These satisfy the assumption of
Proposition~\ref{propo_variance} with
$$
\eta=2\max|\Re(\alpha_{j,\infty}(f))|\leq 1-\frac{2}{(N^2+1)}<1,
$$
by
Proposition~\ref{propo_bounds} and~(\ref{eq_Selberg}). 
\par
Proposition~\ref{propo_variance} leads to the estimate
\begin{equation*}
  \tilde{\Sigma}_1(\sigma,\uple{e})\ll
  Y^{(-1/2+\theta+\epsilon)(\ube{3}+\ube{4})},
\end{equation*}
and hence
$$
\tilde{\Sigma}_1(\sigma,\uple{e})\ll Y^{-1/2+\theta+\epsilon},
$$
unless $\ube{3}+\ube{4}=0$, i.e., unless $(\sigma,\uple{e})$ is
focusing, since $\ube{3}$ and $\ube{4}$ are non-negative integers.
\par
From Proposition \ref{propo_variance}, using the notation introduced
there, we now deduce that, for $X>2p^{N-1}$, we have
\begin{multline}\label{eq_generic}
\Sigma_1=\delta_{2\mid\nu}
\left(r_fH_{f,f^\ast}(1)\right)^{\nu/2}\sum_{\uple{e}\in\{\pm
  1\}^{\nu}}
\sum_{\substack{\sigma\in P(\nu,\nu/2) \\
\ube{3}+\ube{4}=0}} 
\Bigl(\mathcal{M}\left[\left\vert\bfn\right\vert^2\right](1)\Bigr)
^{\ube{1}}
\Bigl(\mathcal{M}\left[\left\vert\bfns\right\vert^2\right](1)\Bigr)
^{\ube{2})} \\
+O_{\epsilon, f}\left(Y^{-1/2+\theta+\epsilon}+Y^{-1/2+\epsilon}+\frac{1}{\sqrt{p}}Y^{\nu/2+\epsilon}\right)
\end{multline}
by \eqref{eq_before_unitary} (we also used the properties
$r_f=r_{f^\ast}$ and $H_{f,f^\ast}(1)=H_{f^\ast,f}(1)$.)
\par
The remaining set of focusing pairs $(\sigma,\uple{e})$ has now a very
clean structure. We state this as a lemma.

\begin{lemma}\label{lm-focusing-pairs}
  With notation as above, for any $(\sigma,\uple{e})$ which is a
  focusing pair,
  we have
  $\kappa=\lambda=\nu/2$. Furthermore, the map
$$
(\sigma,\uple{e})\mapsto (\tilde{\sigma},\tilde{\uple{e}}),
$$
where $\tilde{\sigma}$ is the restriction of $\sigma$ to
$\{\kappa+1,\ldots,\nu\}$ and $\tilde{\uple{e}}$ is the $\kappa$-tuple
$(e_u)_{1\leq u\leq \kappa}$, is a bijection between the set of
focusing pairs $(\sigma,\uple{e})$ and the product set
$S_{\kappa}\times \{\pm 1\}^{\kappa}$, where $S_{\kappa}$ is the set
of bijections from $\{\kappa+1,\ldots,\nu\}$ to $\{1,\ldots,\kappa\}$.
\par
We then have, for all such $(\sigma,\uple{e})$, the relation
\begin{equation}\label{eq-ube}
\ube{1}=|\{u\ ,\  1\leq u\leq \kappa\text{ and } \tilde{e}_u=1\}|,
\end{equation}
and $\ube{2}=\kappa-\ube{1}$. 
\end{lemma}

\newcommand{\uote}{u_1(\tilde{\uple{e}})}
One important consequence of this lemma is that the exponents
$\ube{1}$ and $\ube{2}$ in~(\ref{eq_before_variance}) are
\emph{independent} of $\sigma$ when $\ube{3}=\ube{4}=0$. We will
denote $\ube{1}$ by $\uote$.

\begin{proof}
Using~(\ref{eq_note_2_3}) and the definition of $\omega_1$,
$\omega_2$, the focusing assumption $\ube{3}+\ube{4}=0$ imply that for all $u$
we have  
$$
\oso=\osoe{1}+\osoe{-1}=1,\quad\quad
\ost=\oste{1}+\oste{-1}=1.
$$
\par
By definition of $\beta_u$ (resp. $\gamma_u$), the first property
(resp. second) implies that the restriction of $\sigma$ to
$\{1,\ldots,\kappa\}$ (resp. to $\{\kappa+1,\ldots,\nu\}$) is
surjective (resp. surjective). In particular, $\kappa\geq \nu/2$ and
$\lambda\geq \nu/2$. Using $\nu=\kappa+\lambda$, this means that
$\kappa=\lambda=\nu/2$. 
\par
In turn, the two restrictions of $\sigma$ are then surjective maps
from and to sets of size $\nu/2$, and hence are bijections. In
particular, the map $\tilde{\sigma}$ is indeed an element of the set
$S_{\kappa}$, and $(\tilde{\sigma},\tilde{\uple{e}})$ belongs to
$S_{\kappa}\times \{\pm 1\}^{\kappa}$.
\par
We next show that the map is injective. Indeed, the condition
\eqref{eq_perm} imposed on all elements $\sigma$ of $P(\nu,\nu/2)$
implies, by an elementary induction, that the restriction of $\sigma$
to the initial segment $\{1\,\ldots,\kappa\}$ is the identity
map. Thus $\sigma$ is entirely determined by $\tilde{\sigma}$.
\par
The condition~(\ref{eq_step_2}) and the definition~(\ref{eq-bose}) of
$\bose$ and~(\ref{eq-oms}) of the permitted four-tuples $\omega_1$ and
$\omega_2$ imply that
\begin{equation}\label{eq-alles}
e_{\tilde{\sigma}^{-1}(u)}=-e_u
\end{equation}
for $\kappa+1\leq u\leq \nu$. Indeed, for all $u$ with $1\leq u\leq
\nu/2$, we have either
$$
\oste{-1}=\osoe{1}=1
$$
or
$$
\oste{-1}=\osoe{1}=0.
$$
\par
Assume the first holds: this means, by~(\ref{eq-beta})
and~(\ref{eq-gamma}), that (i) there exists a single $\ell$ with
$\kappa+1\leq \ell\leq \nu$, $\sigma(\ell)=u$, and $e_{\ell}=1$; (ii)
there exists a single $k$ with $1\leq k\leq \kappa$, $\sigma(k)=u$,
and $e_{k}=-1$. Since $\sigma$ is the identity on
$\{1,\ldots,\kappa\}$, we have $\sigma(k)=k=u$. Then, since
$\tilde{\sigma}$ is a bijection, we have
$\ell=\tilde{\sigma}^{-1}(u)=\tilde{\sigma}^{-1}(k)$, and hence
$e_{\tilde{\sigma}^{-1}(u)}=e_{\ell}=-1=-e_k=-e_u$. 
\par
The other case is similar, and we obtain~(\ref{eq-alles}) for all $u$.
Thus $\uple{e}$ is entirely determined by $\tilde{\uple{e}}$, and we
conclude that the map $(\sigma,\uple{e})\mapsto
(\tilde{\sigma},\tilde{\uple{e}})$ is injective.
\par
We now show that it is surjective. Given
$(\tilde{\sigma},\tilde{\uple{e}})\in S_{\kappa}\times\{\pm
1\}^{\kappa}$, we can define a surjective map $\sigma$ from
$\{1,\ldots,\nu\}$ to $\{1,\ldots,\nu/2\}$ by extending
$\tilde{\sigma}$ by the identity on $\{1,\ldots,\kappa\}$. Then we
define a $\nu$-tuple $\uple{e}$ using~(\ref{eq-alles}) and the
bijectivity of $\tilde{\sigma}$. It is clear that this pair
$(\sigma,\uple{e})$ is mapped to $(\tilde{\sigma},\tilde{\uple{e}})$,
but we must check that $\sigma$ satisfies~(\ref{eq_perm}) in order to
conclude. This condition is indeed true: for $1\leq j\leq \kappa$,
this is because $\sigma(j)=j$, which satisfies~(\ref{eq_perm}), and
for $\kappa+1\leq j\leq \nu$, this is because $k=\sigma(j)$ is between
$1$ and $\kappa$, and hence, if $k\not=1$, we have
$\sigma(j)=k=\sigma(k-1)+1$ with $k-1<\kappa<j$.
\par
Finally, we obtain~(\ref{eq-ube}) because $\ube{1}$ is the number of
$k$ with $1\leq k\leq \kappa$ where $e_k=1$ (since,
by~(\ref{eq-alles}), this is also the number of $\ell$ with
$\kappa+1\leq \ell\leq\nu$ for which $e_{\ell}=-1$). By assumption, we
have
$$
\nu/2=\ube{1}+\ube{2}+\ube{3}+\ube{4}=\ube{1}+\ube{2},
$$
hence the formula for $\ube{2}$.
\end{proof}

Using this lemma, the main term in \eqref{eq_generic} becomes
\begin{multline*}
  \delta_{\kappa=\lambda}\left(r_fH_{f,f^\ast}(1)\right)^{\kappa}
  \sum_{\tilde{\sigma}\in S_{\kappa}}\sum_{\tilde{\uple{e}}\in\{\pm
    1\}^\kappa}
  \left(\mathcal{M}\left[\left\vert\bfn\right\vert^2\right](1)\right)^{\uote}
  \left(\mathcal{M}
    \left[\left\vert\bfns\right\vert^2\right](1)\right)^{\kappa-\uote}
  \\
  =
  \delta_{\kappa=\lambda}\kappa!\left(r_fH_{f,f^\ast}(1)\left(\mathcal{M}\left[\left\vert\bfn\right\vert^2\right](1)+\mathcal{M}\left[\left\vert\bfns\right\vert^2\right](1)\right)\right)^{\kappa},
\end{multline*}
since, as we observed, $\uote$ is independent of $\tilde{\sigma}$ in
this double sum. Appealing finally to Corollary
\ref{coro_unitary_Bessel}, this expression is equal to
\begin{equation*}
  \delta_{\kappa=\lambda}\kappa!\left(r_f
    H_{f,f^\ast}(1)\abs{\abs{w}}_2^2\right)^{\kappa}.
\end{equation*}
\par
To conclude this section, we have shown that, for $X>2p^{N-1}$, we have
\begin{equation}\label{eq_Sigma_1_MT}
  \Sigma_1=\delta_{\kappa=\lambda}2^\kappa\kappa!
  \left(\frac{r_fH_{f,f^\ast}(1)}{2}\abs{\abs{w}}_2^2\right)^{\kappa} 
  +O_{\epsilon,f}\left(Y^{-1/2+\theta+\epsilon}+Y^{-1/2+\epsilon}+
    \frac{1}{\sqrt{p}}Y^{\nu/2+\epsilon}\right),
\end{equation}
in the case of a form $f$ which is not self-dual, and of odd $N$.

\subsubsection{The self-dual case for $N$ odd}\label{sec_variance_2}%

We now consider the case where $f$ is \textbf{self-dual}, namely
$f^\ast=f$, which is simpler. Indeed, in this case, the four terms in
\eqref{eq_before_variance} are all equal, and since
$$
\ube{1}+\ube{2}+\ube{3}+\ube{4}=\frac{\nu}{2},
$$
we obtain
\begin{multline*}
  \Sigma_1=\delta_{2\mid\nu}\left(\frac{1}{Y}\sum_{1\leq
      m<p/2}\afs(m)^2
    \left\vert\bfn\left(\frac{m}{Y}\right) \right\vert^2\right)^{\nu/2} \\
  \times\sum_{\uple{e}\in\{\pm 1\}^{\nu}} \sum_{\substack{\sigma\in
      P(\nu,\nu/2) \\\text{resonant}}}1+O_{\epsilon,
    f}\left(Y^{-1/2+\epsilon}+\frac{1}{\sqrt{p}}Y^{\nu/2+\epsilon}\right).
\end{multline*}
for $X>2p^{N-1}$. 
\par
By another application of Proposition \ref{propo_variance} and
Corollary \ref{coro_unitary_Bessel}, we conclude that
\begin{multline*}
  \Sigma_1=\delta_{2\mid\nu}\left(\frac{r_fH_{f,f^\ast}(1)}{2}\abs{\abs{w}}_2^2\right)^{\nu/2}\sum_{\uple{e}\in\{\pm
    1\}^{\nu}} 
\sum_{\substack{\sigma\in P(\nu,\nu/2) \\ \text{resonant}}}1\\
  +O_{\epsilon,f}\left(Y^{-1/2+\theta+\epsilon}+Y^{-1/2+\epsilon}+\frac{1}{\sqrt{p}}Y^{\nu/2+\epsilon}\right).
\end{multline*}
\par
We now have a (standard) lemma:

\begin{lemma}\label{lm-resonant-pairs}
With notation as above, the number of pairs $(\sigma,\uple{e})$ which
are resonant is
$$
\frac{\nu!}{(\nu/2)!}.
$$
\end{lemma}

\begin{proof}[\proofname{} of lemma \ref{lm-resonant-pairs}]
We must count the  pairs $(\sigma,\uple{e})$, where $\sigma\in
P(\nu,\nu/2)$, $\uple{e}\in \{\pm 1\}^{\nu}$, such that the resonance condition
$$
\ose{1}=\ose{-1}=1
$$
holds for $1\leq u\leq \nu/2$.  This condition means exactly that, for
each $u$, there exists exactly one $j$, $1\leq j\leq \nu$, such that
$\sigma(j)=u$ and $e_j=1$, and one $k$, $1\leq k\leq \nu$ with
$\sigma(k)=u$ and $e_j=-1$. This means that $\sigma$ is an element of
$P(\nu,\nu/2)$ such that each $u\in \{1,\ldots,\nu/2\}$ has two
preimages in $\{1,\ldots, \nu\}$. 
\par
Given such a \emph{fixed} $\sigma$, a pair $(\sigma,\uple{e})$ is
resonant if and only if the signs associated to the two elements of
$\sigma^{-1}(u)$ are opposite for each $u$. If we fix a subset
$I\subset \{1,\ldots, \nu\}$, of size $\nu/2$, such that $\sigma$
restricted to $I$ is bijective, the tuples $\uple{e}$ are determined
by $e_j$, $j\in I$, and these signs can be chosen arbitrarily. Thus,
for every fixed $\sigma$ satisfying the condition, there are exactly
$2^{\nu/2}$ tuples $\uple{e}$ with $(\sigma,\uple{e})$ resonant.
\par
It remains to count the number of $\sigma\in P(\nu,\nu/2)$ such that
each $u\in \{1,\ldots,\nu/2\}$ has two preimages in $\{1,\ldots,
\nu\}$. As we observed after introducing~(\ref{eq_perm}), this amounts
to counting the set $\tilde{P}(\nu,\nu/2)$ of partitions of
$\{1,\ldots,\nu\}$ in $\nu/2$ subsets of size $2$, and this set has
order equal to
\begin{equation}\label{eq-match}
\frac{\nu!}{2^{\nu/2}\left(\nu/2\right)!}
\end{equation}
(indeed, the symmetric group $\mathfrak{S}_{\nu}$ acts transitively on
$\tilde{P}(\nu,\nu/2)$, and the stabilizer of the element
$\{\{1,2\},\ldots, \{\nu-1,\nu\}\}$ of $\tilde{P}(\nu,\nu/2)$ is
isomorphic to $(\Z/2\Z)^{\nu/2}\times \mathfrak{S}_{\nu/2}$ (one can
also, for instance, apply \cite[Example 5.2.6 and Exercice
5.43]{Stan00}).
\end{proof}

Thus, if $X>2p^{N-1}$, we have
\begin{equation}\label{eq_Sigma_1_MT_2}
\Sigma_1=\delta_{2\mid \nu}\frac{\nu!}{2^{\nu/2}\left(\nu/2\right)!}\left(r_fH_{f,f^\ast}(1)\abs{\abs{w}}_2^2\right)^{\nu/2} \\
+O_{\epsilon,f}\left(Y^{-1/2+\theta+\epsilon}+
  \frac{1}{\sqrt{p}}Y^{\nu/2+\epsilon}\right)
\end{equation}
if $f$ is self-dual, and $N$ is odd.

\subsubsection{End of the proof of Theorem \ref{theo_moments} for $N$
  odd}%

Equations \eqref{eq_bound_sigma2} and \eqref{eq_Sigma_1_MT_2} imply
Theorem \ref{theo_moments} if $f$ is self-dual.
\par
Equations \eqref{eq_bound_sigma2} and \eqref{eq_Sigma_1_MT} imply
Theorem \ref{theo_moments} if $\kappa\lambda\neq 0$ and $f$ is not
self-dual. To conclude, we briefly indicate what happens if $f$ is not
self-dual and $\lambda=0$ (the case $\kappa=0$ being similar). Arguing
as before, the understanding of $\mathsf{M}_f(X,p,(\kappa,0))$ boils
down to the estimation of
$$
  \Sigma_1=\delta_{2\mid\kappa}\sum_{\uple{e}\in\{\pm 1\}^\kappa}\sum_{\substack{\sigma\in P(\kappa,\kappa/2) \\\text{resonant}}}\Sigma_{1}(\sigma,\uple{e}) 
  +O_{\epsilon,
    f}\left(Y^{-1/2+\epsilon}+\frac{1}{\sqrt{p}}Y^{\kappa/2+\epsilon}\right)
$$
where
\begin{equation*}
  \Sigma_{1}(\sigma,\uple{e})=\left(\frac{1}{Y}\sum_{1\leq m<p/2}
    \afs(m)^2\bfn\left(\frac{m}{Y}\right)\overline{\bfns\left(\frac{m}{Y}\right)}\right)^{\kappa/2},
\end{equation*}
and we see that each such sum is subsumed in the error term by
Proposition~\ref{propo_variance} for $f\not=f^{\ast}$.

\subsection{The combinatorial analysis for $N$
  even}\label{sec_even_combi}%

In this entire section, $N$ is \textbf{even}. We recall that, in this
case, $\epsilon_f$ may be either $1$ or $-1$. We will then, in
addition to $\bfn(x)$ and $\bfns(x)$, use the notation
\newcommand{\bfnp}{\mathcal{B}^{+}_{\uple{\alpha}}}
\newcommand{\bfnm}{\mathcal{B}^{-}_{\uple{\alpha}}}
$$
\bfnp(x)=\mathcal{B}^{+1}_{\uple{\alpha}}[w](x),\quad
\bfnm(x)=\mathcal{B}^{-1}_{\uple{\alpha}}[w](x)
$$
(recall the definition~(\ref{eq_CL_2})).
\par
The general flow of the
argument is similar to that of the previous section, but the
combinatorics involved differs slightly.
\par
We begin as in the case of odd $N$. By \eqref{eq_sigma_1_last}, after
applying
Theorem \ref{theo_katz} (to estimate the sums
$S_{\sigma^{-1}(\uple{e}),\sigma^{1}(\uple{e})}^{(N)}(\uple{j};p)$)
and Proposition \ref{lemma_multiplicity} (to isolate the main terms),
and Lemma~\ref{lemma_easy_Bessel} (to clean-up the weight functions),
one gets
\begin{equation}\label{eq_after_Katz_even}
\Sigma_1=\sum_{\uple{e}\in\{\pm 1\}^{\nu}}\sum_{s=1}^{\nu}\sum_{\sigma\in P(\nu,s)}\Sigma_1(\sigma,\uple{e}) 
+O\left(\frac{1}{\sqrt{p}}\frac{1}{Y^{\nu/2}}\left(\sum_{1\leq\abs{m}<p/2}
|\afs(m)|\left\vert\bfn\left(\frac{m}{Y}\right)\right\vert\right)^{\nu}\right)
\end{equation}
where
\begin{multline*}
  \Sigma_1(\sigma,\uple{e})\coloneqq\left(\frac{\epsilon_f}{\sqrt{Y}}\right)^{\nu}\sum_{\substack{(j_1,\dots,j_{s})\in R^{s} \\
      \text{distinct}}}\prod_{\substack{1\leq u\leq s \\
      2\mid \ose{-1}+\ose{1}}}\afs(j_u)^{\oso} \af(j_u)^{\ost}
  \\  \times \bfn\left(\frac{j_u}{Y}\right)^{\osoe{1}}\bfn\left(\frac{-j_u}{Y}\right)^{\osoe{-1}} \bfns\left(\frac{j_u}{Y}\right)^{\oste{1}}\bfns\left(\frac{-j_u}{Y}\right)^{\oste{-1}}B^{(N)}_{\ose{1}+\ose{-1}},
\end{multline*}
the integer $B^{(N)}_{\ose{1}+\ose{-1}}$ being defined in
Theorem~\ref{theo_katz}.
\par
As earlier, according to Corollary \ref{coro_useful}, the error term
in \eqref{eq_after_Katz_even} is
\begin{equation*}
\ll \frac{1}{\sqrt{p}}Y^{\nu/2+\epsilon}
\end{equation*}
for all $\epsilon>0$ if $2p^{N-1}<X$.
\par
We now show that the main term in $\Sigma_1$ comes from the
contribution of pairs $(\sigma,\uple{e})$ where $\sigma$ in $P(\nu,s)$
and $\uple{e}$ in $\{\pm 1\}^{\nu}$ satisfy
\begin{equation}\label{eq_step_1_even}
\left(\ose{-1},\ose{1}\right)\in\left\{(2,0),(0,2)\right\}.
\end{equation}
for all $u\in\{1,\dots,s\}$. As before, we call these pairs \textbf{resonant}.
\par
Indeed, as in the case of $N$ odd, we first see that if the condition
that $\ose{1}$ and $\ose{-1}$ be even, in the product over $u$ in the
sum $\Sigma_1(\sigma,\uple{e})$, is not satisfied for one $u$ at
least, then Corollary \ref{coro_useful}
gives~(\ref{eq-non-resonant-bd}). Thus, as before, it is enough to
prove that
$$
s\leq \frac{\nu-1}{2}
$$
if $(\sigma,\uple{e})$ is non-resonant and $\ose{1}$, $\ose{-1}$ are
even for all $u$, since this leads to the same
bound~(\ref{eq-non-resonant-bd}), for $2p^{N-1}<X$, as in the odd
case.
\par
If $u$ satisfies~(\ref{eq_step_1_even}), then we have $\os=2$
by~(\ref{eq_note_2_4}). If $u$ does not
satisfy~(\ref{eq_step_1_even}), then since $\ose{1}$ and $\ose{-1}$
are even and $\ose{1}+\ose{-1}=\os\geq 1$, we must have $\ose{1}\geq
2$, $\ose{-1}\geq 2$, and therefore $\sigma_u=\ose{1}+\ose{-1}\geq 4$.
Denoting by $U$ the set of those $u$ which
satisfy~(\ref{eq_step_1_even}), so that $|U|<s$ if $\sigma$ is not
resonant, we get
$$
\nu=\sum_{1\leq u\leq s}{\os}= \sum_{u\in U}{\os}+\sum_{u\notin
  U}{\os} \geq 2|U|+4(s-|U|)\geq 2|U|+2(s-|U|)+1\geq 2s+1,
$$
as desired.
\par
If $(\sigma,\uple{\nu})$ is resonant, then $\nu$ is even and $s=\nu/2$
by \eqref{eq_note_1}. It follows therefore that
$$
\Sigma_1=\delta_{2\mid\nu}
\sum_{\uple{e}\in\{\pm 1\}^{\nu}}\sum_{\substack{\sigma\in P(\nu,\nu/2) \\
    \text{resonant}}} \tilde{\Sigma}_1(\sigma,\uple{e}) +O_{\epsilon,
  f}\left(Y^{-1/2+\epsilon}+\frac{1}{\sqrt{p}}Y^{\nu/2+\epsilon}\right)
$$
if $2p^{N-1}<X$, where
\begin{multline*}
  \tilde{\Sigma}_1(\sigma,\uple{e})=\frac{1}{Y^{\frac{\nu}{2}}}\sum_{\substack{(j_1,\dots,j_{\nu/2})\in
      R^{\nu/2} \\
      \text{distinct}}}\prod_{u=1}^{\nu/2}\afs(j_u)^{\oso}
  \af(j_u)^{\ost} \\
  \times
  \bfn\left(\frac{j_u}{Y}\right)^{\osoe{1}}\bfn\left(\frac{-j_u}{Y}\right)^{\osoe{-1}}
  \bfns\left(\frac{j_u}{Y}\right)^{\oste{1}}\bfns\left(\frac{-j_u}{Y}\right)^{\oste{-1}},
\end{multline*}
by Proposition \ref{lemma_multiplicity} (i.e., the condition $B^{(N)}_{2}=1$).
\par
Defining the $4$-tuple $\bose$ for $1\leq u\leq \nu/2$ as
in~(\ref{eq-bose}), there are now $6$ possibilities for  $\bose$,
namely
\begin{gather}\label{eq-6}
\omega_1=(1,1,0,0),\quad \omega_2=(0,0,1,1),\\
\omega_3=(2,0,0,0),\quad
\omega_4=(0,0,2,0),\quad \omega_5=(0,2,0,0),\quad \omega_6=(0,0,0,2). 
\end{gather}
\par
If we sum in $\tilde{\Sigma}_1$ over all the possible $\nu/2$-tuples
$(j_1,\dots,j_{\nu/2})$ instead of those with distinct components, we
only introduce a difference with is dominated by the error term. Thus,
collecting identical terms in the product, and denoting
\begin{equation*}
  \ube{b}=\left\vert\left\{1\leq u\leq \nu/2, \bose=\omega_b\right\}\right\vert
\end{equation*}
for $1\leq b\leq 6$, similarly to the case of $N$ odd, we can write
$$
\Sigma_1=\delta_{2\mid\nu}\sum_{\uple{e}\in\{\pm
  1\}^{\nu}}\sum_{\substack{\sigma\in P(\nu,\nu/2) \\\text{resonant}}}
\tilde{\Sigma}_1(\sigma,\uple{e}) \\
+O_{\epsilon, f}
\left(Y^{-1/2+\epsilon}+\frac{1}{\sqrt{p}}Y^{\nu/2+\epsilon}\right)
$$
if $2p^{N-1}<X$, with 
\begin{multline}\label{eq_before_variance_even}
  \tilde{\Sigma}_1(\sigma,\uple{e})=\left(\frac{1}{Y}\sum_{1\leq
      m<p/2}
    |\afs(m)|^2\bfn\left(\frac{-m}{Y}\right)\bfns\left(\frac{-m}{Y}\right)\right)^{\ube{1}} \\
  \times\left(\frac{1}{Y}\sum_{1\leq m<p/2}
    |\afs(m)|^2\bfn\left(\frac{m}{Y}\right)\bfns\left(\frac{m}{Y}\right)\right)^{\ube{2}} \\
  \times\left(\frac{1}{Y}\sum_{1\leq m<p/2}\afs(m)^2
    \bfn\left(\frac{-m}{Y}\right)^2\right)^{\ube{3}}
  \left(\frac{1}{Y}\sum_{1\leq m<p/2}\afs(m)^2
    \bfn\left(\frac{m}{Y}\right)^2\right)^{\ube{4}} \\
  \times\left(\frac{1}{Y}\sum_{1\leq m<p/2}\af(m)^2
    \bfns\left(\frac{-m}{Y}\right)^2\right)^{\ube{5}}
  \left(\frac{1}{Y}\sum_{1\leq
      m<p/2}\af(m)^2\bfns\left(\frac{m}{Y}\right)^2\right)^{\ube{6}}.
\end{multline}

\subsubsection{The non self-dual case for $N$ even}\label{sec_variance_3}%

We must again distinguish between the case where $f$ is not self-dual,
and the self-dual case. Here, we assume that $f$ is \textbf{not
  self-dual}, namely $f^\ast\neq f$.
\par
First, note that we can apply again Proposition \ref{propo_variance}
for suitable functions $\mathrm{B}$ to the six sums in
\eqref{eq_before_variance_even}, leading to the bound
\begin{equation*}
  \tilde{\Sigma}_1(\sigma,\uple{e})\ll
Y^{(-1/2+\theta+\epsilon)(\ube{3}+\ube{4}+\ube{5}+\ube{6})},
\end{equation*}
and hence all terms in \eqref{eq_before_variance_even} such that one
of $\ube{3}$, $\ube{4}$, $\ube{5}$ or $\ube{6}$ is non-zero contribute
to the error term. We will say that $(\sigma,\uple{e})$ is
\textbf{focusing} if $\ube{3}=\cdots=\ube{6}=0$.
\par
It follows by \eqref{eq_before_unitary}, again for $X>2p^{N-1}$, that
we have
\begin{multline}\label{eq_generic_even}
  \Sigma_1=\delta_{2\mid\nu}\left(r_fH_{f,f^\ast}(1)\right)^{\nu/2}\sum_{\uple{e}\in\{\pm
    1\}^{\nu}}\sum_{\substack{\sigma\in P(\nu,\nu/2)
      \\\text{focusing}}} \Bigl(
  \mathcal{M}\left[|\bfnm|^2\right](1)\Bigr)^{\ube{1}}
  \Bigl(  \mathcal{M}\left[|\bfnm|^2\right](1)\Bigr)^{\ube{2}} \\
  +O_{\epsilon,
    f}\left(Y^{-1/2+\theta+\epsilon}+\frac{1}{\sqrt{p}}Y^{\nu/2+\epsilon}\right)
\end{multline}
(we also used the properties $r_f=r_{f^\ast}$ and
$H_{f,f^\ast}(1)=H_{f^\ast,f}(1)$.)
\par
As we did in the case of $N$ odd, we determine in a lemma, similar to
Lemma~\ref{lm-focusing-pairs}, the focusing pairs.

\begin{lemma}\label{lm-focusing-pairs2}
  With notation as above, for any $(\sigma,\uple{e})$ which is a
  focusing pair, we have $\kappa=\lambda=\nu/2$. Furthermore, the map
$$
(\sigma,\uple{e})\mapsto (\tilde{\sigma},\tilde{\uple{e}}),
$$
where $\tilde{\sigma}$ is the restriction of $\sigma$ to
$\{\kappa+1,\ldots,\nu\}$ and $\tilde{\uple{e}}$ is the $\kappa$-tuple
$(e_u)_{1\leq u\leq \kappa}$, is a bijection between the set of
focusing pairs $(\sigma,\uple{e})$ and the product set
$S_{\kappa}\times \{\pm 1\}^{\kappa}$, where $S_{\kappa}$ is the set
of bijections from $\{\kappa+1,\ldots,\nu\}$ to $\{1,\ldots,\kappa\}$.
\par
We then have, for all such $(\sigma,\uple{e})$, the relation
\begin{equation}\label{eq-ube2}
  \ube{1}=|\{u\ ,\  1\leq u\leq \kappa\text{ and } \tilde{e}_u=-1\}|,
\end{equation}
and $\ube{2}=\kappa-\ube{1}$. 
\end{lemma} 

As in the earlier case, the point is that $\ube{1}$ and $\ube{2}$ are,
for every focusing pair $(\sigma,\uple{e})$, independent of
$\sigma$. We will denote $\ube{1}$ by $\uote$.

\begin{proof}[\proofname{} of lemma \ref{lm-focusing-pairs2}]
  The focusing condition on $(\sigma,\uple{e})$ means that, for every
  $u$ with $1\leq u\leq \kappa$, either $\bose=\omega_1$ or
  $\bose=\omega_2$. By definition, the condition $\bose=\omega_1$ is
  equivalent with the property that $u$ has exactly one pre-image $j$
  under $\sigma$ with $1\leq j\leq \kappa$, and one pre-image $\ell$
  with $\kappa+1\leq\ell\leq \nu$, and that furthermore
  $e_j=e_{\ell}=-1$.
\par
Similarly, $\bose=\omega_2$ means that that $u$ has exactly one
pre-image $j$ under $\sigma$ with $1\leq j\leq \kappa$, and one
pre-image $\ell$ with $\kappa+1\leq\ell\leq \nu$, and that furthermore
$e_j=e_{\ell}=-1$.
\par
Arguing as in the proof of Lemma~\ref{lm-focusing-pairs}, we deduce
that $\kappa=\lambda=\nu/2$ and that the restriction of $\sigma$ to
$\{1,\ldots,\kappa\}$ is the identity, and the restriction
$\tilde{\sigma}$ of $\sigma$ to $\{\kappa+1,\ldots,\nu\}$ is an
element of $S_{\kappa}$.
\par
In addition, the signs $e_j$ for the two pre-images of $u$ always
coincide, which means that $e_{\tilde{\sigma}^{-1}(u)}=e_u$ for $1\leq
u\leq \kappa$. Thus the map $(\sigma,\uple{e})\mapsto
(\tilde{\sigma},\tilde{\uple{e}})$ is injective. We then check as in
Lemma~\ref{lm-focusing-pairs} that it is surjective, and that the
formula~(\ref{eq-ube2}) holds.
\end{proof}


Using this lemma, it follows that the main term in \eqref{eq_generic}
equals
$$
  \delta_{\kappa=\lambda}\left(r_fH_{f,f^\ast}(1)\right)^{\kappa}
  \sum_{\tilde{\sigma}\in S_{\kappa}}
  \sum_{\tilde{\uple{e}}\in\{\pm 1\}^\kappa}
\left(\mathcal{M}\left[|\bfnm|^2\right](1)\right)^{\uote} 
  \times\left(\mathcal{M}\left[|\bfnp|^2\right](1)\right)^{\kappa-\uote},
$$
which is equal to 
\begin{equation*} 
  \delta_{\kappa=\lambda}\kappa!\left(r_fH_{f,f^\ast}(1)
    \left(\mathcal{M}\left[|\bfnm|^2\right](1)
      +\mathcal{M}\left[|\bfnp|^2\right](1)\right)\right)^{\kappa},
\end{equation*}
because $\uote$ depends only on $\tilde{\uple{e}}$.
\par
By Corollary \ref{coro_unitary_Bessel}, this is equal to
\begin{equation*}
\delta_{\kappa=\lambda}\kappa!\left(r_fH_{f,f^\ast}(1)\abs{\abs{w}}_2^2\right)^{\kappa}
\end{equation*}
and we conclude that, if $X>2p^{N-1}$, then we have
\begin{equation}\label{eq_Sigma_1_MT_even}
  \Sigma_1=\delta_{\kappa=\lambda}2^\kappa\kappa!\left(\frac{r_fH_{f,f^\ast}(1)}{2}\abs{\abs{w}}_2^2\right)^{\kappa} \\
  +O_{\epsilon,f}\left(Y^{-1/2+\theta+\epsilon}+\frac{1}{\sqrt{p}}Y^{\nu/2+\epsilon}\right).
\end{equation}
\subsubsection{The self-dual case for $N$ even}\label{sec_variance_4}%

In this section, $f$ is \textbf{self-dual}, namely $f^\ast=f$ and $N$
is even. Note that this corresponds formally to the case treated
in~\cite{FoGaKoMi} of holomorphic cusp forms with trivial nebentypus
for $N=2$ (although, as we have already discussed, the restriction to
holomorphic forms means that the cases we consider are disjoint.)
\par
In this case, \eqref{eq_before_variance_even} and (once more)
Proposition \ref{propo_variance}
lead to
\newcommand{\vbe}{v(\uple{e})}
\begin{multline}\label{eq_generic_even_2}
  \Sigma_1=\delta_{2\mid\nu}\left(r_fH_{f,f^\ast}(1)\right)^{\nu/2}
\sum_{\uple{e}\in\{\pm 1\}^{\nu}}\sum_{\substack{\sigma\in P(\nu,\nu/2) \\\text{focusing}}} 
  \left(\mathcal{M}\left[|\bfnm|^2\right](1)\right)^{\vbe}
  \left(\mathcal{M}\left[|\bfnp|^2\right](1)\right)^{\nu/2-\vbe} \\
  +O_{\epsilon,
    f}\left(Y^{-1/2+\theta+\epsilon}+\frac{1}{\sqrt{p}}Y^{\nu/2+\epsilon}\right)
\end{multline}
if $2p^{N-1}<X$, where
$$
\vbe=| \left\{1\leq u\leq\nu/2,
  \left(\ose{-1},\ose{1}\right)=(2,0)\right\}|
$$
so that 
$$
\nu/2-\vbe=|\left\{1\leq u\leq\nu/2,
  \left(\ose{-1},\ose{1}\right)=(0,2)\right\}|.
$$
\par
Note that $\vbe$ depends on $\sigma$. To go further, we observe that
if $\sigma\in P(\nu,\nu/2)$ occurs in a focusing pair, it must satisfy
$$
\os=|\sigma^{-1}(u)|=2
$$
for all $u\in \{1,\ldots,\nu/2\}$. Conversely, assume $\sigma$
satisfies this condition. Then from the definition of $\omega_1$ and
$\omega_2$, it follows that a tuple $\uple{e}$ is such that
$(\sigma,\uple{e})$ is focusing if and only if, for each $u$, we have
$e_m=e_n$, where $\sigma^{-1}(u)=\{m,n\}$. This means that there are
precisely $2^{\nu/2}$ focusing pairs $(\sigma,\uple{e})$ with $\sigma$
fixed, corresponding to arbitrary assignments of signs to the $\nu/2$
pairs of elements with the same image under $\sigma$. 
\par
In this context, $\vbe$ is equal to the number of $u$ for which the
corresponding sign $e_m=e_n$ is $-1$, and in particular, for any $r$,
the number of tuples $\uple{e}$ for which $\vbe=r$ is equal to
$$
\binom{\nu/2}{r},
$$
corresponding to the choice of $r$ pairs of elements with common sign
$-1$.
\par
Formally, it follows that for any complex numbers $z_1$ and $z_2$,
we have
\begin{align*}
\sum_{\uple{e}\in\{\pm 1\}^{\nu}} \sum_{\substack{\sigma\in P(\nu,\nu/2)\\\text{focusing}}}z_1^{\vbe}z_2^{\nu/2-\vbe}
&=\sum_{\substack{\sigma\in P(\nu,\nu/2)\\\os=2}}
\sum_{\substack{\uple{e}\\(\sigma,\uple{e})\text{
      focusing}}}z_1^{\vbe}z_2^{\nu/2-\vbe}\\
&=\sum_{\substack{\sigma\in P(\nu,\nu/2)\\\os=2}}
\sum_{r=0}^{\nu/2}\binom{\nu/2}{r}
z_1^{r}z_2^{\nu/2-r}\\
&=(z_1+z_2)^{\nu/2}|\{\sigma\in P(\nu,\nu/2)\,\ \os=2\text{ for all }
u\}|\\
&=\frac{\nu!}{2^{\nu/2}(\nu/2)!}(z_1+z_2)^{\nu/2},
\end{align*}
(where the last step follows from~(\ref{eq-match}), which is
established in the proof of Lemma~\ref{lm-resonant-pairs}).
\par
Applying this formula and using Corollary \ref{coro_unitary_Bessel},
we derive
\begin{equation}\label{eq_Sigma_1_MT_2_even}
\Sigma_1=\delta_{2\mid\nu}\left(r_fH_{f,f^\ast}(1)\abs{\abs{w}}_2^2\right)^{\nu/2}
\frac{\nu!}{2^{\nu/2}(\nu/2)!}  +O_{\epsilon,
  f}\left(Y^{-1/2+\theta+\epsilon}+\frac{1}{\sqrt{p}}Y^{\nu/2+\epsilon}\right)
\end{equation}
if $2p^{N-1}<X$. 

\subsubsection{End of the proof of Theorem \ref{theo_moments} for $N$ even}%
Equations \eqref{eq_bound_sigma2} and \eqref{eq_Sigma_1_MT_2_even} imply Theorem \ref{theo_moments} if $f$ is self-dual.
\par
Equations \eqref{eq_bound_sigma2} and \eqref{eq_Sigma_1_MT_even} imply
Theorem \ref{theo_moments} if $\kappa\lambda\neq 0$ and $f$ is not
self-dual. It is once more easy to check the result when $\kappa$ or
$\lambda=0$. For instance, if $\lambda=0$ (and $f$ is not self-dual),
understanding $\mathsf{M}_f(X,p,(\kappa,0))$ boils down to
understanding
$$
  \delta_{2\mid\kappa}\sum_{\uple{e}\in\{\pm 1\}^\kappa}\sum_{\substack{\sigma\in P(\kappa,\kappa/2) \\\text{resonant}}}\Sigma_{1}(\sigma,\uple{e})  +O_{\epsilon,
    f}\left(Y^{-1/2+\epsilon}+\frac{1}{\sqrt{p}}Y^{\kappa/2+\epsilon}\right)
$$
where $\Sigma_1(\sigma,\uple{e})$ is given by
$$
\left(\frac{1}{Y}\sum_{1\leq m<p/2}
\afs(m)^2\bfn\left(\frac{m}{Y}\right)^2\right)^{\vbe} \\
\left(\frac{1}{Y}\sum_{1\leq m<p/2}\afs(m)^2
  \bfn\left(\frac{-m}{Y}\right)^2\right)^{\nu/2-\vbe}.
$$
\par
These terms are all dominated by the error term, by
Proposition~\ref{propo_variance} (this is because $L(f\times f,s)$
does not have a pole at $s=1$ if $f$ is not self-dual).


\section{Proof of the convergences in law}\label{sec_8}%

This section is devoted to the proof of Corollary
\ref{coro_law}. Thus, $X=p^N/\Phi(p)$ for a function $\Phi$ that tends
to infinity but satisfies $\Phi(x)\ll x^{\epsilon}$ for all $\epsilon>0$.
\subsection{The non self-dual case}%

In this section, we assume that $f$ is \textbf{not self-dual}. In
order to finish the proof of Corollary~\ref{coro_law}, it is enough to
apply the following probabilistic lemma to $Z_f(X,p,\ast)$.

\begin{lemma}\label{lemma_standard_proba}
Let $(X_n)_{n\geq 1}$ be a sequence of complex-valued random variables, let
$\sigma>0$ be a positive real number. Then $(X_n)_{n \geq 1}$ converges in law to a Gaussian vector with covariance matrix
$$
\begin{pmatrix}
\sigma&0\\
0&\sigma
\end{pmatrix}
$$
if and only if, for any non-negative integers $\kappa, \lambda\geq 0$,
we have
$$
\lim_{n\rightarrow +\infty}
\mathbf{E}\left(X_n^{\kappa}\bar{X}_n^{\lambda}\right)=\delta_{\kappa,\lambda}2^{\kappa}
\kappa!\sigma^{k}.
$$
\end{lemma}

\begin{proof}[\proofname{} of lemma \ref{lemma_standard_proba}]%
This is presumably standard, but we give a quick proof for lack of a suitable reference.
\par
The necessity follows by an easy argument from the fact that for a Gaussian variable $Z$ with the stated covariance matrix, we have
$$
\mathbf{E}(Z^{\kappa}\bar{Z}^{\lambda})=\delta_{\kappa,\lambda}2^{\kappa}
\kappa!\sigma^{k}
$$
(this is straightfoward since it can be evaluated using the explicit
density of $Z$ with respect to Lebesgue measure; after checking that
this is non-zero if and only if $\kappa=\lambda$, one can also notice
that $\mathbf{E}(|Z|^{2\kappa})$ is the $2k$-th moment of a so-called
\emph{Rayleigh distribution} with parameter $\sigma$, and check its
value in any table of probability distributions.)
\par
For sufficiency, write $X_n=A_n+iB_n$ where the random variables $A_n$
and $B_n$ are real-valued. By a well-known result (see, e.g.,
\cite[Lemma 5.1]{FoGaKoMi}), convergence to the Gaussian holds
provided, for any $k$, $l\geq 0$, we have
$$
\mathbf{E}(A_n^kB_n^l)\rightarrow \sigma^{(k+l)/2}m_km_l.
$$
But, denoting
$M(\kappa,\lambda)=\mathbf{E}(X_n^{\kappa}\bar{X}_n^{\lambda})$, we
have
$$
\mathbf{E}(A_n^kB_n^l)=\frac{1}{2^{\nu}i^\lambda}
\sum_{\substack{0\leq k\leq\kappa \\
    0\leq\ell\leq\lambda}}(-1)^\ell\binom{\kappa}{k}\binom{\lambda}{\ell}
M(\nu-k-l,k+l)
$$
since, as recalled above, the assumption means that
$$
M(\nu-k-l,k+l)\rightarrow \mathbf{E}(Z^{\kappa}\bar{Z}^{\lambda})
$$
where $Z$ is the Gaussian as above. Denoting $Z=A+iB$, we deduce that
$$
\mathbf{E}(A_n^kB_n^l)\rightarrow \frac{1}{2^{\nu}i^\lambda}
\sum_{\substack{0\leq k\leq\kappa \\
    0\leq\ell\leq\lambda}}(-1)^\ell\binom{\kappa}{k}\binom{\lambda}{\ell}
\mathbf{E}(Z^{\kappa}\bar{Z}^{\lambda})
=\mathbf{E}(A^kB^l)=\sigma^{(k+l)/2}m_km_l
$$
(by rewinding the first formula).
\end{proof}

\begin{remark}
  This result, easy as it is, implies the following combinatorial
  identities, by writing all expectations of Gaussians in
  ``numerical'' terms: for any integers $\kappa$, $\lambda\geq 0$
  satisfying $2\mid\nu$, we have
$$
\sum_{\substack{0\leq k\leq\kappa \\
    0\leq\ell\leq\lambda \\
    k+\ell=\nu/2}}(-1)^{\ell}\binom{\kappa}{k}\binom{\lambda}{\ell}
=\delta_{\substack{2\mid\kappa \\
    2\mid\lambda}}(-1)^{\delta/2}\frac{\binom{\nu/2}{\delta/2}\binom{\nu}{\nu/2}}{\binom{\nu}{\delta}}
$$
where $\delta=\min{(\kappa,\lambda)}$. These identities are not so
easy to establish directly and are just stated without proof in \cite[Formulas (3.58) and (3.80)]{MR0354401}.
\end{remark}

\subsection{The self-dual case}%

In this section, we assume that the cusp form $f$ is
\textbf{self-dual}. The $k$-th moment of $E_f(X,p,\ast)$ is
$\mathsf{M}_f(X,p,(k,0))$ for all non-negative integer $k$. By Theorem
\ref{theo_moments}, we get
\begin{equation*}
\mathsf{M}_f(X,p,(k,0))=m_k\left(2c_{f,w}\right)^{k/2}+O_{\epsilon, f}\left(\frac{1}{\Phi(p)^{1/2-\theta+\epsilon}}\right)
\end{equation*}
such that
\begin{equation*}
\lim_{\substack{p\in\prem \\
p\to+\infty}}\mathsf{M}_f(X,p,(k,0))=m_k\left(2c_{f,w}\right)^{k/2}.
\end{equation*}
By standard results, convergence to a centered Gaussian random
variable is equivalent to convergence of the moments.  Hence
the sequence of random variables $E_f(X,p,\ast)$ converges in law to a
centered Gaussian random variable with variance is $2c_{f,w}$.
\section{The case of the multiple divisor functions}\label{sec_9}%
\label{sec-dn}

In this section, we will give a sketch of the proof of Theorem
\ref{theo_moments_d_N}, which is very similar to the self-dual case of
Theorem~\ref{theo_moments}, the additional ingredient being the
presence of main terms arising from the positivity of the divisor
functions.
\par
We begin by stating the corresponding version of the Vornon\u{\i}
summation formula.  A. Ivi\'{c} proved such a formula for $d_N$, when
$N\geq 3$, in \cite[Theorem 2]{MR1635762}. The following statement is
both a simplified (but not straightforward) statement for prime
denominators and a slightly renormalised version of this formula.
\par
We note that we could use a less precise version, as far as
understanding the main term is concerned, but we give the full version
as it might be potentially useful for other purposes.
\par
We will need for this the constants $\gamma_n(\alpha)$ defined by
$\gamma_{-1}(\alpha)=1$ and
\begin{equation*}
\gamma_{n}(\alpha)=\frac{(-1)^n}{n!}\lim_{m\to +\infty}\left(\sum_{k=0}^m\frac{\log^n{(k+\alpha)}}{k+\alpha}-\frac{\log^{n+1}{(m+\alpha)}}{n+1}\right)
\end{equation*}
for $n\geq 0$ and $0\leq\alpha\leq 1$. For $\alpha=0$, these are the
Stieltjes numbers $\gamma_n$ for $n\geq 0$, for instance
$\gamma_0=\gamma$ is the Euler-Mascheroni constant. These numbers
occur in the Laurent expansion of $\zeta(s)$ at $s=1$.
\par
We will denote by $\mathcal{B}[w]$ the Mellin transform
$\mathcal{B}_{\uple{0}}[w]$ as in~(\ref{eq_CL_2}) for
$\uple{\alpha}=\uple{0}=(0,\ldots, 0)$. Note that
$\uple{0}^{\ast}=\uple{0}$. We also extend $d_N$ to non-zero integers
by defining $d_N(m)=d_N(|m|)$ if $m\leq -1$.

\begin{proposition}[Vorono\u{\i} summation formula for
  $d_N$]\label{propo_voronoi_non_cusp}
  Let $N\geq 2$ be an integer.  Let $w:\R_+^\ast\to\R$ be a smooth and
  compactly supported function. Let $p$ be a prime number and let $b$
  be an integer. If $p$ does not divide $b$ then
\begin{multline}\label{eq_voronoi_non_cusp}
  \sum_{n\geq 1}d_N(n)e\left(\frac{bn}{p}\right)w(n)=
\frac{1}{p}\int_{x=0}^{+\infty}\sum_{k=1}^N\frac{\beta_k(p)}{(k-1)!}\log^{k-1}{(x)}w(x)\dd x \\
  +\frac{1}{p^{N/2}}\sum_{m\in\Z^\ast}d_N(m)K_{N-1}(\bar{b}m,p)\mathcal{B}[w]
  \left(\frac{m}{p^N}\right) \\
  +\sum_{\ell=1}^{N-2}\binom{N-1}{\ell}\frac{(-1)^{\ell+1}}{p^\ell}\sum_{m\in\Z^\ast}
d_N(m)\mathcal{B}[w]\left(\frac{m}{p^\ell}\right)
\end{multline}
where $\bar{b}$ denotes the inverse of $b$ modulo $p$ and
\begin{multline}\label{eq_coeff_beta}
\beta_k(p)=\frac{1}{p^{N-1}}\sum_{1\leq a_1,\dots,a_N\leq p}e\left(\frac{ba_1\dots a_N}{p}\right) \\
\times\sum_{m=0}^{N-k}\frac{(-1)^mr^m}{m!}\sum_{\substack{n_1,\dots,n_N\geq -1 \\
n_1+\dots n_N=-k-m}}\prod_{j=1}^N\gamma_{n_j}\left(\frac{n_j}{p}\right)\log^m{(p)}
\end{multline}
for $1\leq k\leq N$. 
\end{proposition}

\begin{proof}[\proofname{} of proposition \ref{propo_voronoi_non_cusp}]%
  We use the notation of A. Ivi\'{c} in~\cite[Theorem 2]{MR1635762}. Note
  that A. Ivi\'{c} considers the case $N=3$ in \cite[Page
  211]{MR1635762}, but there are a number of typos in this argument.
\par
First, we compute explicitly
\begin{equation*}
\text{res}_{s=1}\mathcal{M}[w](s)E_N\left(s,\frac{b}{p}\right).
\end{equation*}
\par
By \cite[Equations (2.2) and (2.3)]{MR1635762}, we have
\begin{equation}\label{eq_expansion}
  E_N\left(s,\frac{b}{p}\right)=\frac{1}{p}\sum_{k=1}^N\frac{\beta_k(b,p)}{(s-1)^k}+H(s)
\end{equation}
where $H(s)$ is an entire function on $\C$ and $\beta_k(b,p)$ is
defined in \eqref{eq_coeff_beta} for $1\leq k\leq N$. These
coefficients do not depend on $b$ for the following reason. Let us fix
$0\leq m\leq N-k$ and $n_1,\dots,n_N\geq -1$ satisfying
$n_1+\dots+n_N=-k-m$. Obviously, there exists at least one $1\leq
j_0\leq N$ such that $n_{j_0}=-1$, for which
$\gamma_{n_{j_0}}(a_{j_0}/p)=1$. Performing the summation over
$a_{j_0}$ in \eqref{eq_coeff_beta}, one gets
\begin{equation*}
\sum_{1\leq a_1,\dots,a_N\leq p}e\left(\frac{ba_1\dots a_N}{p}\right)\prod_{\substack{1\leq j\leq N \\
j\neq j_0}}\gamma_{n_j}\left(\frac{n_j}{p}\right)=p\sum_{\substack{1\leq a_1,\dots a_{j_0-1},a_{j_0+1},\dots,a_N\leq p \\
p\mid a_1\dots a_{j_0-1}a_{j_0+1}\dots a_N}}\prod_{\substack{1\leq j\leq N \\
j\neq j_0}}\gamma_{n_j}\left(\frac{n_j}{p}\right).
\end{equation*}
This last equality also implies that
\begin{equation*}
\beta_k(b,p)=\beta_k(p)\ll_\epsilon p^\epsilon
\end{equation*}
for $1\leq k\leq N$ and for all $\epsilon>0$. Finally, \eqref{eq_expansion} also implies that this residue equals the first term in \eqref{eq_voronoi_non_cusp}.
\par
Then, let us compute explicitly
\begin{eqnarray*}
A_N\left(n,\frac{b}{p}\right) & = & \frac{1}{2}\left(C_N^{+}\left(n,\frac{b}{p}\right)+C_N^{-}\left(n,\frac{b}{p}\right)\right), \\
B_N\left(n,\frac{b}{p}\right) & = & \frac{1}{2}\left(C_N^{+}\left(n,\frac{b}{p}\right)-C_N^{-}\left(n,\frac{b}{p}\right)\right)
\end{eqnarray*}
for all positive integers $n$, where
\begin{eqnarray*}
C_N^{\pm}\left(n,\frac{b}{p}\right) & = & \sum_{n_1\dots n_N=n}C_N^{\pm}\left(\uple{n},\frac{b}{p}\right), \\
C_N^{\pm}\left(\uple{n},\frac{b}{p}\right) & \coloneqq & \sum_{x_1,\dots, x_N\bmod{p}}e\left(\frac{n_1x_1+\dots n_Nx_N\pm bx_1\dots x_N}{p}\right)
\end{eqnarray*}
where $\uple{n}=(n_1,\dots,n_N)$. Let us fix $n_1,\dots,n_N$ satisfying $n=n_1\dots n_N$.
\par
If $p\nmid n$, then we find
\begin{align}
C_N^{\pm}\left(\uple{n},\frac{b}{p}\right) & =p\sum_{\substack{x_2,\dots,x_N\bmod{p} \\
x_2\dots x_N\equiv\mp\bar{b}n_1\bmod{p}}}e\left(\frac{n_1x_1+\dots n_Nx_N\pm bx_1\dots x_N}{p}\right) \\
& =p^{N/2}K_{N-1}(\mp\bar{b}n,p).  \label{eq_second_term}
\end{align}
\par
On the other hand, if $p\mid n$, then there exists $1\leq k_0\leq N$
such that $p$ divides $n_{k_0}$. Thus,
\begin{align}
C_N^{\pm}\left(\uple{n},\frac{b}{p}\right) & =p\sum_{\substack{x_1,\dots,\widehat{x_{k_0}},\dots,x_N\bmod{p} \\
p\mid x_1\dots\widehat{x_{k_0}}\dots x_N}}e\left(\frac{n_1x_1+\dots+\widehat{n_{k_0}x_{k_0}}+\dots n_Nx_N}{p}\right) \\
& =(-1)^{N-2}p+\sum_{\ell=1}^{N-2}(-1)^{\ell+N-2}p^{\ell+1}\sum_{\substack{1\leq k_1<\dots<k_\ell\leq N \\
\forall 1\leq i\leq \ell, k_i\neq k_0}}\prod_{j=1}^{\ell}\delta_{p\mid n_{k_j}}  \label{eq_third_term}
\end{align}
by a simple induction on $N$, using the notation
$\;\widehat{\cdot}\;$ as usual to omit a term.
\par
The contribution of \eqref{eq_second_term} and of the first term in \eqref{eq_third_term} leads to the second term in \eqref{eq_voronoi_non_cusp}, after a suitable renormalisation of the integral transforms given in \cite[Equations (3.9) and (3.10)]{MR1635762}. The contribution of the other terms in \eqref{eq_third_term} leads to the third term in \eqref{eq_voronoi_non_cusp}.
\end{proof}

We recall from Section~\ref{ssec-dn} that for an invertible residue
class $a$ in $\mathbb{F}_p^\times$, we have
$$
E_{d_N}(X,p,a)=\frac{S_{d_N}(X,p,a)-M_{d_N}(X,p)}{(X/p)^{1/2}},
$$ 
where $S_{d_N}(X,p,a)$ and $M_{d_N}(X,p)$ are defined
in~(\ref{eq-def-dn1}) and~(\ref{eq-def-dn2}).
\par
For $\kappa\geq 1$, we consider
\begin{equation*}
  \mathsf{M}_{d_N}(X,p,\kappa)=\frac{1}{p}\sums_{a\bmod{p}}E_{d_N}(X,p,a)^{\kappa}.
\end{equation*}
\par
As in the case of cusp forms, we denote $Y=X/p^N$. Then, detecting the
congruence $n\equiv a\bmod{p}$ using additive characters and applying
the Vorono\u{\i} summation formula for $d_N$ (Proposition
\ref{propo_voronoi_non_cusp}), we get
\begin{multline}\label{eq_d_N_after_voro}
E_{d_N}(X,p,a)=\frac{1}{\sqrt{Y}}\sum_{m\in\Z^\ast}d_N(m)K_N(-am,p)\mathcal{B}[w]\left(\frac{m}{Y}\right) \\
+\sum_{\ell=1}^{N-2}\binom{N-1}{\ell}\frac{(-1)^{\ell}}{p^{(\ell+1)/2}\sqrt{X/p^\ell}}\sum_{m\in\Z^\ast}d_N(m)\mathcal{B}[w]\left(\frac{m}{X/p^\ell}\right).
\end{multline}
\par
The second term in \eqref{eq_d_N_after_voro} is then seen to be $\ll
p^{-1/2}$. Thus, we have
\begin{multline*}
\mathsf{M}_{d_N}(X,p,\kappa)=\frac{1}{Y^{\kappa/2}}\frac{1}{p}
\sums_{a\bmod{p}}\sum_{1\leq\abs{m_1},\dots,\abs{m_\kappa}<p/2}\prod_{k=1}^\kappa d_N\left(m_k\right)K_N\left(am_k,p\right)\mathcal{B}[w]\left(\frac{m_k}{Y}\right) \\
+O_{\epsilon}\left(\frac{p^\epsilon}{\sqrt{p}}Y^{\kappa/2}+p^{1+\epsilon}\left(\frac{p^{N-1}}{X}\right)^AY^{\kappa/2-1+\epsilon}\right)
\end{multline*}
if $2p^{N-1}<X$, for all $A>1$, by the decay properties of the
generalized Bessel transforms.
\par
Using again the combinatorial identity in Lemma \ref{lemma_combi}, we
rearrange this into
\begin{multline*}
  \mathsf{M}_{d_N}(X,p,\kappa)=\frac{1}{Y^{\kappa/2}}\sum_{\uple{e}\in\left\{\pm
      1\right\}^\kappa}\sum_{s=1}^\kappa\sum_{\sigma\in
    P(\kappa,s)}\sum_{\substack{(j_1,\dots,j_s)\in R^s \\
      \text{distinct}}} \prod_{u=1}^sd_N\left(j_u\right)^{\os}
  \\
  \times
  \mathcal{B}[w]\left(\frac{j_u}{Y}\right)^{\ose{1}}\mathcal{B}[w]\left(\frac{-j_u}{Y}\right)^{\ose{-1}}
  \Bigl(
  \frac{1}{p}\sums_{a\bmod{p}}\prod_{u=1}^sK_N\left(aj_u,p\right)^{\ose{1}}K_N\left(-aj_u,p\right)^{\ose{-1}} \Bigr)\\
  +O_{\epsilon}\left(\frac{p^\epsilon}{\sqrt{p}}Y^{\kappa/2}+p^{1+\epsilon}\left(\frac{p^{N-1}}{X}\right)^AY^{\kappa/2-1+\epsilon}\right),
\end{multline*}
where we use the same notation as in Section~\ref{ssec-reduce-katz}.
\subsection{The combinatorial analysis for $N$ odd}%

Arguing precisely along the same lines as in Sections
\ref{sec_combi_odd} and~\ref{sec_variance_2} (the self-dual, $N$
odd, case), we obtain
\begin{multline*}
  \mathsf{M}_{d_N}(X,p,\kappa)=\delta_{2\mid\kappa}
  \frac{\kappa!}{(\kappa/2)!}
  \left(\frac{1}{Y}\sum_{1\leq
      m<p/2}d_N(m)^2\left\vert\mathcal{B}[w]\left(\frac{m}{Y}\right)\right\vert^2\right)^{\kappa/2}\\
  +O_{\epsilon}\left(\frac{p^\epsilon}{\sqrt{p}}Y^{\kappa/2}+Y^{-1/2+\epsilon}+\frac{1}{\sqrt{p}}Y^{\kappa/2+\epsilon}\right)
\end{multline*}
if $2p^{N-1}<X$. 
Using Proposition \ref{propo_variance_d_N} below and 
$$
\mathcal{M}\left[\left\vert\mathcal{B}[w]\right\vert^2\right](1)
=\frac{\abs{\abs{w}}_2^2}{2} 
$$
(by~(\ref{eq_mellin_odd})), we derive Theorem~\ref{theo_moments_d_N}
for $N$ odd.
\subsection{The combinatorial analysis for $N$ even}%

Arguing precisely along the same lines as in Section
\ref{sec_even_combi} and Section~\ref{sec_variance_4} (the self-dual,
$N$ even, case, with $\lambda=0$ so that $\nu=\kappa$), we get
\begin{multline*}
  \mathsf{M}_{d_N}(X,p,\kappa)=\delta_{2\mid\kappa}\sum_{\uple{e}\in\left\{\pm
      1\right\}^\kappa}\sum_{\substack{\sigma\in P(\kappa,\kappa/2)
      \\\text{focusing}}} \left(\frac{1}{Y}\sum_{1\leq
      m<p/2}d_N(m)^2\mathcal{B}[w]\left(\frac{m}{Y}\right)^2\right)^{\vbe}
  \\
  \times \left(\frac{1}{Y} \sum_{1\leq
      m<p/2}d_N(m)^2\mathcal{B}[w]\left(\frac{-m}{Y}\right)^2\right)^{\kappa/2-\vbe}
  +O_{\epsilon}\left(\frac{p^\epsilon}{\sqrt{p}}Y^{\kappa/2}+Y^{-1/2+\epsilon}+\frac{1}{\sqrt{p}}Y^{\kappa/2+\epsilon}\right)
\end{multline*}
if $2p^{N-1}<X$ where
\begin{equation*}
  \vbe=|\{1\leq u\leq\kappa/2, \left(\ose{-1},\ose{1}\right)=(0,2)\}|
\end{equation*}
\par
Applying twice Proposition \ref{propo_variance_d_N} and
\eqref{eq_mellin_even}, we derive Theorem~\ref{theo_moments_d_N} in
the case $N$ even.
\subsection{Asymptotic expansion}%

This section is the analogue of Section~\ref{sec_variance} for the
divisor functions.
For a smooth function $\mathrm{B}$ and $Y<Z$, let
$$
W(Y,Z)=\frac{1}{Y}\sum_{1\leq  m<Z}
d_N(m)^2\mathrm{B}\Bigl(\frac{m}{Y}\Bigr).
$$

\begin{proposition}\label{propo_variance_d_N}
  We have
\begin{equation*}
  W(Y,Z)=Q(\log{(Y)})+O_{\epsilon, N}
  \left(Z^{\epsilon}\left(\frac{Y}{Z}\right)^A+Y^{-1/2+\epsilon}\right)
\end{equation*}
for all $A>1$, where $Q$ is a polynomial of degree $N^2-1$ given by
\begin{equation}\label{eq_Q_N}
  Q(X)=\sum_{m=0}^{N^2-1}\frac{1}{m!}
  \left(\sum_{\substack{m_1+\dots+m_{N^2}+k+\ell=-1-m \\
        m_1,\dots,m_{N^2}\geq -1 \\
        k,\ell\geq 0}}
    \prod_{j=1}^{N^2}\gamma_{m_j}\frac{H_N^{(k)}(1)\mathcal{M}
      \left[\mathrm{B}\right]^{(\ell)}(1)}{k!\ell!}\right)X^m,
\end{equation}
\begin{equation*}
H_N(s)=\prod_{q\in\mathcal{P}}\left(1-q^{-s}\right)^{(N-1)^2}P_N(q^{-s})
\end{equation*}
and the polynomial $P_N(T)$ is defined in Proposition \ref{propo_d_N_series}. In particular, the leading coefficient of $Q_{N}^{\uple{\epsilon}}[w](X)$ equals
\begin{equation*}
  \frac{H_N(1)\|\mathrm{B}\|^2}{(N^2-1)!}.
\end{equation*}
\par
Moreover we have $H_N(1)>0$.
\end{proposition}

\begin{proof}[\proofname{} of proposition \ref{propo_variance}]%
Arguing as in the proof of Proposition \ref{propo_variance}, one gets
\begin{equation}\label{eq_before_shift_d_N}
  W(Y,Z)=
  \frac{1}{Y}\frac{1}{2i\pi}
  \int_{(3)}D_{d_N}(s)
  Y^s\mathcal{M}\left[\mathrm{B}\right](s)\dd s+
  O_{\epsilon, N}\left(Z^{\epsilon}\left(\frac{Y}{Z}\right)^{A}\right)
\end{equation}
for all $A>1$, where
\begin{equation*}
  D_{d_N}(s)\coloneqq
  \sum_{m\geq 1}\frac{d_N(m)^2}{m^s}=
  \prod_{q\in\mathcal{P}}\frac{P_N(q^{-s})}{(1-q^{-s})^{2N-1}}=\zeta(s)^{N^2}H_N(s)
\end{equation*}
defines a meromorphic function on $\Re{(s)}>1/2$ with a pole at $s=1$
of order $N^2$ by Proposition \ref{propo_d_N_series}. The proposition
follows from \eqref{eq_before_shift_d_N} by shifting the contour to
$Re{(s)}=1/2+\epsilon$, hitting the pole at $s=1$.
\par
The fact that $H_N(1)>0$ is clear here, since $P_N(q^{-1})>0$ for all
prime numbers $q$.
\end{proof}
\appendix%
\section{Computation of the residue of Rankin-Selberg $L$-functions}\label{sec_A}%
A formula for the residue of the Rankin-Selberg $L$-function
$L(f\times f^\ast,s)$ in terms of the $L^2$-norm of $f$ is implicit,
but not fully stated in \cite{MR2254662}. For convenience, we give the
details of this computation.

\begin{proposition}\label{propo_residue}
  The residue of $\;L(f\times f^\ast,s)$ at $s=1$ is equal to
\begin{equation*}
r_f=\frac{4\pi^{N^2/2}}{N\fGamma_{\alpha_\infty(f)}}\abs{\abs{f}}^2>0
\end{equation*}
where
\begin{equation*}
\fGamma_{\alpha_\infty(f)}\coloneqq\prod_{1\leq j\leq N}\fGamma\left(\frac{1+2\Re{\left(\alpha_{j,\infty}(f)\right)}}{2}\right)\prod_{1\leq j<k\leq N}\left\vert\fGamma\left(\frac{1+\alpha_{j,\infty}(f)+\overline{\alpha_{k,\infty}(f)}}{2}\right)\right\vert^2
\end{equation*}
and the Petersson norm of $f$ is given by
\begin{equation*}
  \abs{\abs{f}}^2=\int_{SL_N(\Z)z\in
    SL_N(\Z)\setminus\mathbb{H}^N}\left\vert f(z)\right\vert^2\dd^\ast z, 
\end{equation*}
with $\dd^\ast z$ being the $SL_N(\R)$-invariant measure on
$\mathbb{H}^N\simeq SL_N(\R)/SO_N(\R)$ defined in \cite[Proposition
1.5.3]{MR2254662}.
\end{proposition}

\begin{proof}[\proofname{} of proposition \ref{propo_residue}]%
The last equation Page 369 in the proof of \cite[Theorem 12.1.4]{MR2254662} tells us that
\begin{equation}\label{eq_res_1}
\scal{f\bar{f}}{\pi^{-N\bar{s}/2}\fGamma(N\bar{s}/2)\zeta(N\bar{s})E_P(\ast,\bar{s})}=\pi^{-Ns/2}\fGamma(Ns/2)G_{\nu(f)}(s)L(f\times f^\ast,s)
\end{equation}
where $E_P(z,\bar{s})$ is the maximal parabolic Eisenstein series defined in \cite[Equation (10.7.1)]{MR2254662} and
\begin{equation*}
G_{\nu(f)}(s)=\int_{y_1,\dots,y_{N-1}=0}^{+\infty}\left\vert W_{\text{Jacquet}}(y,\nu(f),\psi_{1,\dots,1})\right\vert^2\prod_{j=1}^{N-1}y_j^{(N-j)s}\prod_{j=1}^{N-1}y_j^{-j(N-j)}\frac{\dd y_j}{y_j}
\end{equation*}
where $W_{\text{Jacquet}}$ stands for the Jacquet Whittaker function defined in \cite[Equation (5.5.1)]{MR2254662}. In particular,
\begin{align}
G_{\nu(f)}(1) & =\int_{y_1,\dots, y_{N-1}=0}^{+\infty}W_{\text{Jacquet}}(y,\nu(f),\psi_{1,\dots,1})W_{\text{Jacquet}}(y,\nu(f^\ast),\psi_{1,\dots,1}) \\
& \times \prod_{j=1}^{N-1}y_j^{N-j}\prod_{j=1}^{N-1}y_j^{-j(N-j)}\frac{\dd y_j}{(y_j} \\
& =\frac{1}{2\pi^{N(N-1)/2}\fGamma(N/2)}\prod_{1\leq j, k\leq N}\fGamma\left(\frac{1+\alpha_{j,\infty}(f)+\alpha_{k,\infty}(f^\ast)}{2}\right) \\
& =\frac{1}{2\pi^{N(N-1)/2}\fGamma(N/2)}\prod_{1\leq j\leq N}\fGamma\left(\frac{1+2\Re{\left(\alpha_{j,\infty}(f)\right)}}{2}\right) \\
& \times\prod_{1\leq j<k\leq N}\left\vert\fGamma\left(\frac{1+\alpha_{j,\infty}(f)+\overline{\alpha_{k,\infty}(f)}}{2}\right)\right\vert^2\label{eq_res_2}
\end{align}
by Stade's formula (\cite[Proposition 11.6.17]{MR2254662}) and \eqref{eq_unitaricity_2}. By \cite[Proposition 10.7.5]{MR2254662}, $s\mapsto\pi^{-Ns/2}\fGamma(Ns/2)\zeta(Ns)E_P(z,s)\coloneqq E_P^\ast(z,s)$ has a simple pole at $s=1$ but the accurate value of this residue is not computed. Let us show quickly that
\begin{equation*}
\text{res}_{s=1}E_P^\ast(z,s)=2/N
\end{equation*}
which concludes the proof. The last equation in the proof of \cite[Theorem 10.7.5]{MR2254662} tells us that
\begin{equation*}
E_P^\ast(z,s)=\det(z)^s\int_{u=0}^{+\infty}\left[\sum_{\uple{a}\in\Z^N}f_u(\uple{a}z)-1\right]u^{Ns/2}\frac{\dd u}{u}
\end{equation*}
where
\begin{eqnarray*}
f_u(\uple{x})\coloneqq e^{-\pi(x_1^2+\dots+x_N^2)u} & \rightsquigarrow & \widehat{f_u}(\uple{x})=u^{-N/2}f_{1/u}(\uple{x})
\end{eqnarray*}
for $u>0$. Breaking the $u$-integral into two parts $[0,1]$ and $[1,+\infty[$, changing the variable $u\mapsto 1/u$ in the second part and applying the Poisson summation formula given in \cite[Equation (10.7.2)]{MR2254662}, one gets
\begin{multline*}
E_P^\ast(z,s)=\det(z)^s\frac{-2/N}{s}+\det(z)^{s-1}\frac{2/N}{s-1}+\int_{u=1}^{+\infty}\left[\sum_{\uple{a}\in\Z^N}f_u(\uple{a}z)-1\right]u^{Ns/2}\frac{\dd u}{u} \\
+\int_{u=1}^{+\infty}\left[\sum_{\uple{a}\in\Z^N}f_u(\uple{a}\mathstrut^tz^{-1})-1\right]u^{N(1-s)/2}\frac{\dd u}{u}.
\end{multline*} 
\end{proof}
\section{Generating series involving Schur polynomials}\label{sec_B}%

Our goal here is to state a precise form of an identity involving
Schur polynomials that we used in Section~\ref{sec_variance}.

\begin{notations}
  The following notations will be used throughout this section. For
  $\uple{k}=(k_1,\dots,k_N)$ a $N$-tuple of
  non-negative integers and $\uple{x}=(x_1,\dots,x_N)$ a $N$-tuple of
  indeterminates, we define the $N\times N$ matrix
\begin{equation*}
\uple{x}(\uple{k})=\left[x_j^{k_i}\right]_{1\leq i,j\leq N},
\end{equation*} 
and note that
\begin{equation*}
\det\left(\uple{x}(\uple{k})\right)=\sum_{\sigma\in\mathcal{\sigma}_N}\epsilon(\sigma)x_{\sigma(1)}^{k_1}\dots x_{\sigma(N)}^{k_N}
\end{equation*}
vanishes if two components of $\uple{k}$ match. Recall that $e_m$ stands for the $m$'th elementary symmetric polynomial defined in \eqref{eq_alpha_sym}.
\end{notations}

We will prove:

\begin{proposition}\label{propo_schur}
  Let $N\geq 2$, $\uple{x}=(x_1,\dots,x_N)$,
  $\uple{y}=(y_1,\dots,y_N)$ and $T$ be indeterminates.
\par
\emph{(1)} One has
\begin{equation*}
\sum_{k\geq 0}S_{0,\dots,0,k}(\uple{x})S_{0,\dots,0,k}(\uple{y})T^k=\frac{P_N(\uple{x},\uple{y},T)}{\prod_{1\leq j,k\leq N}\left(1-x_jy_kT\right)}
\end{equation*}
for some polynomial
$P_N(\uple{x},\uple{y},T)\in\Z[\uple{x},\uple{y},T]$.
\par
\emph{(2)} If
\begin{equation}\label{eq_constraint_t}
0<t<\min_{1\leq j,k\leq N}\frac{1}{\abs{x_j}\abs{x_k}}
\end{equation}
then $P_N(\uple{x},\overline{\uple{x}},t)>0$.
\par
\emph{(3)} We have the formula
\begin{multline*}
  P_N(\uple{x},\uple{y},T)=1+\sum_{m=2}^{N(N-1)}(-T)^m\sum_{k=1}^{\min{(m,N-1)}}\sum_{\substack{1\leq m_1,\dots, m_k\leq\min{(m,N)} \\
      m_1+\dots+m_k=m}} \\
  \times\prod_{j=1}^ke_{m_j}(\uple{y})\sum_{2\leq J_1<\dots<J_k\leq
    N}\widetilde{S}_{(m_1,\dots,m_k)}^{(J_1,\dots,J_k)}(\uple{x})
\end{multline*}
with
\begin{equation*}
\widetilde{S}_{(m_1,\dots,m_k)}^{(J_1,\dots,J_k)}(\uple{x})=\frac{1}{V(\uple{x})}\det\left(\begin{array}{c}
 \\
 \\
J_1\to \\
 \\
J_k\to \\
 \\
 \end{array}\begin{pmatrix}
x_1^{N-1} & \dots & x_N^{N-1} \\
\vdots & \vdots & \vdots \\
x_1^{N-J_1+m_1} & \dots & x_N^{N-J_1+m_1} \\
\vdots & \vdots & \vdots \\
x_1^{N-J_k+m_k} & \dots & x_N^{N-J_k+m_k} \\
\vdots & \vdots & \vdots \\
1 & \vdots & 1
\end{pmatrix}\right)
\end{equation*}
for $1\leq k\leq N$, $2\leq J_1<\dots<J_k\leq N$ and $m_1,\dots,m_k\geq 1$. 
\end{proposition}

\begin{remark}\label{rem_N=2_N=3}
For example, for $N=2$,
\begin{equation*}
P_2(\uple{x},\uple{y},T)=1-e_2(\uple{x})e_2(\uple{y})T^2
\end{equation*}
whereas for $N=3$,
\begin{multline*}
P_3(\uple{x},\uple{y},T)=1-e_2(\uple{x})e_2(\uple{y})T^2 \\
+\left(e_3(\uple{x})\sum_{1\leq k_1\neq k_2\leq 3}y_{k_1}y_{k_2}^2+e_3(\uple{y})\sum_{1\leq j_1\neq j_2\leq 3}x_{j_1}x_{j_2}^2+4e_3(\uple{x})e_3(\uple{y})\right)T^3 \\
-e_1(\uple{x})e_3(\uple{x})e_1(\uple{y})e_3(\uple{y})T^4+e_3(\uple{x})^2e_3(\uple{y})^2T^6.
\end{multline*}
\end{remark}

\begin{remark}
  In general,
  $\widetilde{S}_{(m_1,\dots,m_k)}^{(J_1,\dots,J_k)}(\uple{x})$ can be
  related to a Schur polynomial as follows. Let us assume that $J_k<N$
  (a similar process also works if $J_k=N$) for simplicity. The finite
  sequence
\begin{equation*}
\left(-1,\dots,-J_1+m_1,\dots,-J_k+m_k,\dots,-(N-1)\right)\coloneqq(u_1,\dots,u_{N-1})
\end{equation*}
of length $N-1$ can be ordered increasingly as
\begin{equation*}
u_{\tau(1)}\leq\dots\leq u_{\tau(N-1)}
\end{equation*}
where $\tau$ is the appropriate permutation in $\sigma_{N-1}$. One can check that
\begin{equation*}
\widetilde{S}_{(m_1,\dots,m_k)}^{(J_1,\dots,J_k)}(\uple{x})=\epsilon(\tau)S_{\left(u_{\tau(1)}+N-1,u_{\tau(2)}-u_{\tau(1)}-1,\dots,u_{\tau(N-1)}-u_{\tau(N-2)}-1\right)}(\uple{x}).
\end{equation*}
\end{remark}

\begin{proof}[\proofname{} of proposition \ref{propo_schur}]%
Let us denote by $\Sigma$ the generating series. By \eqref{eq_def_schur},
\begin{align*}
\Sigma & =\frac{1}{V(\uple{x})V(\uple{y})}\sum_{k\geq 0}\det\left(\uple{x}(N-1+k,N-2,\dots,1,0)\right)\det\left(\uple{y}(N-1+k,N-2,\dots,1,0)\right)T^k \\
& =\frac{1}{V(\uple{x})V(\uple{y})}\sum_{(\sigma,\tau)\in\sigma_N^2}\epsilon(\sigma\tau)\left(x_{\sigma(1)}y_{\tau(1)}\right)^{N-1}\dots x_{\sigma(N-1)}y_{\tau(N-1)}\sum_{k\geq 0}\left(x_{\sigma(1)}y_{\tau(1)}T\right)^k \\
& =\frac{1}{V(\uple{x})V(\uple{y})}\sum_{(\sigma,\tau)\in\sigma_N^2}\epsilon(\sigma\tau)\frac{\left(x_{\sigma(1)}y_{\tau(1)}\right)^{N-1}\dots x_{\sigma(N-1)}y_{\tau(N-1)}}{1-x_{\sigma(1)}y_{\tau(1)}T} \\
& =\frac{1}{V(\uple{x})V(\uple{y})}\sum_{\sigma\in\sigma_N}\epsilon(\sigma)x_{\sigma(1)}^{N-1}\dots x_{\sigma(N-1)}F(\uple{y},x_{\sigma(1)}T)
\end{align*}
where
\begin{equation*}
F(\uple{y},Z)=\sum_{\tau\in\sigma_N}\epsilon(\tau)\frac{y_{\tau(1)}^{N-1}\dots y_{\tau(N-1)}}{1-y_{\tau(1)}Z}
\end{equation*}
where $Z$ is an indeterminate. One has
\begin{align*}
F(\uple{y},Z) & =\frac{1}{\prod_{1\leq k\leq N}\left(1-y_kZ\right)}\sum_{\tau\in\sigma_N}\epsilon(\tau)y_{\tau(1)}^{N-1}\dots y_{\tau(N-1)}\prod_{k=2}^N\left(1-y_{\tau(k)}Z\right) \\
& =\frac{V(\uple{y})}{\prod_{1\leq k\leq N}\left(1-y_kZ\right)}
\end{align*}
by Lemma \ref{lemma_vandermonde} below. Thus,
\begin{align*}
\Sigma & =\frac{1}{V(\uple{x})}\sum_{\sigma\in\sigma_N}\epsilon(\sigma)\frac{x_{\sigma(1)}^{N-1}\dots x_{\sigma(N-1)}}{\prod_{1\leq k\leq N}\left(1-y_kx_{\sigma(1)}T\right)} \\
& =\frac{1}{\prod_{1\leq j,k\leq N}\left(1-x_jy_kT\right)V(\uple{x})}\sum_{\sigma\in\sigma_N}\epsilon(\sigma)x_{\sigma(1)}^{N-1}\dots x_{\sigma(N-1)}\prod_{\substack{2\leq j\leq N \\
1\leq k\leq N}}\left(1-x_{\sigma(j)}y_kT\right) \\
& \coloneqq\frac{1}{\prod_{1\leq j,k\leq N}\left(1-x_jy_kT\right)V(\uple{x})}Q(\uple{x},\uple{y},T)
\end{align*}
\par
As a function of $\uple{x}$, $Q(\uple{x},\uple{y},T)$ is a
skew-symmetric polynomial. As such, $V(\uple{x})\mid
Q(\uple{x},\uple{y},T)$. For $\sigma\in\sigma_N$, Let say that the
quantities $x_{\sigma(j)}y_k$ ($2\leq j\leq N$, $1\leq k\leq N$) are
ordered lexicographically, namely
\begin{equation*}
\left(j_1,k_1\right)<\left(j_2,k_2\right) \text{ if } j_1<j_2 \text{ or } j_1=j_2 \text{ and } k_1<k_2.
\end{equation*}
\par
Once again,
\begin{equation*}
\prod_{\substack{2\leq j\leq N \\
1\leq k\leq N}}\left(1-x_{\sigma(j)}y_kT\right)=1+\sum_{m=1}^{N(N-1)}(-1)^me_m\left(\left\{x_{\sigma(j)}y_k, 2\leq j\leq N, 1\leq k\leq N\right\}\right)T^m
\end{equation*}
where
\begin{equation*}
e_m\left(\left\{x_{\sigma(j)}y_k, 2\leq j\leq N, 1\leq k\leq N\right\}\right)\coloneqq\sum_{\substack{2\leq j_1,\dots, j_m\leq N \\
1\leq k_1,\dots, k_m\leq N \\
\left(j_1,k_1\right)<\dots<\left(j_m,k_m\right)}}x_{\sigma(j_1)}y_{k_1}\dots x_{\sigma(j_m)}y_{k_m}.
\end{equation*}
\par
The condition $\left(j_1,k_1\right)<\dots<\left(j_m,k_m\right)$ is equivalent to saying that there exists $1\leq k\leq\min{(m,N-1)}$ and some positive integers $1\leq m_j\leq\min{(m,N)}$ ($1\leq j\leq k$) satisfying
\begin{equation*}
\sum_{j=1}^k m_j=m
\end{equation*}
and
\begin{eqnarray*}
j_1=\dots=j_{m_1}\coloneqq J_1 & \text{and} & 1\leq k_1<\dots<k_{m_1}\leq N, \\
j_{m_1+1}=\dots=j_{m_1+m_2}\coloneqq J_2 & \text{and} & 1\leq k_{m_1+1}<\dots<k_{m_1+m_2}\leq N, \\
\vdots & \vdots & \vdots \\
j_{m_1+\dots+m_{k-1}}=\dots=j_{m_1+\dots m_k}\coloneqq J_k & \text{and} & 1\leq k_{m_1+\dots+m_{k-1}}<\dots<k_{m_1+\dots+m_k}\leq N
\end{eqnarray*}
with $2\leq J_1<\dots<J_k\leq N$. Consequently,
\begin{multline*}
Q(\uple{x},\uple{y},T)=V(\uple{x})+\sum_{m=1}^{N(N-1)}(-T)^m\sum_{k=1}^{\min{(m,N-1)}}\sum_{\substack{1\leq m_1,\dots, m_k\leq\min{(m,N)} \\
m_1+\dots+m_k=m}} \\
\times\prod_{j=1}^ke_{m_j}(\uple{y})\sum_{2\leq J_1<\dots<J_k\leq N}\sum_{\sigma\in\sigma_N}\epsilon(\sigma)x_{\sigma(1)}^{N-1}\dots x_{\sigma(N-1)}x_{\sigma(J_1)}^{m_1}\dots x_{\sigma(J_k)}^{m_k}.
\end{multline*} 
\par
We now check that coefficient of $T$ in the previous equation is
$0$. This coefficient equals $-e_1(\uple{y})$ times the determinant of
the matrix
\begin{equation*}
\begin{array}{c}
 \\
 \\
J_1\to \\
 \\
 \end{array}\begin{pmatrix}
x_1^{N-1} & \dots & x_N^{N-1} \\
\vdots & \vdots & \vdots \\
x_1^{N-J_1+1} & \dots & x_N^{N-J_1+1} \\
\vdots & \vdots & \vdots \\
1 & \vdots & 1
\end{pmatrix}.
\end{equation*}
\par
We note that the $(J_1-1)$-th and the $J_1$-th rows of this matrix are
equal, and therefore its determinant vanishes.
\par
Finally, we prove the positivity of
$P_N(\uple{x},\overline{\uple{x}},t)$ if $t$ satisfies
\eqref{eq_constraint_t}. We have
\begin{multline*}
  0<1+\sum_{k\geq 1}\left\vert S_{0,\dots,0,k}(\uple{x})\right\vert^2t^k=P_N(\uple{x},\overline{\uple{x}},t) \\
  \times\left(\prod_{1\leq j\leq
      N}\left(1-\abs{x_j}^2t\right)\prod_{1\leq j<k\leq
      N}\left(1-2\Re{\left(x_j\overline{x_k}\right)}t+\abs{x_j}^2\abs{x_k}^2t^2\right)\right)^{-1}.
\end{multline*}
\par
The denominator in the previous equation is a positive real number
when the constraint \eqref{eq_constraint_t} is satisfied.
\end{proof}

We used the following elementary observation:

\begin{lemma}\label{lemma_vandermonde}
Let $Y$ be an indeterminate and $\uple{x}=(x_1,\dots,x_N)$. We have
\begin{equation*}
\left\vert\begin{array}{ccc}
x_1^{N-1} & \dots & x_N^{N-1} \\
x_1^{N-2}(1-x_1Y) & \dots & x_N^{N-2}(1-x_NY) \\
\vdots & \vdots & \vdots \\
x_1(1-x_1Y) & \dots & x_N(1-x_NY) \\
1-x_1Y & \dots & 1-x_NY
\end{array}\right\vert=V(\uple{x}).
\end{equation*}
\end{lemma}
\begin{proof}[\proofname{} of lemma \ref{lemma_vandermonde}]%
Of course, the determinant in the previous lemma equals
\begin{equation*}
\sum_{\sigma\in\mathcal{\sigma}_N}\epsilon(\sigma)x_{\sigma(1)}^{N-1}\dots x_{\sigma(N-1)}\prod_{\ell=2}^N\left(1-x_{\sigma(\ell)}Y\right).
\end{equation*}
Then,
\begin{equation*}
\prod_{\ell=2}^N\left(1-x_{\sigma(\ell)}Y\right)=1+\sum_{m=1}^{N-1}(-1)^me_m\left(x_{\sigma(2)},\dots,x_{\sigma(N)}\right)Y^m
\end{equation*}
where $e_m$ is the $m$'th elementary symmetric polynomial defined in \eqref{eq_alpha_sym}. Thus, the previous determinant equals
\begin{multline*}
\sum_{m=1}^{N-1}(-Y)^m\sum_{2\leq j_1<\dots<j_m\leq N}\det\left(\uple{x}(N-1,N-2+\epsilon_2(\uple{j}),\dots,1+\epsilon_{N-1}(\uple{j}),\epsilon_N(\uple{j}))\right) \\
+\det\left(\uple{x}(N-1,N-2,\dots,0)\right)
\end{multline*}
where
\begin{equation*}
\epsilon_\ell(\uple{j})\coloneqq\begin{cases}
0 & \text{if $\forall k, j_k\neq\ell$,} \\
1 & \text{otherwise}.
\end{cases}
\end{equation*}
for $2\leq\ell\leq N$. The last determinant in the previous equals is
nothing else than $V(\uple{x})$. All the other determinants vanish:
indeed, if
\begin{equation*}
  \ell_0(\uple{j})\coloneqq\min{\left\{2\leq\ell\leq N, \epsilon(\uple{j})=1\right\}}
\end{equation*}
then 
\begin{multline*}
  \det\left(\uple{x}(N-1,N-2+\epsilon_2(\uple{j}),\dots,1+\epsilon_{N-1}(\uple{j}),\epsilon_N(\uple{j}))\right) \\
  =\det\left(\uple{x}(N-1,N-2,\dots,N-(\ell_0(\uple{j})-1),N-\ell_0(\uple{j})+1,\dots,1+\epsilon_{N-1}(\uple{j}),\epsilon_N(\uple{j}))\right)=0,
\end{multline*}
since there are two identical rows in the matrix.
\end{proof}
\section{Generating series involving the multiple divisor functions}\label{sec_C}%

We begin by recalling a formula for the value of the multiple divisor
functions at prime powers.

\begin{lemma}\label{lemma_d_N_p_k}
For $N\geq 2$, $k$ a non-negative integer and a prime number $q$, 
\begin{equation}
  d_N(q^k)
  =\binom{N-1+k}{k}
=\frac{N(N+1)\cdots (N+k-1)}{k!}.
\end{equation}
\end{lemma}


\begin{proof}[\proofname{} of lemma \ref{lemma_d_N_p_k}]%
  This amounts to the formula for the number of monomials of degree
  $k$ in $N$ variables, which is well-known.
\end{proof}


We now prove a formula for the generating function of the square of
the divisor function:

\begin{proposition}\label{propo_d_N_series}
For $q$ a prime number, one has
\begin{equation*}
\sum_{k\geq 0}d_N(q^k)^2T^k=\frac{P_N(T)}{(1-T)^{2N-1}}
\end{equation*}
where 
$$
P_N(T)=\sum_{k=0}^{N-1}\binom{N-1}{k}^2T^k\in\Z[T].
$$
\par
In particular, we have $P_N(t)>0$ for $t>0$.  Moreover, the constant
term of $P_N(T)$ is equal to $1$ and the coefficient of $T$ is
$(N-1)^2$.
\end{proposition}


\begin{proof}[\proofname{} of proposition \ref{propo_d_N_series}]%
Let
$$
{}_2F_1(u,v;1;z)=
\sum_{k\geq 0}\frac{u(u+1)\cdots
  (u+k-1)v(v+1)\cdots (v+k-1)}
{(k!)^2}z^k
$$
denote (a special case of) the classical Gauss hypergeometric
function. By the previous lemma, we have
$$
\sum_{k\geq 0}d_N(q^k)^2T^k={}_2F_1(N,N;1;T).
$$
\par 
Since
$$
\sum_{k=0}^{N-1}\binom{N-1}{k}^2T^k=
\sum_{k\geq 0}\binom{N-1}{k}^2T^k={}_2F_1(-(N-1),-(N-1);1;T),
$$
the formula we claim is
$$
{}_2F_1(-(N-1),-(N-1);1;T)=(1-T)^{2N-1}{}_2F_1(N,N;1;T),
$$
which is a special case of the formula known as Euler's transformation
for the hypergeometric function (see, e.g.,~\cite[9.131.1 (3)]{GR00}).
\end{proof}

\bibliographystyle{plain}
\bibliography{biblio}
\end{document}